\documentclass[english]{article}
\usepackage[utf8]{inputenc}
\usepackage[margin=2.5cm]{geometry}
\usepackage{babel}

\usepackage{amssymb}
\usepackage{amsmath}
\usepackage{amsthm}
\usepackage{xspace}
\usepackage{latexsym}
\usepackage{upgreek}
\usepackage{mathrsfs}
\usepackage{relsize}
\usepackage{bbm}
\usepackage{bm}

\usepackage[ruled]{algorithm}
\usepackage{algorithmic}

\usepackage{graphicx}
\usepackage{tikz}
\usetikzlibrary{positioning, shapes, calc, arrows.meta}
\usepackage{caption}
\usepackage{multirow}

\usepackage{enumerate}
\usepackage{xargs}
\usepackage[active]{srcltx}
\usepackage[inline]{enumitem} 

\usepackage{url}
\usepackage[colorlinks,linkcolor=blue,citecolor=red,urlcolor=blue]{hyperref}
\usepackage{cleveref}
\usepackage{autonum}

\makeatletter
\newtheorem{theorem}{Theorem}
\crefname{theorem}{theorem}{Theorems}
\Crefname{Theorem}{Theorem}{Theorems}

\newtheorem{lemma}[theorem]{Lemma}
\crefname{lemma}{lemma}{lemmas}
\Crefname{Lemma}{Lemma}{Lemmas}

\newtheorem{proposition}[theorem]{Proposition}
\crefname{proposition}{proposition}{propositions}
\Crefname{Proposition}{Proposition}{Propositions}

\newtheorem{definition}[theorem]{Definition}
\crefname{definition}{definition}{definitions}
\Crefname{Definition}{Definition}{Definitions}

\newtheorem{assumption}{\textbf{A}\hspace{-3pt}}
\crefformat{assumption}{{\textbf{A}}#2#1#3}

\theoremstyle{remark}
\newtheorem{example}[theorem]{Example}
\Crefname{Example}{Example}{Example}

\usepackage{authblk}

\usepackage[textwidth=1.8cm, textsize=scriptsize]{todonotes}

\DeclareMathOperator{\mse}{MSE}
\DeclareMathOperator{\ise}{ISE}

\def\rmH{\mathrm{H}}
\def\bfnu{\bm{\nu}}
\def\bflambda{\bm{\lambda}}
\def\ctt{\mathtt{c}}
\newcommand{\transference}{\mathbf{T}}
\newcommandx{\norm}[2][1=]{\ifthenelse{\equal{#1}{}}{\left\Vert #2 \right\Vert}{\left\Vert #2 \right\Vert^{#1}}}
\newcommandx{\normLigne}[2][1=]{\ifthenelse{\equal{#1}{}}{\Vert #2 \Vert}{\Vert #2\Vert^{#1}}}
\def\bfX{\mathbf{X}}
\def\lambdabf{\bm{\lambda}}
\def\bfM{\mathbf{M}}
\def\bfB{\mathbf{B}}
\def\msa{\mathsf{A}}
\def\msf{\mathsf{F}}
\def\msx{\mathsf{X}}
\newcommand{\mcb}[1]{\mathcal{B}(#1)}
\def\rset{\mathbb{R}}
\def\nset{\mathbb{N}}
\def\nsets{\mathbb{N}^{\star}}
\def\rmd{\mathrm{d}}
\def\rmc{\mathrm{C}}
\def\rmC{\mathrm{C}}
\def\rmA{\mathrm{A}}
\def\rmW{\mathrm{W}}
\def\rmL{\mathrm{L}}
\def\rms{\mathrm{s}}
\def\bmk{\tilde{b}}
\def\bmketa{\bmk_{\eta}}
\def\Ltt{\mathtt{L}}
\newcommand{\argmax}{\operatorname*{arg\,max}}
\newcommand{\argmin}{\operatorname*{arg\,min}}
\newcommand{\1}{\mathbbm{1}}
\newcommand{\floor}[1]{\left\lfloor #1 \right\rfloor}
\newcommand{\PE}{\mathbb{E}}
\newcommand{\abs}[1]{\left\vert #1 \right\vert}
\newcommand{\absLigne}[1]{\vert #1 \vert}
\newcommand{\parenthese}[1]{\left(#1 \right)}
\newcommand{\parentheseLigne}[1]{(#1 )}
\newcommand{\defEns}[1]{\left\lbrace #1 \right\rbrace }
\newcommand{\defEnsLigne}[1]{\lbrace #1 \rbrace }
\newcommand{\expe}[1]{\PE \left[ #1 \right]}
\newcommand{\expesq}[1]{\PE^{1/2} \left[ #1 \right]}
\newcommand{\expesqLigne}[1]{\PE^{1/2} [ #1 ]}
\newcommand{\expeLigne}[1]{\PE [ #1 ]}
\def\ie{\textit{i.e.}}
\def\eqsp{}
\newcommand{\coint}[1]{\left[#1\right)}
\newcommand{\ocint}[1]{\left(#1\right]}
\newcommand{\ooint}[1]{\left(#1\right)}
\newcommand{\ccint}[1]{\left[#1\right]}
\newcommand{\ccintLigne}[1]{[#1]}
\newcommand{\cball}[2]{\bar{\operatorname{B}}(#1,#2)}
\def\eg{e.g.}
\def\Id{\operatorname{Id}}
\newcommand{\ensembleLigne}[2]{\{#1\,:\eqsp #2\}}
\def\rmD{\mathrm{D}}
\newcommand{\complementary}{\mathrm{c}}
\def\vareps{\varepsilon}
\def\Fun{\mathscr{F}}
\def\Gun{\mathscr{G}}
\def\Hun{\mathscr{H}}
\def\Lun{\mathscr{L}}
\def\Mun{\mathscr{M}}
\def\Hent{\mathrm{H}}
\def\Pens{\mathscr{P}}
\newcommandx{\KL}[2]{\textup{KL}\left( #1 | #2 \right)}
\newcommandx{\KLLigne}[2]{\text{KL}( #1 | #2 )}
\newcommandx{\wassersteinD}[1][1=\distance]{\mathbf{W}_{#1}}
\newcommand{\rref}[1]{\tup{\Cref{#1}}}
\def\Kcoupling{\mathrm{K}}
\def\kker{\mathrm{k}}
\def\Kker{\Kcoupling}
\def\mtt{\mathtt{m}}
\def\dtt{\mathtt{d}}
\def\Mtt{\mathtt{M}}
\def\Ktt{\mathtt{K}}
\newcommand{\tup}[1]{\textup{#1}}
\title{Solving A Class of Fredholm Integral Equations\\ of the First Kind via Wasserstein Gradient Flows}
\author[1]{Francesca R. Crucinio\thanks{Corresponding author: francesca\_romana.crucinio@kcl.ac.uk}}
\author[2]{Valentin De Bortoli}
\author[3]{Arnaud Doucet}
\author[4, 5]{Adam M. Johansen}
\affil[1]{Department of Mathematics, King's College London}

\affil[2]{Computer Science Department,
ENS, CNRS, PSL University}
\affil[3]{University of Oxford}

\affil[4]{Department of Statistics, University of Warwick}
\affil[5]{The Alan Turing Institute}
\date{ }
\begin{document}
\maketitle
\graphicspath{{Images/}}

\abstract{Solving Fredholm equations of the first kind is crucial in many areas of the applied sciences. In this work we consider integral equations featuring kernels which may be expressed as scalar multiples of conservative (i.e. Markov) kernels and we adopt a variational point of view by considering a minimization problem in the space of probability measures with an entropic regularization. Contrary to classical approaches which discretize the domain of the solutions, we introduce an algorithm to asymptotically sample from the unique solution of the regularized minimization problem. As a result our estimators do not depend on any underlying grid and have better scalability properties than most existing methods.  Our algorithm is based on a particle approximation of the solution of a McKean--Vlasov stochastic differential equation associated with the Wasserstein gradient flow of our variational formulation. We prove the convergence towards a minimizer and provide practical guidelines for its numerical implementation. Finally, our method is compared with other approaches on several examples including density deconvolution and epidemiology.}

\section{Introduction}
Fredholm integral equations of the first kind are ubiquitous in applied sciences
and engineering, they model density deconvolution
\cite{delaigle2008alternative,ma2011indirect,pensky2017minimax,yang2020density}, image reconstruction
\cite{aster2018parameter,clason2019regularization,jin2019expectation, snyder1992deblurring}, inverse boundary problems for partial differential equations \cite{tanana2016approximate, colton2012inverse} and find applications in epidemiology \cite{goldstein2009reconstructing, gostic2020practical, marschner2020back} and statistics
\cite{hall2005nonparametric,miao2018identifying}.
Fredholm integral equations of the first kind are
typically defined by
\begin{equation}
g(y) = \int_{\rset^d} f(x) \kker(x,y) \rmd x
\end{equation}
where the kernel $\kker:\rset^d \times \rset^p \to \rset$ is known, and $g:\rset^p\to\rset$ is known but indirectly observed, and the objective is to identify the unknown function $f:\rset^d \to\rset$. If we interpret this as a Lebesgue integral with $\rmd x$ denoting the Lebesgue measure then we may identify the functions with measures $\mu(\rmd y) := g(y) \rmd y, \pi(\rmd x) := f(x) \rmd x$ and $\kker$ with an integral operator $\Kker(x,\rmd y) := \kker(x,y) \rmd y$ which leads to the representation that we will consider throughout this paper:
\begin{equation}
  \label{eq:fe}
  \textstyle{
\mu(\msa)= \pi \Kker(\msa) := \int_{\rset^d} \Kker(x, \msa)\rmd \pi(x), }
\end{equation}
for any Borel set $\msa \in \mcb{\rset^p}$, with $\pi$ an unknown measure on
$\mcb{\rset^d}$, $\mu$ an observed measure on $\mcb{\rset^p}$ and
$\Kker: \ \rset^d \times \mcb{\rset^p} \to \rset$ an integral operator. These equations
model the task of reconstructing the signal $\pi$ from a distorted observed
version $\mu$.  Since many applications of Fredholm integral equations are
concerned with the reconstruction of signals that are \textit{a priori} known to
be non-negative \cite{clason2019regularization}, we consider the case in which
both $\pi \in \Pens(\rset^d)$ and $\mu \in \Pens(\rset^p)$ are probability measures on $\rset^d$ and $\rset^p$, respectively, and $\Kker$ is a Markov kernel.  A great many Fredholm integral
equations, particularly those for which the kernel is conservative up to a scalar multiplier in the sense that $\int_{\rset^p} \Kker(x,\rmd y) = \int_{\rset^p} \Kker(x',\rmd y)$ for every $x,x' \in \rset^d$, can be recast in this framework with appropriate translation and
normalization \cite[Section 6]{chae2018algorithm}.

In a probabilistic and variational approach, we seek a solution of~\eqref{eq:fe}
by minimizing the Kullback--Leibler divergence between $\mu$ and $\pi\Kker$
\cite{multhei1987iterative}:
\begin{align}
  \label{eq:kl}
  \textstyle{
\pi^\star = \argmin \ensembleLigne{\KL{\mu}{\pi \Kker} }{\pi \in \Pens(\rset^d)},}
\end{align}
where
$\KL{\mu}{\pi \Kker}=\int_{\rset^p} \log ((\rmd \mu /\rmd \pi\Kker)(y)) \rmd
\mu(y)$ if $\mu$ admits a density w.r.t. $\pi \Kker$ and $+\infty$ otherwise. The probability measure $\pi^\star$ corresponds to the maximum likelihood estimator (MLE) for $\pi$. Using this formulation, it is possible to define approximate
solutions of Fredholm integral equations of the first kind as solutions of a minimization
problem on the space of measures. However, in most interesting cases, the integral equation~\eqref{eq:fe}
is ill-posed, \ \ie\ there might exists more than one solution $\pi$
to~\eqref{eq:fe}, and, even when the solution is unique, it is unstable with
respect to changes in $\mu$. This lack of stability is the primary
concern when attempting to solve Fredholm integral equations of the first kind: in practical applications we often only have access to (noisy) observations from $\mu$ and not to its analytic form, and the instability of~\eqref{eq:fe} means that small errors in $\mu$ do not necessarily correspond to small errors in the recovered solution $\pi$  \cite[Chapter 15]{kress2014linear}. This issue
reflects onto the minimization problem~\eqref{eq:kl}, which often does not admit a unique
minimizer \cite{laird1978nonparametric}.

In order to circumvent this issue, a regularization term is often added to the
variational formulation~\eqref{eq:kl} \cite{byrne2015algorithms, green1990use}. In a similar fashion to the Tikhonov
regularization \cite{groetsch1984theory}, we consider a regularized version of~\eqref{eq:kl}
in which a cross--entropic penalty with respect to a reference measure
$\pi_0 \in \Pens(\rset^d)$ is introduced as in \cite{eggermont1993maximum, resmerita2007joint, iusem1994new}
\begin{equation}
  \pi^\star = \argmin \ensembleLigne{\Fun_{\alpha}(\pi)}{\pi \in \Pens(\rset^d)}, \qquad \Fun_\alpha(\pi) = \KL{\mu}{\pi \Kker} + \alpha \KL{\pi}{\pi_0},\label{eq:minimisation}
\end{equation}
for a given regularization
parameter $\alpha>0$. In practice, \eqref{eq:minimisation} becomes easier to solve for large values of $\alpha$ but the agreement between $\pi^\star \Kker$ and $\mu$ is weak
and in the limiting case $\alpha \to +\infty$, $\pi^\star = \pi_0$. For small
values of $\alpha$, $\pi^\star \Kker \approx \mu$ but the minimization problem
is harder to solve and in the limit $\alpha \to 0$, it collapses onto~\eqref{eq:kl}, becoming ill-posed.
The parameter $\alpha$ should therefore be chosen to achieve a trade-off between minimizing $\KL{\mu}{\pi\Kker}$ (i.e. $\alpha$ should be small enough), ensuring that~\eqref{eq:minimisation} can be solved and that the resulting solution is sufficiently regular (i.e. $\alpha$ should not be zero); this is usually achieved by selecting $\alpha$ through cross--validation, see
\cite{wahba1977practical, amato1991maximum}.

The use of cross--entropy regularization is not new in the literature on Fredholm
integral equations \cite{eggermont1993maximum, iusem1994new, resmerita2007joint}. In particular, \cite{eggermont1993maximum} considers a cross--entropy regularization in a least-square setting, \cite{iusem1994new}
studies~\eqref{eq:minimisation} in the discrete setting, i.e. when $\pi$ and $\mu$
are vectors in the simplex of $\rset^d, \rset^p$ respectively, and \cite{resmerita2007joint} studies~\eqref{eq:minimisation} when $\pi$ and $\mu$ are integrable densities.
In the discrete setting, \cite{iusem1994new} establishes uniqueness of the minimizers of a discrete version of~\eqref{eq:minimisation} and studies the behaviour of the minimizers as $\alpha\to0$; 
in the continuous setting, \cite{resmerita2007joint} shows
existence and uniqueness of a minimizer of~\eqref{eq:minimisation} under the assumption that~\eqref{eq:fe} admits a
maximum entropy solution.

In related works, functionals involving the Kullback--Leibler divergence
$\KL{\mu}{\pi \Kker}$ are minimized iteratively by discretizing the domain of
$\pi$, $\mu$ \cite{green1990use, chae2018algorithm, burger2019entropic,
  yang2020density, snyder1992deblurring}, while maximum entropy estimators are
usually obtained approximating $\pi$ through a set of basis functions
\cite{kopec1993application, islam2020approximating, jin2016solving,
  mead1986approximate}.  However, when $\pi, \mu$ are probability measures it is
natural to consider Monte Carlo strategies, and approximate $\pi$ using particles \cite{crucinio2020particle}. This is particularly convenient when the only information on $\mu$
is given by a sample drawn from it, as it is the case in most applications
\cite{delaigle2008alternative,ma2011indirect,pensky2017minimax,yang2020density,
  goldstein2009reconstructing, gostic2020practical,
  marschner2020back,hall2005nonparametric,miao2018identifying}. 
Recently, \cite{crucinio2020particle} introduced a sequential Monte Carlo method (SMC-EMS) to approximately solve~\eqref{eq:fe}, this algorithm provides an adaptive stochastic discretization of the expectation maximization smoothing algorithm first studied in \cite{silverman1990smoothed}, which aims at obtaining solutions of~\eqref{eq:fe} which achieve a good trade-off between solving~\eqref{eq:kl} and having enough smoothness.

The functional~\eqref{eq:minimisation} can be seen as a probabilistic counterpart of the classic Tikhonov regularization setting (e.g~\cite{groetsch1984theory}); \cite{garbuno2020interacting} propose an interacting particle system based on a Wasserstein gradient flow construction to approximately solve the Tikhonov-regularized problem for a wide class of (nonlinear) inverse problems. While the particle system introduced in \cite{garbuno2020interacting} is similar in spirit to the approach taken here, the former requires an a priori discretization of the support of $\pi$, which makes it impractical when we deal with Fredholm integral equations, since the object of interest (\ie\ $\pi$) is infinite dimensional. By focussing on the case of linear integral equations, we can avoid the discretization step required by the particle system of \cite{garbuno2020interacting}, and work directly with the measures $\pi, \mu$.

In this work, we
adopt a probabilistic approach and present a novel particle algorithm to solve Fredholm integral equations of the first kind. More precisely, we derive a particle approximation to the Stochastic
Differential Equation (SDE) associated with the Wasserstein gradient flow
\cite{ambrosio2008gradient,santambrogio2017euclidean}
minimizing surrogates of~\eqref{eq:minimisation}. 
Wasserstein gradient flows have been widely used to solve a variety of optimization problems in the space of probability measures \cite{arbel2019maximum, jordan1998variational, liu2017stein, durmus2019analysis}, including finding solutions to finite dimensional inverse problems \cite{garbuno2020interacting}, but, to the best of our knowledge, not for infinite dimensional inverse problems like the integral equation~\eqref{eq:fe}.
The SDE associated with this
gradient flow is non-linear in the sense that it depends not only upon the
position of the solution but also on its marginal distribution. Such SDEs are
called McKean--Vlasov SDEs (MKVSDE), and have been actively studied during the past decades
\cite{tanaka1984limit,sznitman1991topics,meleard1996asymptotic,meleard1987propagation,oelschlager1984martingale,mckean1966class,carmona2013probabilistic,gottlieb2000markov}. In
this work, we derive a particular MKVSDE whose solution converges towards a minimizer of~\eqref{eq:minimisation}.
However, this MKVSDE cannot be solved analytically, and we introduce a numerical approximation using an Euler--Maruyama discretization and a particle
system.

We perform an extensive simulation study in which we compare our method with SMC-EMS and problem-specific algorithms and show that it achieves comparable performances. The results obtained with our method and SMC-EMS are particularly similar, since both methods aim at minimizing a regularized Kullback--Leibler divergence. However, while in this work we explicitly have a minimization problem for the regularized functional~\eqref{eq:minimisation}, the SMC-EMS algorithm does not minimize a particular
functional but corresponds to a smoothed version of the EM algorithm minimizing
the Kullback--Leibler divergence~\eqref{eq:kl}.
When $\alpha$ and the smoothness parameter controlling the amount of regularization in SMC-EMS are both small, the two algorithms will therefore give very similar reconstructions. 
Nevertheless, the method proposed in this work comes with a number of additional theoretical guarantees and practical improvements. From the theoretical point of view, we show that~\eqref{eq:minimisation} admits a unique minimizer for any $\alpha$ and that this minimizer can be seen as the probabilistic counterpart of the classical solution to the Tikhonov regularized least-squares problem. On the other hand, uniqueness of the fixed point of the EMS recursion approximated by SMC-EMS has not been established.
From the practical point of view, minimizing~\eqref{eq:minimisation} allows us to introduce additional information about the solution through $\pi_0$, this proves particularly beneficial when the kernel $\Kker$ leads to model misspecification (see Section~\ref{sec:epidem}).
When compared on higher dimensional problems, the proposed method outperforms both standard discretization-based methods and SMC-EMS (Section~\ref{sec:hd}), opening the door to solving Fredholm integral equations for high-dimensional problems.

In addition, the value of the regularization parameter $\alpha$ can be chosen through common
approaches for selecting the regularization parameter in Tikhonov regularization settings (e.g. 
cross--validation \cite{wahba1977practical, amato1991maximum}), while in the case of SMC-EMS the link between the particle system and the functional which is minimized is less clear, and cannot be exploited as straightforwardly.

To summarize, our main
contributions are as follows:
\begin{enumerate}[label=(\alph*)]
\item First, we establish that the functional in~\eqref{eq:minimisation} is continuous and
admits a unique minimizer under mild conditions on $\mu, \Kker$. In addition, we show that a modification of~\eqref{eq:minimisation} is stable with respect to perturbations of the observation $\mu$ and study the limit behaviour of $\pi^\star$ as $\alpha \to 0$.
\item Second, we derive the Wasserstein gradient flow used to minimize a
  surrogate of~\eqref{eq:minimisation} and show that the associated MKVSDE
  converges towards a unique minimizer using recent results from optimal
  transport \cite{hu2019mean}.
\item Third, we present and study a discrete-time particle approximation of
  this MKVSDE. In particular, we establish geometric ergodicity of the particle system under regularity conditions.
\item Finally, we illustrate the efficiency of our method to approximately solve
  Fredholm integral equations of the first kind on a range of different
  numerical examples and show that it achieves superior or comparable performances to those of problem-specific algorithms and SMC-EMS, outperforming the latter for high dimensional problems. In addition, we show on a task from epidemiology that, when the kernel $\Kker$ does not correctly model the problem at hand, the presence of a reference measure $\pi_0$ allows us to outperform even problem-specific estimators.
\end{enumerate}

The rest of the manuscript is organized as follows.  In \Cref{sec:vari-point-view} we
establish that the functional in~\eqref{eq:minimisation} is well-founded and study its key properties. 
The link between gradient flows and MKVSDE in the context of Fredholm integral
equations of the first kind is studied in \Cref{sec:part-appr-mcke}.  In
\Cref{sec:implementation} we describe a numerical scheme which approximates the
continuous time dynamics of the particle system and provide error bounds.
Finally, in \Cref{sec:ex}, we test the algorithm on a number of applications of
Fredholm integral equations from statistics, epidemiology and image processing, and explore the scaling properties of our method through a comparison with the SMC-EMS algorithm recently proposed in \cite{crucinio2020particle}.


\section*{Notation}
\label{sec:notation}
We endow $\rset^d$ with the Borel $\sigma$-field $\mcb{\rset^d}$ with respect to
the Euclidean norm $\norm{\cdot}$. For a matrix $A$ we denote $\norm{A}=(\sum_{ij} a_{ij}^2)^{1/2}$ the Hilbert-Schmidt norm.
We denote by $\rmC(\rset^d)$ the set of continuous functions defined
over $\rset^d$ and by $\rmC^n(\rset^d)$ the set of $n$-times differentiable functions
defined over $\rset^d$ for any $n \in \nsets$, where $\nsets$ denotes the non-zero natural numbers.  For all $f \in \rmC^1(\rset^d)$,
we denote by $\nabla f$ its gradient. Furthermore, if $f \in \rmC^2(\rset^d)$ we
denote by $\nabla^2f$ its Hessian and by $\Delta f$ its Laplacian.  
For all $f \in \rmC^1(\rset^{d_1} \times \dots \times \rset^{d_m})$ with $m\in\nsets$, we denote by $\nabla_i f$ its gradient w.r.t. component $i$. 
We say that a
function $f: \ \msx \to \rset$ (where $\msx$ is a metric space) is coercive if
for any $t \in \rset$, $f^{-1}(\ocint{-\infty, t})$ is relatively compact. This
definition can be extended to the case where $\msx$ is only a topological space,
see \cite[Definition 1.12]{dalmaso1993introduction}. A family of functions
$\ensembleLigne{f_\alpha}{\alpha \in \msa}$ such that for any $\alpha \in \msa$,
$f_\alpha : \ \msx \to \rset$, is said to be equicoercive if there exists
$g: \ \msx \to \rset$ such that for any $\alpha \in \msa$, $f_\alpha \geq g$ and
$g$ is coercive.
For any $p \in \nset$ we denote by
$\Pens_p(\rset^d) = \ensembleLigne{\pi \in \Pens(\rset^d)}{\int_{\rset^d}
  \norm{x}_p^p \rmd \pi(x) < +\infty}$ the set of probability measures over
$\mcb{\rset^d}$ with finite $p$-th moment. For ease of notation, we define
$\Pens(\rset^d) = \Pens_0(\rset^d)$ the set of probability measures over
$\mcb{\rset^d}$ and endow this space with the topology of weak convergence.  For any $\mu,\nu\in\Pens_p(\rset^d)$ we define the
$p$-Wasserstein distance $\wassersteinD[p](\mu, \nu)$ between $\mu$ and $\nu$ by
\begin{equation}
  \label{eq:def_distance_wasser}
  \textstyle{
    \wassersteinD[p](\mu, \nu) =\parenthese{ \inf_{\gamma \in \transference(\mu,\nu)} \int_{\rset^d\times \rset^d}  \norm{x -y}_p^p \rmd \gamma (x,y)}^{1/p}
    }
\end{equation}
where
$\transference(\mu, \nu)=\ensembleLigne{\gamma\in\Pens(\rset^d\times
  \rset^d)}{\gamma(\msa \times \rset^d) = \mu(\msa),\ \gamma(\rset^d \times
  \msa) = \nu(\msa)\ \forall \msa \in \mcb{\rset^d}}$ denotes the set of all transport
plans between $\mu$ and $\nu$. In the following, we metrize $\Pens_p(\rset^d)$ with $\wassersteinD[p]$. Assume that $\mu \ll \nu$ and denote by
$\frac{\rmd \mu}{\rmd \nu}$ its Radon-Nikodym derivative. The
Kullback--Leibler divergence, $\KL{\mu}{\nu}$, between $\mu$ and $\nu$ is defined by
\begin{equation}
  \textstyle{
    \KL{\mu}{\nu} = \int_{\msx} \log\parenthese{\frac{\rmd \mu}{\rmd \nu}(x)} \rmd \mu(x).
    }
\end{equation}
We say that $\Kker : \ \rset^d \times \mcb{\rset^p} \to \coint{0,+\infty}$ is a Markov kernel if for any $x \in \rset^d$, $\Kker(x, \cdot)$ is a probability measure over $\mcb{\rset^p}$ and for any $\msa \in \mcb{\rset^p}$, $\Kker(\cdot, \msa)$ is a Borel-measurable function over $\rset^d$. 

\section{A minimization problem for Fredholm integral equations}

\label{sec:vari-point-view}

In this section, we study the properties of the regularized functional
$\Fun_{\alpha}: \ \Pens(\rset^d) \to \rset \cup {+\infty}$ given for any 
$\alpha \geq 0$ and $\pi \in \Pens(\rset^d)$ by
\begin{equation}
  \Fun_{\alpha}(\pi) = \KL{\mu}{\pi \Kker} + \alpha \KL{\pi}{\pi_0}, 
\end{equation}
where we recall that $\mu \in \Pens(\rset^p)$, $\pi_0 \in \Pens(\rset^d)$ and
the Markov kernel $\Kker: \ \rset^d \times \mcb{\rset^p} \to \ccint{0,1}$ are
given.  For ease of notation we denote $\Fun = \Fun_0$.

In \Cref{sec:regul-stab-prop} we introduce a surrogate functional which enjoys
better stability properties than $\Fun_\alpha$. We prove that under mild
assumption on $\mu$ and $\Kker$, a minimizer of the surrogate exists, is
unique and continuous w.r.t the parameters of the functional. In
\Cref{sec:variants}, we investigate the links between our approach and other
regularization methods for Fredholm integration equations of the first kind.

\subsection{A surrogate functional: stability and regularization}
\label{sec:regul-stab-prop}

First, we introduce a surrogate functional $\Gun_\alpha$ whose minimizers are
the same as $\Fun_\alpha$ if $\mu$ is absolutely continuous w.r.t the Lebesgue
measure but which is also easier to study when $\mu$ is approximated by
an empirical distribution. Assume that there exists
$\kker: \ \rset^d \times \rset^p \to \coint{0, +\infty}$ such that for any
$x \in \rset^d$ and $\msa \in \mcb{\rset^p}$,
$\Kker(x, \msa) = \int_{\msa} \kker(x,y) \rmd y$. Then, using the Fubini--Tonelli
theorem we have that for any $\msa \in \mcb{\rset^p}$ and
$\pi \in \Pens(\rset^d)$,
\begin{equation}
  \textstyle{
    \pi \Kker(\msa) = \int_{\rset^d} \Kker(x, \msa) \rmd \pi(x) = \int_{\msa} \parenthese{\int_{\rset^d} \kker(x,y) \rmd \pi(x)} \rmd y.
    }
\end{equation}
Therefore $\pi \Kker$ admits a density w.r.t. the Lebesgue measure given by
$y \mapsto \pi[\kker(\cdot, y)]$.  In addition, assume that $\mu$ admits a
density w.r.t. the Lebesgue measure $\rmd \mu(y) = \mu(y)\rmd y$ such that
$\int_{\rset^p}\absLigne{\log(\mu(y))} \rmd\mu(y) <
+\infty$. Then for any $\pi \in \Pens(\rset^d)$
\begin{equation}
  \textstyle{
    \KL{\mu}{\pi \Kker} = -\Hent(\mu) -\int_{\rset^p} \log(\pi[\kker(\cdot, y)]) \rmd \mu(y),
    }
  \end{equation}
  where
  $\Hent(\mu) :=-\int_{\rset^p}\log(\mu(y)) \rmd \mu(y)$ is
  the entropy of $\mu$. 
  
  Since
  $\int_{\rset^p} \abs{\log(\pi[\kker(\cdot, y)])}\rmd \mu(y) < +\infty$ if
  $\KL{\mu}{\pi \Kker}<+\infty$, minimizing
  $\Fun_{\alpha}$ on the set
  $\ensembleLigne{\pi \in \Pens(\rset^d)}{\KL{\mu}{\pi \Kker} < +\infty}$ for any $\alpha \geq 0$ is
  equivalent to minimizing $\Gun_{\alpha}$ on the same set, where for any $\pi \in \Pens(\rset^d)$
\begin{equation}
  \label{eq:G}
  \textstyle{
  \Gun_{\alpha}(\pi) = -\int_{\rset^p} \log(\pi[\kker(\cdot, y)]) \rmd \mu(y) + \alpha \KL{\pi}{\pi_0},
}
\end{equation}
with $\Gun = \Gun_0$ for ease of notation.

 The functional~\eqref{eq:G} is not
defined if $\int_{\rset^p} \log(\pi[\kker(\cdot, y)]) \rmd \mu(y) = +\infty$ or
$\KL{\pi}{\pi_0} = +\infty$. We consider the following assumption on the kernel
$\Kker$ so that the functional is well-defined for any $\pi \in \Pens(\rset^d)$
and regular enough, see \Cref{prop:g_prop}.
\begin{assumption}
  \label{assum:general_kker}
  There exists $\kker \in \rmc^{\infty}(\rset^d \times \rset^p, \coint{0, +\infty})$
  such that for any $x \in \rset^d$ and $\msa \in \mcb{\rset^p}$,
  $\Kker(x, \msa) = \int_{\msa} \kker(x,y) \rmd y$ and
  $-\int_{\rset^p} \log(\kker(0,y)) \rmd \mu(y) < +\infty$. In addition, there
  exists $\Mtt \geq 0$ such that for any $(x ,y)\in\rset^d\times\rset^p$ we have
  $\kker(x,y) + \normLigne{\nabla \kker(x,y)} + \normLigne{\nabla^2 \kker(x,y)} \leq \Mtt$. In addition, $\normLigne{\partial^3_{1, ijl} \kker(x,y)} \leq \Mtt$, where $\partial_i$ denotes the derivative w.r.t. the $i$-th variable.
\end{assumption}

In this section, we could restrict ourselves to bounded second derivatives. The condition
$\normLigne{\partial^3_{1, ijl} \kker(x,y)} \leq \Mtt$ is used only in
\Cref{prop:euler}, in which we establish a convergence result for the time discretization of the MKVSDE. Under \tup{\Cref{assum:general_kker}}, the density $\kker$ of the Markov kernel
$\Kker$ is Lipschitz continuous for every $y\in\rset^p$.

To address stability issues we also
consider the following regularized functional: for any $\alpha, \eta \geq 0$ and
$\pi \in \Pens(\rset^d)$
\begin{equation}
\label{eq:G_eta}
\textstyle{
  \Gun_{\alpha}^{\eta}(\pi) = -\int_{\rset^p} \log(\pi[\kker(\cdot, y)] + \eta) \rmd \mu(y) + \alpha \KL{\pi}{\pi_0},
}
\end{equation}
where $\eta \geq 0$ is a hyperparameter. Using $\Gun_\alpha^\eta$ with
$\eta > 0$ will be crucial in order to establish the Lipschitz continuity of the
drift of the MKVSDE in \Cref{sec:part-appr-mcke}, and hence the stability of the
proposed procedure. For clarity, we write $\Gun_{\alpha}=\Gun^{0}_{\alpha}$, $\Gun^{\eta} = \Gun_{0}^{\eta}$ and $\Gun = \Gun_{0}^{0}$.
In the following proposition, we derive regularity properties for the family of
functionals $\ensembleLigne{\Gun^{\eta}}{\eta \geq 0}$. 
\begin{proposition}
  \label{prop:g_prop}
  Assume \tup{\Cref{assum:general_kker}}, then the following hold:
  \begin{enumerate}[wide, labelindent=0pt, label=(\alph*)]
  \item For any $\eta \geq 0$, $\Gun^\eta$ is lower bounded, lower-semi
    continuous, convex and is not coercive.
\item For any $\eta > 0$, $\Gun^\eta$ is proper and $\Gun^\eta \in \rmc(\Pens(\rset^d), \rset)$.
\end{enumerate}
\end{proposition}
\begin{proof}
See \Cref{proof:prop:g_prop}.
\end{proof}

\Cref{prop:g_prop} shows that $\Gun^\eta$ enjoys better regularity properties for $\eta >0$. In \Cref{prop:Gun_approx_mu}, we will also show that if $\eta, \alpha > 0$
then the minimizers of $\Gun^\eta_\alpha$ are stable w.r.t $\mu$. In addition,
note that for any $\eta \geq 0$, $\Gun^\eta$ is \emph{not} coercive. As a
result, the uniqueness of the minimizer of $\Gun^\eta$ is difficult to
establish.  In what follows we study the function $\Gun_{\alpha}^\eta$ for
$\alpha, \eta > 0$ and show that $\Gun^\eta_\alpha$ is coercive in this case. 
Hence, $\Gun_{\alpha}^\eta$ admits a unique minimizer $\pi_{\alpha,\eta}^\star$,
and we show that the family of minimizers
$\ensembleLigne{\pi_{\alpha, \eta}^\star}{\alpha, \eta > 0}$ is smooth with
respect to the parameters $\alpha$ and $\eta$. In \Cref{prop:tikhonov_reg} we
will characterize the limit of $\pi_{\alpha, \eta}^\star$ when both
$\alpha, \eta \to 0$.

\begin{proposition}
  \label{prop:convergence_minimum}
  Assume \tup{\Cref{assum:general_kker}}. Then the following hold:
  \begin{enumerate}[label=(\alph*)]
  \item For any $\alpha, \eta > 0$, $\Gun_{\alpha}^\eta$ is proper, strictly
    convex, coercive and lower semi-continuous. In particular, $\Gun_\alpha^\eta$
    admits a unique minimizer $\pi_{\alpha, \eta}^\star \in \Pens(\rset^d)$.
  \item
    $\mtt: \ (\alpha, \eta) \mapsto \pi_{\alpha, \eta}^\star \in \rmc(\ooint{0,
      +\infty}^2, \Pens(\rset^d))$ and
    $\dtt: \ (\alpha, \eta) \mapsto \inf_{\Pens(\rset^d)} \Gun_{\alpha}^\eta \in
    \rmc(\ooint{0,+\infty}^2, \rset)$.
  \end{enumerate}
\end{proposition}
\begin{proof}
See \Cref{proof:prop:convergence_minimum}.
\end{proof}

In most applications, $\mu$ is unknown but we have access to samples
$y^{1:M} = \{y^{k,M}\}_{k=1}^M \in (\rset^p)^M$ drawn from $\mu$, see
\cite{delaigle2008alternative,ma2011indirect,pensky2017minimax,yang2020density,
  marschner2020back,hall2005nonparametric,miao2018identifying}.  The following
stability result guarantees that if we substitute $\mu$ with the empirical
measure $\mu^M = (1/M) \sum_{k=1}^M \updelta_{y^{k,M}}$ in $\Gun_\alpha^\eta$ we
obtain a minimization problem whose solution converges to the minimizer we would obtain if we had full knowledge of $\mu$ as the number
of samples $M$ increases.
\begin{proposition}
\label{prop:Gun_approx_mu}
  Assume \tup{\Cref{assum:general_kker}}. Let
  $\alpha, \eta > 0$ and for any $\nu \in \Pens(\rset^p)$ denote
  $\pi_{\alpha, \eta}^{\nu, \star}$ the unique minimizer of~\eqref{eq:G_eta}
  with $\mu \leftarrow \nu$. Let
  $(\mu_n)_{n \in \nset} \in (\Pens_1(\rset^d))^\nset$ such that
  $\underset{n \to +\infty}{\lim} \wassersteinD[1](\mu_n, \mu) = 0$ with $\mu \in \Pens_1(\rset^d)$. Then, we have
  $\underset{n \to +\infty}{\lim} \wassersteinD[1](\pi_{\alpha, \eta}^{\mu_n, \star}, \pi_{\alpha,
    \eta}^{\mu, \star}) = 0$.
\end{proposition}
\begin{proof}
See \Cref{proof:prop:Gun_approx_mu}.
\end{proof}
We emphasize that in the proof of \Cref{prop:Gun_approx_mu}, the fact that
$\eta > 0$ plays a crucial role. Indeed, if $\eta>0$ and \Cref{assum:general_kker} holds we can exploit the fact that $y \mapsto \log(\pi[\kker(\cdot, y)] + \eta)$ is Lipschitz continuous.

\subsection{Variants and connections with other methods}
\label{sec:variants}

\subsubsection{Maximum entropy methods}
\label{sec:maximum-entropy-}
The functional $\Fun_\alpha$ can be seen as the Lagrangian
associated with the following primal problem
\begin{equation}
  \argmin \ensembleLigne{\KL{\pi}{\pi_0}}{\pi \in \Pens(\rset^d), \ \KL{\mu}{\pi \Kker} = 0}.
\end{equation}
The latter problem is a maximum entropy problem in the sense of
\cite{csiszar1975divergence} (where the entropy functional is replaced by the
Kullback--Leibler divergence).  Closely related to~\eqref{eq:minimisation} is
the functional $\tilde{\Fun}_{\alpha}$ given (when it is defined) by
\begin{equation}
\label{eq:funct_entropy}
\tilde{\Fun}_{\alpha}(\pi) = \KL{\mu}{\pi \Kker} - \alpha \Hent(\pi).
\end{equation}
The choice of an entropic penalty does not require specifying a
reference measure $\pi_0$ and connects~\eqref{eq:minimisation} with the
classical maximum entropy methods \cite{jaynes1957information}. Indeed,
\eqref{eq:funct_entropy} is the Lagrangian associated with the following primal
problem
\begin{equation}
  \argmax \ensembleLigne{\Hent(\pi)}{\pi \in \Pens_{\mathrm{H}}(\rset^d), \ \KL{\mu}{\pi \Kker} = 0},
\end{equation}
where $\Pens_{\mathrm{H}}(\rset^d)$ is the set of probability distributions with
finite entropy.  However, the functional $\tilde{\Fun}_{\alpha}$ is not lower
bounded and the corresponding minimization problem is not well-defined and
therefore is not considered here. On the contrary, $\Fun_{\alpha}$ is always
non-negative making the optimization problem well-defined.  A number of maximum
entropy approximations of solutions of Fredholm integral equations have been
proposed in the literature, most of which maximize the entropy subject to moment
constraints obtained by integrating $\mu$ and $\pi \Kker$ w.r.t. a (possibly
infinite) set of basis functions \cite{kopec1993application,
  islam2020approximating, jin2016solving,
  mead1986approximate}. 
These approaches work particularly well in the one-dimensional case, as in this case the maximum entropy solution can be written analytically (see, e.g., \cite[Proposition 3.1]{islam2020approximating}). For higher dimensional problems, the maximum entropy solution is usually approximated by discretizing the support of $\pi$ (see, e.g., \cite{molina1992bayesian}).
\subsubsection{Generalized Bayesian inference}

As mentioned in the introduction, minimizing $\Fun$ results in the maximum
likelihood estimator~\eqref{eq:kl}. Given a set of observations
$y_1, \ldots, y_M$ from $\mu$, the log-likelihood for the infinite-dimensional
parameter $\pi$ is \cite{chae2018convergence}
\begin{equation}
\label{eq:gbi_likelihood}
\textstyle{
  (1/M)\sum_{j=1}^M \log \pi[\kker(\cdot, y_j)]= \int_{\rset^p} \log \pi[\kker(\cdot, y)]\rmd \mu^M(y),
  }
\end{equation}
with $\mu^M = (1/M) \sum_{j=1}^M \updelta_{y_j}$.
For any $\alpha>0$, minimizing the functional
$\Gun_{\alpha}$ in~\eqref{eq:G} is equivalent to minimizing
\begin{align}
  \textstyle{
  -\frac{1}{\alpha}\int_{\rset^p} \log(\pi[\kker(\cdot, y)]) \rmd \mu(y) + \KL{\pi}{\pi_0},
  }
\end{align}
which can be interpreted as generalized Bayesian inference
in the sense of adjusting the weight (or learning rate) of the likelihood~\eqref{eq:gbi_likelihood} w.r.t. the prior $\pi_0$ (see, e.g., \cite[Section 3]{bissiri2016general} and \cite{grunwald2012safe}).
In particular, if we assume that the
target measure $\pi$ is a discrete distribution with a fixed and finite number of
components (a straightforward case which is not considered here), \ie \
$\pi = \sum_{i=1}^{N} \updelta_{x^{i}}p_i$, with $\sum_{i=1}^N p_i = 1$ and
$\{x^i\}_{i=1}^N \in (\rset^d)^N$, then~\eqref{eq:fe} is a finite mixture model
and minimizing~\eqref{eq:G} corresponds to posterior inference for the locations
$\{x^i\}_{i=1}^N$ and the weights $\{p_i\}_{i=1}^N$ with prior $\pi_0$ and learning rate $\alpha^{-1}$.

Choosing the appropriate $\alpha$ then amounts to the calibration problem in generalized Bayesian inference with standard Bayesian inference obtained when $\alpha=1$ \cite{grunwald2012safe, bissiri2016general}. Contrary to the standard setting, in which values of $0<\alpha^{-1}<1$ are preferred since the model is believed to be misspecified, we are in the setting in which we want to reduce the influence of the prior $\pi_0$ which is used as a regularizer in~\eqref{eq:G} and favour values of $\alpha^{-1}>1$.

\subsubsection{Tikhonov regularization}

We now move onto considering the link between $\Gun_{\alpha}$ and Tikhonov
regularization. In fact, minimising $\Gun_{\alpha}$ can be seen as the
probabilistic counterpart to the classical Tikhonov regularization setting
\cite{groetsch1984theory} which we recall briefly.

Let $\rmA \in \ \rset^{p \times d}$ and $y \in \rset^p$ (note that this discussion
is still valid in more general Hilbert spaces but we restrict ourselves to the
finite-dimensional case for the sake of clarity).  Define
$g: \ \rset^d \to \coint{0, +\infty}$ for any $x \in \rset^d$ by
$g(x) = \normLigne{\rmA x -y}^2$. Denote by $\rmA^\dagger$ the Moore-Penrose
inverse of $\rmA$, see \cite{moore1920reciprocal} for a definition.
$\rmA^\dagger y \in \argmin_{\rset^d} g$ and moreover
$\rmA^\dagger y \in \argmin \ensembleLigne{x \mapsto \norm{x}}{x \in
  \argmin_{\rset^d}g}$, \ie \ $\rmA^\dagger y$ is the solution of the
minimization of $g$ with least norm. For any $\alpha \geq 0$, the Tikhonov regularization with level $\alpha$ of $g$ is given by $g_{\alpha}: \ \rset^d \to \coint{0, +\infty}$ such that $g_{\alpha}(x) = g(x) + \alpha \norm{x}^2$ for any $x \in \rset^d$. There exists a unique $x_{\alpha} \in \rset^d$ such that
$x_{\alpha} \in \argmin_{\rset^d} g_{\alpha}$. It is known that
$\lim_{\alpha \to 0} x_{\alpha} = \rmA^\dagger y$, \ie \ the solution of the
Tikhonov regularization converges towards the solution of $g$ with minimal norm
\cite{groetsch1984theory}.

In \Cref{prop:tikhonov_reg}, we show a similar property in the probabilistic
setting. In this case the norm function is replaced by the Kullback--Leibler
divergence with respect to some reference measure $\pi \mapsto
\KL{\pi}{\pi_0}$. Given the functional
$\Gun: \ \Pens(\rset^d) \to \ocint{-\infty, \infty}$ in~\eqref{eq:G_eta}, we
obtain for any $\alpha, \eta \geq 0$ a regularized functional $\Gun_{\alpha}^\eta(\pi) = \Gun^\eta(\pi) + \alpha \KL{\pi}{\pi_0}$. Similarly to the
Tikhonov regularization, there exists a unique
minimizer to $\Gun_{\alpha}^\eta$, $\pi_{\alpha, \eta}^\star, \in \Pens(\rset^d) $, see \Cref{prop:convergence_minimum}. In \Cref{prop:tikhonov_reg}, we show that any
limiting point of the sequences $(\pi_{\alpha_n, \eta_n}^\star)_{n \in \nset}$ with
$\lim_{n \to +\infty} \alpha_n =0$ and $\lim_{n \to +\infty} \eta_n =0$, is a minimizer of $\Gun$ with least
Kullback--Leibler divergence with respect to $\pi_0$.

To control the behaviour of the family of minima
$\ensembleLigne{\inf_{\pi \in \Pens(\rset^d)}\Gun_{\alpha}^\eta(\pi)}{\alpha, \eta > 0}$
when $\alpha$ is close to $0$, we make use of the following assumption which
controls the tail behaviour of $\pi_0$ and the tail behaviour of the density $\kker$.

\begin{assumption}
  \label{assum:pi0}The following hold:
  \begin{enumerate}[wide, labelindent=0pt, label=(\alph*)]
  \item  $\pi_0$ admits a density w.r.t. the Lebesgue measure, $\rmd \pi_0(x) = \pi_0(x)\rmd x$, with $\pi_0(x)\propto \exp\left[-U(x)\right]$, where, 
  $U: \ \rset^d \to \rset$ is such that there exist $\tau, C_1 > 0$ satisfying for any $x \in \rset^d$
  \begin{equation}
     -C_1 - \tau \norm{x}^2 \leq U(x) \leq C_1 + \tau \norm{x}^2.
    \end{equation}
  \item $\mu \in \Pens_2(\rset^p)$ and there exists $C_2 \geq 0$ such that for any
    $x \in \rset^d$ and $y \in \rset^p$
    \begin{equation}
      \kker(x,y) \geq C_2^{-1} \exp[-C_2 (1 + \norm{x}^2+\norm{y}^2)]. 
    \end{equation}
  \end{enumerate}
  \end{assumption}

  Under this assumption we have the following result.

\begin{proposition}
  \label{prop:tikhonov_reg}
  Assume \tup{\Cref{assum:general_kker}} and \tup{\Cref{assum:pi0}}. Then, for
  any $\alpha > 0$, $\eta \geq 0$,
  $\pi_{\alpha, \eta}^\star \in \Pens_2(\rset^d)$
  admits a density w.r.t the Lebesgue measure.  If there exist 
  $\pi^\star \in \Pens_2(\rset^d)$,
  $(\alpha_n)_{n \in \nset} \ooint{0,+\infty}^\nset$,
  $(\eta_n)_{n \in \nset} \in \coint{0, +\infty}^\nset$ such that
  $\lim_{n \to +\infty} \alpha_n = 0$, $\lim_{n \to +\infty} \eta_n = 0$ and
  $\lim_{n \to +\infty} \wassersteinD[2](\pi_{\alpha_n,
    \eta_n}^{\star},\pi^\star) = 0$, then
  \begin{equation}
    \textstyle{
      \pi^{\star} \in \argmin \ensembleLigne{\KL{\pi}{\pi_0}}{\pi \in \argmin_{\Pens_2(\rset^d)} \Gun}.
      }
  \end{equation}
\end{proposition}
\begin{proof}
See \Cref{app:proof_tikhonov}.
\end{proof}


\section{McKean-Vlasov Stochastic Differential Equation and minimization of functionals}
\label{sec:part-appr-mcke}

In most interesting cases the direct minimization of~\eqref{eq:G_eta} is not
tractable.  In this work, we propose to follow a Wasserstein gradient flow
approach in order to approximate these minimizers. Our methodology stems from
the connection between minimization of functionals in the space of probability
measures and partial differential equations (PDEs) pointed out in
\cite{jordan1998variational, otto2001geometry}. In particular, we draw links
between the minimization of $\Gun_{\alpha}^{\eta}$ and a SDE whose invariant measure is the minimizer of
$\Gun_\alpha^\eta$. An informal presentation of the bridges between Wasserstein
gradient flows and a certain class of SDEs, namely McKean-Vlasov SDEs (MKVSDEs), is
presented in \Cref{sec:gradient-flows}. We investigate the long-time behaviour of
these MKVSDEs in the context of the minimization of $\Gun^\eta_\alpha$ in
\Cref{sec:existence-uniqueness}. Finally, we show that there exists a system of
interacting SDEs which approximates the obtained MKVSDE in \Cref{sec:particle-system}.

\subsection{Minimization and gradient flows}
\label{sec:gradient-flows}
In this section, we set $\alpha, \eta >0$ and informally derive the
continuous-time dynamics we study in the rest of this paper by drawing a link
with gradient flows; formal results are given in \Cref{app:subdifferential}. We assume that all the probability measures we consider admit
densities w.r.t the Lebesgue measure and do not distinguish between the
distribution and its density. Let $(\pi_t)_{t \geq 0}$ be a family of
probability measures satisfying (in a weak sense)
\begin{equation}
\label{eq:gf}
  \partial_t \pi_t = -\mathrm{div}(\pi_t \mathbf{v}_t),
\end{equation}
where $\mathbf{v}_t$ belongs to the Wasserstein subdifferential of a functional
$\mathscr{F}: \ \Pens_2(\rset^d) \to \ccint{0,+\infty}$, see \cite[Definition
10.1.1]{ambrosio2008gradient} for a definition.
 Then $(\pi_t)_{t \geq 0}$ is called a Wasserstein gradient flow associated
with $\mathscr{F}$ \cite[Definition 11.1.1]{ambrosio2008gradient} and it has
been shown in numerous settings that such gradient flows converge towards the
minimizers of $\mathscr{F}$ when they exist
\cite{ambrosio2008gradient,santambrogio2017euclidean}. In
\Cref{prop:total_subdifferential} in \Cref{app:subdifferential} we show that
\begin{align}\label{eq:subdifferential}
  \textstyle{
   x \mapsto -\int_{\rset^p} (\eta + \pi[\kker(x, y)])^{-1} \nabla_1
  \kker(x, y) \rmd \mu(y) + \alpha \{\nabla U(x) + \nabla \log\pi(x)\},
  }
 \end{align}
 belongs to the subdifferential of $\Gun_{\alpha}^{\eta}$ at $\pi$ under mild
 conditions and assuming that $\pi_0(x) \propto
 \exp[-U(x)]$. We also derive the corresponding gradient flow equation
\begin{equation}
\textstyle{\partial_{t}\pi_{t}=  -\mathrm{div}\left(\pi_{t}\left[\int_{\rset^p} (\pi_t[\kker(\cdot, y)]+\eta)^{-1}\nabla_1 \kker(x,y) \rmd \mu(y) - \alpha\nabla U(x)\right]\right)+\alpha\Delta\pi_{t}. }
\label{eq:PDE_BM}
\end{equation}
For strongly geodesically convex (i.e. convex along geodesics) functionals the
gradient flow~\eqref{eq:gf} converges geometrically towards the unique
minimizer.  However, in our setting $\Gun_{\alpha}^\eta$ is not geodesically
convex but only convex. Using recent results from \cite{hu2019mean} we establish the convergence of the Wasserstein gradient flow in
\Cref{sec:existence-uniqueness} but without quantitative convergence rates.

We use the gradient flow equation~\eqref{eq:PDE_BM} to draw a link between the
minimization of $\Gun_{\alpha}^{\eta}$ and MKVSDEs \cite{mckean1966class}, a class of SDEs in which the drift
and diffusion coefficients depend not only on the current position of the
solution but also on its distribution.  Since~\eqref{eq:PDE_BM} is a
Fokker-Plank equation, we informally derive the corresponding SDE
\begin{equation}
  \label{eq:nonlinearSDE}
  \textstyle{
  \rmd \bfX_t^\star =  \defEnsLigne{\int_{\rset^p}\bmketa(\bfX_t^\star, \lambdabf_t^\star, y) \rmd \mu(y) - \alpha \nabla U(\bfX_t^\star)} \rmd t + \sqrt{2\alpha} \rmd \bfB_t \qquad \bfX_0^\star \in \rset^d,}
\end{equation}
where $(\bfB_t)_{t \geq 0}$ is a $d$-dimensional Brownian motion, $\lambdabf_t^\star$ is
the distribution of $(\bfX_t^\star)_{t \geq 0}$ and for any $\nu \in \Pens(\rset^d)$ and $(x, y) \in \rset^d\times\rset^p$
\begin{equation}
  \label{eq:drifto}
  \bmketa(x, \nu, y) = \nabla_1 \kker(x, y) / (\nu[\kker(\cdot, y)] + \eta).
\end{equation}
In particular, note that $\bmketa$ is always defined since $\eta > 0$. 
In the following section, we establish various theoretical properties of~\eqref{eq:nonlinearSDE}: we show that its solution exists, is
unique and converges towards the minimizer of $\Gun_\alpha^\eta$.

\subsection{Existence, uniqueness and long-time behaviour}
\label{sec:existence-uniqueness}
We are now concerned with the theoretical properties
of~\eqref{eq:nonlinearSDE}. In particular, we show that
under~\tup{\rref{assum:general_kker}} the MKVSDE
in~\eqref{eq:nonlinearSDE} admits a unique strong solution and is ergodic
provided that the gradient of the potential  $U$ is Lipschitz continuous. We present here several assumptions on the regularity of $U$ which we use throughout the remainder of this paper. However, the results presented in this section only rely on \tup{\rref{assum:lip_U}-\ref{item:lip}} and \tup{\rref{assum:lip_U}-\ref{item:dissi}}.
\begin{assumption}
  \label{assum:lip_U}
  The following hold:
  \begin{enumerate}[wide, labelindent=0pt, label=(\alph*)]
  \item \label{item:lip}    There exists $\Ltt \geq 0$ such that
    $\norm{\nabla U(x_1) - \nabla U(x_2)} \leq \Ltt \norm{x_1 - x_2}$, for any $x_1, x_2 \in \rset^d$.
      \item \label{item:dissi}  There exist $\mtt,\ctt > 0$ such that for any
    $x_1, x_2 \in \rset^d$,
    $\langle \nabla U(x_1) - \nabla U(x_2), x_1 - x_2 \rangle \geq \mtt
    \normLigne{x_1 - x_2}^2 - \ctt$.
  \item \label{item:lip_gro}    There exists $\Ltt_2 \geq 0$ such that for any $x_1, x_2 \in \rset^d$,
    $\norm{\nabla^2 U(x_1) - \nabla^2 U(x_2)} \leq \Ltt_2 \norm{x_1 - x_2}$. 
  \end{enumerate}     
\end{assumption}
Under these conditions, if $\eta>0$ the drift given in~\eqref{eq:drifto} is
Lipschitz continuous and we can use standard tools for McKean-Vlasov processes
to establish existence and uniqueness.
\begin{proposition}
\label{prop:existence_uniqueness}
Assume \tup{\rref{assum:general_kker}} and
\tup{\rref{assum:lip_U}-\ref{item:lip}}. Then for any $\alpha, \eta > 0$ there exists a unique strong
solution to~\eqref{eq:nonlinearSDE} for any initial condition
$\bfX_0^\star \in \Pens_1(\rset^d)$.
\end{proposition}
\begin{proof}
See \Cref{app:proof_eu}.
\end{proof}

The previous proposition is limited to the case where $\eta > 0$.
Indeed, if $\eta=0$ the drift is \emph{not} Lipschitz continuous and the
SDE~\eqref{eq:nonlinearSDE} might be unstable, with solutions
existing up to a (possibly small) explosion time.  Having shown
that~\eqref{eq:nonlinearSDE} admits a unique strong solution when $\eta>0$, we verify that the 
distribution $\bflambda_t^\star$ converges to the unique minimizer
$\pi_{\alpha, \eta}^\star$ of $\Gun_{\alpha}^\eta$ when $t \to +\infty$.
\begin{proposition}
  \label{prop:la_convergence_star}
  Assume \tup{\rref{assum:general_kker}},  
  \tup{\rref{assum:lip_U}-\ref{item:lip}} and \tup{\rref{assum:lip_U}-\ref{item:dissi}}. Then for any $\alpha, \eta > 0$, we have
  $\underset{t \to +\infty}{\lim} \wassersteinD[2](\bflambda_t^\star, \pi_{\alpha,
    \eta}^\star) = 0$.
\end{proposition}
\begin{proof}
See \Cref{app:proof_invariant}.
\end{proof}
There exists a rich literature on the exponential ergodicity of MKVSDE, see
\cite{butkovsky2014ergodic,eberle2016reflection,malrieu2001logarithmic,bogachev2019convergence}
for instance. However, these results require the interaction term $\bmketa$ to
be small when compared to $\alpha \nabla U$. This is not the case in our
applications where we study the behaviour of the McKean-Vlasov process for small
values of $\alpha > 0$. The original approach of \cite{hu2019mean} differs from
these previous works and relies on the fact the MKVSDE under consideration stems
from a Wasserstein gradient flow. However, \cite[Theorem 2.11]{hu2019mean} does
not establish quantitative bounds contrary to
\cite{butkovsky2014ergodic,eberle2016reflection,malrieu2001logarithmic,bogachev2019convergence}
which derive explicit convergence rates.
The result above can be strengthened to show that, if $\alpha_t\to 0$ at the appropriate rate, then $\Gun^\eta(\bflambda_t^\star)$ converges towards the minimum of the unregularized
functional (see \cite[Theorem 4.1]{chizat2022mean} and \cite[Theorem 1]{nitanda2022convex}).

\subsection{Particle system}
\label{sec:particle-system}
The MKVSDE~\eqref{eq:nonlinearSDE} has several shortcomings from a
methodological point of view. First, since the drift of the SDE depends not only
on the current position of the solution but also on its distribution, it is not
possible to sample from~\eqref{eq:nonlinearSDE} even in simplistic
settings. Second, as is the case for classical SDEs, it is usually difficult to sample
from continuous-time processes. We circumvent the first issue by introducing a
particle system \cite{mckean1966class, bossy1997stochastic}
$(\bfX_t^{1:N})_{t \geq 0} = \{(\bfX_t^{k,N})_{t \geq 0}\}_{k=1}^N$ for any
$N \in \nset^\star$ which satisfies a classical SDE. In the limit $N \to +\infty$
we have that $(\bfX_t^{1,N})_{t \geq 0}$ approximates
$(\bfX_t^{\star})_{t \geq 0}$. In \Cref{sec:implementation} we will investigate
a time discretization scheme for this particle system SDE in order to obtain a
particle system Markov chain which can be implemented. 

We introduce the particle system $(\bfX_t^{1:N})_{t \geq 0}$ which satisfies the
following SDE: for any $k \in \{1, \dots, N\}$, $\bfX_0^{k,N} \in \rset^d$ and 
\begin{equation}
  \label{eq:particle}
  \textstyle{
    \bfX_t^{k,N} = \bfX_0^{k,N} + \int_0^t b(\bfX_{s}^{k,N}, \lambdabf_s^{N}(\bfX_0^{1:N})) \rmd s + \sqrt{2 \alpha} \bfB_t^{k},
    }
\end{equation}
where for any $x \in \rset^d$ and $\nu \in \Pens(\rset^d)$ we have 
\begin{equation}
  \label{eq:b}
  \textstyle{
    b(x, \nu) = \int_{\rset^p}\bmketa(x, \nu, y) \rmd \mu(y) - \alpha \nabla U(x),
    }
\end{equation}
$\ensembleLigne{(\bfB_t^k)_{t \geq 0}}{k \in \nset}$ is a family of independent
Brownian motions and $\bflambda_t^{N}(\bfX_0^{1:N})$ is the empirical measure
associated with $\{(\bfX_t^{k,N})_{t \geq 0}\}_{k=1}^N$ with starting point
$\bfX_0^{1:N}$, \ie\ we have for any $N \in \nsets$
\begin{equation}
  \label{eq:particle_empirical}
  \textstyle{
    \bflambda^N_t(\bfX_0^{1:N}) = (1/N) \sum_{k=1}^N \updelta_{(\bfX_t^{k,N})}.
    }
\end{equation}
In the following proposition we show using classical propagation of chaos tools
that there exists a unique strong solution to~\eqref{eq:particle} and that this
solution approximates~\eqref{eq:nonlinearSDE} for any finite time horizon.

\begin{proposition}
  \label{prop:propagation_chaos}
  Assume \tup{\rref{assum:general_kker}} and
  \tup{\rref{assum:lip_U}-\ref{item:lip}}. Then, for any $\alpha, \eta > 0$ and $N \in \nset^\star$
  there exists a unique strong solution to~\eqref{eq:particle} for any initial
  condition $\bfX_0^{1:N}$ such that $\mathcal{L}(\bfX_0^{1:N}) \in \Pens_1((\rset^d)^N)$ and
  $\{\bfX_0^{k,N}\}_{k=1}^N$ is exchangeable. In addition, if
  $\bfX_0^{\ell,N} = \bfX_0^\star$, for any $T \geq 0$ there exists $C_T \geq 0$
  such that for any $N \in \nsets$ and $\ell \in \{1, \dots, N\}$
  \begin{equation}
    \textstyle{
      \expesqLigne{\sup_{t \in \ccint{0,T}} \normLigne{\bfX_t^\star - \bfX_t^{\ell,N}}^2} \leq C_T N^{-1/2}.
      }
  \end{equation}
\end{proposition}
\begin{proof}
See \Cref{sec:proof-prop-chaos}.
\end{proof}

We now turn to the exponential ergodicity of the particle system. Using recent
advances in the theory of convergence of Markov chains
\cite{debortoli2019convergence} we derive quantitative bounds. However, these bounds are not uniform w.r.t to the number of particles and therefore
we cannot use this result to conclude that the McKean-Vlasov process itself is
exponentially ergodic by letting $N \to +\infty$.

\begin{proposition}
\label{prop:particle_ergodic}
Assume \tup{\rref{assum:general_kker}}, \tup{\rref{assum:lip_U}-\ref{item:lip}} and \tup{\rref{assum:lip_U}-\ref{item:dissi}}. Then, for any $\alpha, \eta > 0$ and $N \in \nsets$ there
exist $C_N \geq 0$ and $\rho_N \in \coint{0,1}$ such that for any
$x_1^{1:N}, x_2^{1:N} \in (\rset^d)^N$ and $t \geq 0$
  \begin{equation}
    \wassersteinD[1](\bflambda_t^N(x_1^{1:N}), \bflambda_t^N(x_2^{1:N})) \leq C_N \rho_N^t \normLigne{x_1^{1:N} - x_2^{1:N}},
  \end{equation}
  where for any $x^{1:N} \in (\rset^d)^N$, $\bflambda_t^N(x^{1:N})$ is the
  distribution of $(\bfX_t^{1:N})_{t \geq 0}$ with initial condition $x^{1:N}$.
  In particular,~\eqref{eq:particle} admits a unique invariant probability
  measure denoted $\pi^N \in \Pens_p((\rset^d)^N)$ and for any
  $x^{1:N} \in (\rset^d)^N$ and $t \geq 0$
  \begin{equation}
    \textstyle{
      \wassersteinD[1](\bflambda_t^N(x^{1:N}), \pi^N) \leq C_N \rho_N^t \parenthese{\normLigne{x^{1:N}} + \int_{\rset^d} \normLigne{\tilde{x}} \rmd \pi^N(\tilde{x})}.
      }
  \end{equation}
\end{proposition}
\begin{proof}
  See \Cref{proof:prop:particle_ergodic}
\end{proof}
In the two previous propositions, the exponential ergodicity and the
approximation between the particle system and the McKean-Vlasov process are only
valid for a finite number of particles or a finite time-horizon, respectively. In
particular, we cannot directly conclude that the projection of the invariant
distribution of the particle system onto its first component on $\Pens(\rset^d)$, denoted by $\pi^{1,N}$, is close to
$\pi_{\alpha, \eta}^\star$. This is the topic of the next proposition.

\begin{proposition}
\label{prop:particles_min}
Assume \tup{\rref{assum:general_kker}}, \tup{\rref{assum:lip_U}-\ref{item:lip}} and \tup{\rref{assum:lip_U}-\ref{item:dissi}}. Then, for any $\alpha, \eta > 0$,
$\lim_{N \to +\infty} \wassersteinD[1](\pi^{1,N}, \pi_{\alpha, \eta}^\star) =
0$. In addition, $\pi_{\alpha, \eta}^\star$ is the unique invariant probability
measure of~\eqref{eq:nonlinearSDE}.
\end{proposition}
\begin{proof}
See \Cref{proof:prop:particles_minapp:ps}.
\end{proof}


\section{Numerical Implementation}
\label{sec:implementation}
In the previous section, we introduced a continuous time Markov process $(\bfX_t^{1:N})_{t \geq 0}$ such that the distribution $\bfX_t^{1,N}$
approximates $\pi_{\alpha, \eta}^\star$ for large values of $t \geq 0$ and
$N \in \nset$. In this section, we derive an algorithm approximating this
particle system. 

In order to obtain a numerical approximation of~\eqref{eq:particle} we need two consider two factors: first, as in the case of classical SDEs, the continuous time process~\eqref{eq:particle} cannot be simulated directly, and we need to introduce time discretization, secondly, in practice the integral w.r.t. $\mu$ in~\eqref{eq:b} is intractable. 

First, in \Cref{sec:time-discretization} we consider an Euler--Maruyama discretization of the
particle system. For stability
issues we consider a tamed version of the classical Euler--Maruyama
discretization.
To address the second issue, given that in applications we will often be in the situation in which $\mu$ is only known through a sample, we approximate the integral w.r.t. $\mu$ in~\eqref{eq:b} with a sample average
In \Cref{sec:approximation-mu}, we introduce a second discrete-time particle system, obtained by approximating $\mu$ with its empirical average $(1/M) \sum_{k=1}^M \updelta_{y^{k,M}}$
where $y^{1:M} = \{y^{k,M}\}_{k=1}^M$ is a sample from $\mu^{\otimes M}$. We show that the two particle systems are close for large values of $M$. In particular,
\Cref{prop:approx_mu} ensures the quantitative stability of the particle system.
Finally, in \Cref{sec:choice} we
discuss the choice of parameters for the algorithm.

\subsection{Tamed Euler--Maruyama discretization}
\label{sec:time-discretization}

In order to
obtain a Markov chain which can be implemented, we consider a time discretization of~\eqref{eq:particle}.  More precisely, for any
$N \in \nset$, we consider the following tamed Euler--Maruyama discretization
given by $X_0^{1:N} \in (\rset^d)^N$ and for any $n \in \nset$ and $k \in \{1, \dots, N\}$
\begin{equation}
  \label{eq:tamed}
  X_{n+1}^{k,N}  = X_n^{k,N} + \gamma b( X_n^{k,N}, \lambda^{N}_n)/(1 + \gamma \normLigne{b( X_n^{k,N}, \lambda^{N}_n)}) + \sqrt{2 \alpha} Z_{n+1}^k,
\end{equation}
where $\{Z_n^k\}_{k, n \in \nset}$ is a family of independent Gaussian random
variables, $\gamma > 0$ is a stepsize, $b$ is given by~\eqref{eq:b} and
for any $n \in \nset$, we have that
$\lambda^{N}_n = (1/N) \sum_{k=1}^N \updelta_{X_n^{k,N}}$. 
The choice of a tamed scheme is motivated by the behaviour of the drift~\eqref{eq:b}. As shown in the Appendix, $b$ is Lipschitz continuous for any $\eta>0$, however, for small values of $\eta$ the Lipschitz constant can be large.
In practice, we observe that the use of a tamed Euler--Maruyama discretization scheme prevents
some numerical stability issues appearing for small values of $\eta >0$. 

The
tamed Euler--Maruyama discretization (and its extension to the Milstein scheme) for
classical SDEs has been investigated in
\cite{hutzenthaler2015numerical,hutzenthaler2012strong,wang2013tamed,brosse2019tamed}. More
recently, strong approximation results have been established for McKean--Vlasov
SDEs in \cite{bao2020first}. In the following proposition, we use the results of
\cite[Theorem 2.5]{bao2020first} to derive strong approximation results between
$(X_n^{k,N})_{n \in \nset}$ and $(\bfX_t^{k,N})_{t \geq 0}$ for any $k \in \{1, \dots, N\}$; since the diffusion coefficient of~\eqref{eq:particle} is constant, the tamed scheme~\eqref{eq:tamed} coincides with a Milstein scheme, allowing us to show strong convergence of order 1.

\begin{proposition}
\label{prop:euler}
Assume \tup{\rref{assum:general_kker}} and \rref{assum:lip_U}. Then for any
$\eta, \alpha > 0$ and any $T \geq 0$ there exists $C_T \geq 0$ such that for
any $N \in \nsets$, $\ell \in \{1, \dots, N\}$ and $\gamma >0$
  \begin{equation}
    \textstyle{\expeLigne{\sup_{n \in \{0, \dots, n_T\}} \normLigne{\bfX_{n \gamma}^{\ell,N}- X_{n}^{\ell,N}}} \leq C_T \gamma}
  \end{equation}
  where $n_T = \floor{T/\gamma}$.
\end{proposition}
\begin{proof}
See \Cref{app:euler}.
\end{proof}
A similar result holds for a continuous-time interpolation of $(X_{n}^{\ell,N})_{n \in \nset}$, see
\cite[Theorem 2.5]{bao2020first} for more details. For small values of
$\gamma >0$ and large values of $n\gamma$ and $N\in \nset$ we get that
$(1/N) \sum_{k=1}^N \updelta_{X_n^{k,N}}$ is an approximation of
$\pi_{\alpha, \eta}^\star$.

\subsection{Stability of the particle system w.r.t the observed measure}
\label{sec:approximation-mu}
\Cref{prop:Gun_approx_mu} shows that if $\mu$ is approximated
by the empirical measure $\mu^M$ obtained from a sample $y^{1:M} = \{y^{k,M}\}_{k=1}^M$
then the minimizer of $\Gun_{\alpha}^{\eta}$ obtained using $\mu^M$ instead of
$\mu$ converges to the minimizer of $\Gun_{\alpha}^{\eta}$ obtained using $\mu$
as $M\rightarrow +\infty$. We now consider the discrete-time alternative particle system
$\{(X_n^{k,N, M})_{n\in\nset}\}_{k=1}^N$ which evolves according to~\eqref{eq:tamed} with $b$ replaced by its empirical approximation obtained by approximating $\mu$ with the empirical measure of the $M$ samples $y_1,\dots, y_M$:
\begin{equation}
  \label{eq:particle_M}
  \textstyle{
    X_{n+1}^{k,N,M} = X_n^{k,N,M} + \gamma b^M( X_n^{k,N, M}, \lambda^{N, M}_n)/(1 + \gamma \normLigne{b^M( X_n^{k,N, M}, \lambda^{N, M}_n)}) + \sqrt{2 \alpha} Z_{n+1}^k
    }
\end{equation}
where $\{Z_n^k\}_{k, n \in \nset}$ is a family of independent Gaussian random
variables, $\gamma > 0$ is a stepsize, $\lambda^{N, M}_n = (1/N) \sum_{k=1}^N \updelta_{X_n^{k,N, M}}$ for any $n \in \nset$ and for any $x \in \rset^d$ and $\nu \in \Pens(\rset^d)$ we have
\begin{equation}
  \label{eq:bM_def}
  \textstyle{b^M(x, \nu) = \int_{\rset^p}\bmketa(x, \nu, y) \rmd \mu^M(y) - \alpha \nabla U(x) = (1/M)\sum_{j=1}^M\bmketa(x, \nu, y_{j})  - \alpha \nabla U(x).}
\end{equation}
In the next proposition, we show that the two particle systems~\eqref{eq:tamed} and~\eqref{eq:particle_M} are close when $M$ is large.

\begin{proposition}
\label{prop:approx_mu}
Assume \tup{\rref{assum:general_kker}}, \tup{\rref{assum:lip_U}-\ref{item:lip}}
and \tup{\rref{assum:lip_U}-\ref{item:dissi}}. Let $\eta, \alpha > 0$. For any
$N,M \in \nsets$, let $(X_n^{1,N})_{t \geq 0}$ and $(X_n^{1,N,M})_{t \geq 0}$ be
obtained with~\eqref{eq:tamed} and~\eqref{eq:particle_M} respectively, driven by
the same underlying family of independent Gaussian random variables
$\{Z_n^k\}_{k, n \in \nset}$, with initial condition $X_0^{1:N}$ such that
$\mathcal{L}(X_0^{1:N}) \in \Pens_1((\rset^d)^N)$, $\{X_0^{k,N}\}_{k=1}^N$ is
exchangeable and $\{X_0^{k,N}\}_{k=1}^N = \{X_0^{k,N,M}\}_{k=1}^N$.  Then, for
any $T \geq 0$ there exists $C_T \geq 0$ such that for any
$\ell \in \{1, \dots, N\}$ and $\gamma>0$
  \begin{equation}
    \textstyle{
      \expeLigne{\sup_{n\in \{0, \dots, n_T\}} \normLigne{X_n^{\ell,N}- X_n^{\ell,N, M}}} \leq C_T\gamma M^{-1/2}.
      }
  \end{equation}
\end{proposition}
\begin{proof}
See \Cref{app:approx_mu}.
\end{proof}
An equivalent result showing convergence of the continuous process $$\bfX_t^{k,N} = \bfX_0^{k,N} + \int_0^t b^M(\bfX_{s}^{k,N}, \lambdabf_s^{N}(\bfX_0^{1:N})) \rmd s + \sqrt{2 \alpha} \bfB_t^{k},$$ with $b^M$ as in~\eqref{eq:bM_def}, to~\eqref{eq:particle} could also be obtained. The argument follows the same line of that used to prove \Cref{prop:propagation_chaos} and exploits the convergence of $b^M$ to $b$ established in the proof of \Cref{prop:approx_mu}. The resulting rate is $M^{-1/2}$ (see \Cref{app:prop10_continuous_time}).

As we discuss in the next section, in some cases it might be beneficial to resample from the empirical measure $\mu^M$ at each time step $n$. In the following proposition we show that this resampling operation does not change the error bound obtained in \Cref{prop:approx_mu}.
\begin{proposition}
\label{prop:resampling}
Assume \tup{\rref{assum:general_kker}}, \tup{\rref{assum:lip_U}-\ref{item:lip}}
and \tup{\rref{assum:lip_U}-\ref{item:dissi}}. Let $\eta, \alpha > 0$. For any
$N,M \in \nsets$, let $(X_n^{1,N})_{t \geq 0}$ be obtained with~\eqref{eq:tamed}
and $(X_n^{1,N,m})_{t \geq 0}$ with~\eqref{eq:particle_M} where at each
time $n$ an i.i.d. sample of size $m$ is drawn from $\mu^M$.  Assume that
$(X_n^{1,N})_{t \geq 0}$ and $(X_n^{1,N,m})_{t \geq 0}$ are driven by
the same underlying family of independent Gaussian random variables
$\{Z_n^k\}_{k, n \in \nset}$, with initial condition $X_0^{1:N}$ such that
$\mathcal{L}(X_0^{1:N}) \in \Pens_1((\rset^d)^N)$,  $\{X_0^{k,N}\}_{k=1}^N$ is
exchangeable and $\{X_0^{k,N}\}_{k=1}^N = \{X_0^{k,N,m}\}_{k=1}^N$.
Then, for any $T \geq 0$ there exists $C_T \geq 0$ such that for any
$\ell \in \{1, \dots, N\}$ and $\gamma>0$
  \begin{equation}
    \textstyle{
      \expeLigne{\sup_{n\in \{0, \dots, n_T\}} \normLigne{X_n^{\ell,N}- X_n^{\ell,N, m}}} \leq C\gamma m^{-1/2}.
      }
  \end{equation}
\end{proposition}
\begin{proof}
See \Cref{app:resampling}.
\end{proof}

The empirical measure $\lambda_n^{N, M}$ provides an approximation of $\pi_{\alpha, \eta}^\star$. In low dimensional settings, for the purpose of visualization we derive a smooth approximation of
$\pi_{\alpha, \eta}^\star$ using standard kernel density estimation tools and define
$\hat{\pi}_{n}^{N,M}: \ \rset^d \to \rset$ such that for any $x \in \rset^d$
\begin{align}
  \label{eq:kde}
  \textstyle{
\hat{\pi}_n^{N,M}(x) = (1/N) \sum_{k=1}^N \det(\rmH)^{-1/2} \upvarphi(\rmH^{-1/2}(x-X_{n}^{k, N, M})),}
\end{align}
where $\upvarphi$ is the density of a $d$-dimensional Gaussian distribution with zero mean
and identity covariance matrix and $\rmH$ is a positive definite bandwidth matrix, see, e.g, \cite{silverman1986density}.
Combining the results in \Cref{prop:propagation_chaos} and \Cref{prop:euler} with standard arguments from the kernel density estimation literature, the estimator~\eqref{eq:kde} can be shown to converge to $\pi_{\alpha, \eta}^\star$ as $N\to\infty$ and $\gamma\to0$ (e.g., \cite[Theorem 3.1]{antonelli2002rate} and \cite[Theorem 2.2]{bossy1997stochastic}).
Our final algorithm is summarized in \Cref{alg:wgf_algoo}.

\begin{algorithm}
\caption{Solving Fredholm integral equations with Wasserstein gradient flows (FE-WGF)}\label{alg:wgf_algoo}
\begin{algorithmic}
\STATE{\textbf{Require:} $N,M,n_T \in \nset$, $\alpha, \eta, \gamma > 0$,  $\rmH \in \rset^{d\times d}$, $\mu, \pi_0, \pi_{\mathrm{init}} \in \Pens(\rset^d)$.}
\STATE{Draw $\{X_0^{k,N,M}\}_{k=1}^N$ from $\pi_{\mathrm{init}}^{\otimes N}$}
\FOR{$n=0:n_T$}
\STATE{Draw $\{y^{k,M}\}_{k=1}^M$ from $\mu^{\otimes M}$}
\STATE{Compute $b_M(X_n^{k,N,M}, \lambda^{N,M}_n)$ as in~\eqref{eq:bM_def}}
\STATE{Update $X_{n+1}^{k,N,M}  = X_n^{k,N,M} + \gamma b^M(X_n^{k,N,M}, \lambda^{N,M}_n)/(1 + \gamma \normLigne{b^M(X_n^{k,N,M}, \lambda^{N,M}_n)}) + \sqrt{2 \alpha} Z_{n+1} $ as in~\eqref{eq:tamed}}
\ENDFOR
\RETURN $ \hat{\pi}_n^{N,M}(x) $ as in~\eqref{eq:kde}
\end{algorithmic}
\end{algorithm}

\subsection{Implementation guidelines}
\label{sec:choice}

Obtaining the estimator~\eqref{eq:kde} of the minimizer of
$\Gun_{\alpha}^{\eta}$ requires specification of a number of parameters: the
reference measure $\pi_0$ and the regularization parameter $\alpha$ control the
properties of the measure $\pi$ we want to reconstruct (e.g. its smoothness),
while the parameter $\eta$, the number of particles $N, M$ and the time
discretization step $\gamma$ control accuracy and stability of the numerical
implementation.

\paragraph{Choice of $\pi_0$} There are of course, many possible choices for the
reference measure $\pi_0$ to which we want the regularized solution to be close
depending on the constraints we want to impose on the reconstructed solution.
The choice of an improper reference measure $\pi_0\propto C$ with $C>0$ results
in a McKean-Vlasov SDE in which there is no dependence on $\pi_0$
\begin{equation}
  \textstyle{
  \rmd \bfX_t^\star = \defEnsLigne{ \int_{\rset^p}\bmketa(\bfX_t^\star, \lambdabf_t^\star, y) \rmd \mu(y)} \rmd t + \sqrt{2\alpha} \rmd \bfB_t,  \qquad \bfX_0^\star \in \rset^d,}
\end{equation}
where $(\bfB_t)_{t \geq 0}$ is a $d$-dimensional Brownian motion and
$\lambdabf_t^\star$ is the distribution of $(\bfX_t^\star)_{t \geq 0}$.  This
scheme corresponds to what we would obtain by applying the gradient flow
procedure to the functional~\eqref{eq:funct_entropy}.
Under additional assumptions on the tail behaviour of the kernel $\kker$, it is possible to show that this functional is proper, lower semi-continuous and coercive in an appropriate sense; \cite[Chapter 7]{crucinio2021some} shows empirically that the corresponding particle system is stable.
Nevertheless, in many scenarios some knowledge of the target will be embedded in the problem at hand. In deconvolution problems, $\mu$ is the distribution of noisy samples from
the target distribution $\pi$; in this case it is sensible to expect that $\pi$
itself will be close to $\mu$ and therefore choose $\pi_0$ to be reasonably
close to $\mu$.  More generally, for problems in which
$\kker(x, y) = \kker(y-x)$ is a deconvolution kernel, the distribution $\mu$
will be the convolution of $\pi$ and $\kker$ and therefore encapsulates
information on $\pi$. In these problems we choose $\pi_0$ to be a Gaussian
distribution with mean and variance given by the empirical mean and variance of
the sample from $\mu$. For the epidemiology examples in~\Cref{sec:epidem}, we set $\pi_0$ to be $\mu$ shifted
back by a number of days equal to the mean \cite{goldstein2009reconstructing} or
the mode \cite{wang2020bayesian} of the infection-to-confirmed delay
distribution $\kker$.  For image reconstruction, $\pi_0$ could be selected to
enforce particular characteristics in the reconstructed image, e.g. smoothness
\cite{molina1993using} or sparsity \cite{zhang2014multi, liu2019image}.
Alternatively, if a large data set of previously reconstructed images is
available one could use score matching priors \cite{kingma2010regularized} to
obtain a reference measure $\pi_0$ encapsulating all the relevant features of
the image to reconstruct. We study the influence of the reference measure
$\pi_0$ on a toy example in \Cref{app:pi0} and find that $\pi_0$
influences the shape of the minimizer of $\Gun_{\alpha}^\eta$ and its minimum
but not the speed at which convergence to the minimum occurs.

\paragraph{Choice of $\alpha$} As the parameter $\alpha>0$ controls the amount
of regularization introduced by the cross-entropy penalty, its value should be
chosen to give a good trade-off between the distance from the data distribution
$\mu$, $\KL{\mu}{\pi \Kker}$, and that from $\pi_0$, $\KL{\pi}{\pi_0}$.  A common
approach for selecting the regularization parameter in Tikhonov regularization
is cross-validation \cite{wahba1977practical, amato1991maximum}.  Since the case
in which $\mu$ is not known but a sample drawn from it is available is the most
likely in applications, we propose the following approach to cross-validation:
to estimate the value of $\alpha$ we divide the sample from $\mu$ into $L$
subsets and find the value of $\alpha$ which minimizes
\begin{align}
  \label{eq:cv}
  \textstyle{
\mathrm{CV}(\alpha) = (1/L)\sum_{j=1}^L \hat{\Gun}_{\alpha}^{\eta}(\hat{\pi}^{j}),}
\end{align}
where $\hat{\pi}^{j}$ is obtained using \Cref{alg:wgf_algoo} with the samples
from $\mu$ which are not in group $j \in \{1, \dots, L\}$ and $\widehat{\Gun}_{\alpha}^{\eta}$ is an approximation of $\Gun_\alpha^\eta$
\begin{align}
\label{eq:Gun_numerical}
\hat{\Gun}_{\alpha}^{\eta}(\hat{\pi}^{j}) := -\frac{1}{M}\sum_{\ell=1}^M\log\left(\frac{1}{N}\sum_{k=1}^N\kker(X_n^{k, N, M}, y_j)\right) +\frac{\alpha}{N}\sum_{k=1}^N \log\left(\frac{\hat{\pi}^{j}(X_n^{k, N, M})}{\pi_0(X_n^{k, N, M})}\right).
\end{align}
If some prior information on the smoothness of the solution $\pi$ is known (e.g. its variance), one could chose $\alpha$ so that the smoothness of~\eqref{eq:kde}, matches the expected smoothness of $\pi$.
We show how the value of $\alpha$ influences the reconstructions of $\pi$
on simple problems in \Cref{app:alpha}.

\paragraph{Choice of $\eta$} The parameter $\eta$ has been introduced
in~\eqref{eq:G_eta} to deal with the possible instability of the functional
$\Gun_{\alpha}^{\eta}$; we did not find performances to be significantly
influenced by this parameter as long as its value is sufficiently small.
In practice, in the
experiments in Section~\ref{sec:ex} we set $\eta\equiv 0$.
This choice might seem counter-intuitive, since most of our theoretical results are obtained with $\eta>0$, 
On one hand, this suggests that for regular enough $\kker$, the functional $\Gun_\alpha$ is well-behaved; on the other hand, we observe that the introduction of the tamed Euler--Maruyama scheme~\eqref{eq:tamed} prevents most of the numerical instability issues arising when discretizing~\eqref{eq:particle}.

\paragraph{Choice of $N, \gamma$ and $m$} 
The values of the number of particles
$N$, the time discretization step $\gamma$ and the number of samples $m$ from
$\mu$ to use at each iteration control the quality of the numerical approximation of~\eqref{eq:nonlinearSDE}, 
their choice is therefore largely application dependent.
However, the
results in \Cref{prop:propagation_chaos}, \ref{prop:euler} and \ref{prop:resampling} give the following global error estimate: for any $T \geq 0$ we have
\begin{align}
\label{eq:guideline}
  \textstyle{\expeLigne{\sup_{n \in \{0, \dots, n_T\}} \normLigne{\bfX_{n \gamma}^{\star}- X_n^{\ell,N,m}}} \leq C_T (N^{-1/2} +\gamma + \gamma m^{-1/2}).}
\end{align}

Choosing $N$ amounts to the classical task of selecting an appropriate sample size for Monte Carlo approximations, while the choice of $\gamma$ corresponds to the specification of a timescale on which to discretize a continuous time (stochastic) process; hence, one can exploit the vast literature on Monte Carlo methods and discrete time approximations of SDEs to select these values \cite{kloeden1992stochastic}.
In practice, we found that the choice of $N$ largely depends on the dimensionality of the problem; for
one-dimensional examples values between $N=200$ and $N=1000$
achieve high enough accuracy to compete with specialized algorithms (see
Sections~\ref{sec:dd}, \ref{sec:epidem}); as the dimension $d$ increases larger
values of $N$ should be considered (see Section \ref{sec:ct}). 
Similar
considerations apply to the choice of the time discretization step $\gamma$. In
particular, this parameter should be chosen taking into account the order of
magnitude of the gradient, to give a good trade-off between the Monte Carlo
error and the time discretization error. In practice, we observed good results
with $\gamma$ between $10^{-1}$ and $10^{-3}$.

The number of samples $m$ from $\mu$ to employ at each iteration is again problem dependent. In many cases, we are given a number $M$ of observations from $\mu$, if $M$ is very large, the resulting algorithm might become computationally too expensive.
In this setting, we propose the following batch strategy to approximate $\mu$: fix $N$ and at each time step approximate the integral with respect to $\mu$
using $m$ samples where $m$ is the smallest between the number of particles $N$
chosen for the particle system~\eqref{eq:particle} and the total number of
samples available from $\mu$; if $N$ is larger than the total number of observed samples from $\mu$ then the whole
sample is used at each iteration, if $N$ is smaller we resample without replacement $m$ times from the empirical measure of the sample.
This amounts to a maximum cost of $\mathcal{O}(N^2)$ per time step.

In cases in which we have access to a sampling mechanism from $\mu$, and we can therefore draw $m$ samples at each iteration, we simply set $m=N$.

\paragraph{Choice of $n_T$} It is straightforward to choose the number of time steps $n_T$, adaptively. The number of steps necessary to give
convergence of~\eqref{eq:particle} to its stationary distribution is estimated
by approximating the value of $\Gun_{\alpha}^\eta$ through numerical
integration as in~\eqref{eq:Gun_numerical} and once the value of $\Gun_{\alpha}^\eta$ stops decreasing, a
minimizer has apparently been reached and the iteration in \Cref{alg:wgf_algoo} can be stopped.

\section{Experiments}
\label{sec:ex}

We test FE-WGF on a number of examples of Fredholm integral
equations~\eqref{eq:fe} and compare our density estimator~\eqref{eq:kde} with the
reconstructions given by state-of-the-art methods for each experiment.

First, in \Cref{sec:dd} we consider a classical problem in statistics, density
deconvolution.  In this case, $\kker(x, y)= \kker(y-x)$ and~\eqref{eq:fe} models
the task of reconstructing the density of a random variable $X$ from noisy
measurements $Y$. In \Cref{sec:epidem} we
test FE-WGF on a problem drawn from epidemiology. More precisely, we aim at
reconstructing the incidence profile of a disease from the observed number of
cases. This is a particular instance of the deconvolution problem in which a
smoothness constraint on the solution $\pi$ of~\eqref{eq:fe} is necessary to
make the estimator robust to noise in the observations $Y$
\cite{miller2020statistical}.  In \Cref{sec:ct}, we consider an application to
computed tomography (CT), in which a cross-sectional image of the lungs needs to
be recovered from the radial measurements provided by the CT scanner.  Finally,
in \Cref{sec:hd}, we study how the performances of the proposed method scale
with the dimensionality of the support of $\pi$ on a toy model.

In most examples, $\mu$ is known through a sample and at each iteration
we resample $m$ times from its empirical distribution. For image restoration
problems we consider the observed distorted image as the empirical distribution
of a sample from $\mu$ and use it to draw $m$ samples at each iteration in \Cref{alg:wgf_algoo}. In order to implement FE-WGF we follow the guidelines
provided in \Cref{sec:choice}.
Julia code to reproduce all examples is available online at \url{https://github.com/FrancescaCrucinio/FE_WGF}.

\subsection{Density Deconvolution}
\label{sec:dd}

The focus of this section is on deconvolution problems, \ie \ those in which
$\mu$ and $\pi$ satisfy~\eqref{eq:fe} with $\kker(x, y) = \kker(y - x)$. In
particular, we consider the case in which we want to recover the density of a
random variable $X$ from observations of $Y=X+Z$ where $Z$ is a zero mean error variable independent of~$X$.
The deconvolution problem is widely studied, and recently a number of approaches based on normalizing flows \cite{dockhorn2020density}, deep learning \cite{andreassen2021scaffolding}, and generative adversarial networks \cite{datta2018unfolding} have been explored. 

We consider the Gaussian mixture model of \cite{ma2011indirect} (see also \cite{crucinio2020particle}) where, with a slight abuse of notation, we denote both a measure and its density w.r.t. the Lebesgue measure with the same symbol:
\begin{align}
& \pi(x)  =(1/3)\mathcal{N}(x;0.3,0.015^{2})+(2/3)\mathcal{N}(x;0.5,0.043^{2}),\\
& \kker(x, y)  =\mathcal{N}(y; x,0.045^{2}),\\
& \mu(y)  =(1/3)\mathcal{N}(y;0.3,0.045^{2}+0.015^{2})+(2/3)\mathcal{N}(y;0.5,0.045^{2}+0.043^{2}),
\end{align}
and compare our method with the deconvolution kernel density estimators with plug-in
bandwidth (DKDEpi) \cite{delaigle2004practical}, a class of estimators for deconvolution problems which given a sample from $\mu$ produce a kernel
density estimator for $\pi$, and SMC-EMS, a particle method to solve~\eqref{eq:fe} which achieves state-of-the-art performance.

We check the accuracy of the reconstructions through the integrated square error
\begin{align}
  \label{eq:ise}
  \textstyle{
  \ise(\hat{\pi})=\int_{\rset} \{ \pi(x) - \hat{\pi}(x)\}^2 \rmd x,
  }
\end{align}
with $\hat{\pi}$ an estimator of $\pi$. Even if the analytic
form of $\mu$ is known, for SMC-EMS and~\eqref{eq:kde} we draw a sample of size
$M=10^3$ and draw without replacement $m=\min(N, M)$ samples at each iteration;
for DKDE we draw a sample of size $N$ for each $N$.

\begin{figure}[t]
\centering
\begin{tikzpicture}[every node/.append style={font=\normalsize}]
\node (img1) {\includegraphics[width=0.7\textwidth]{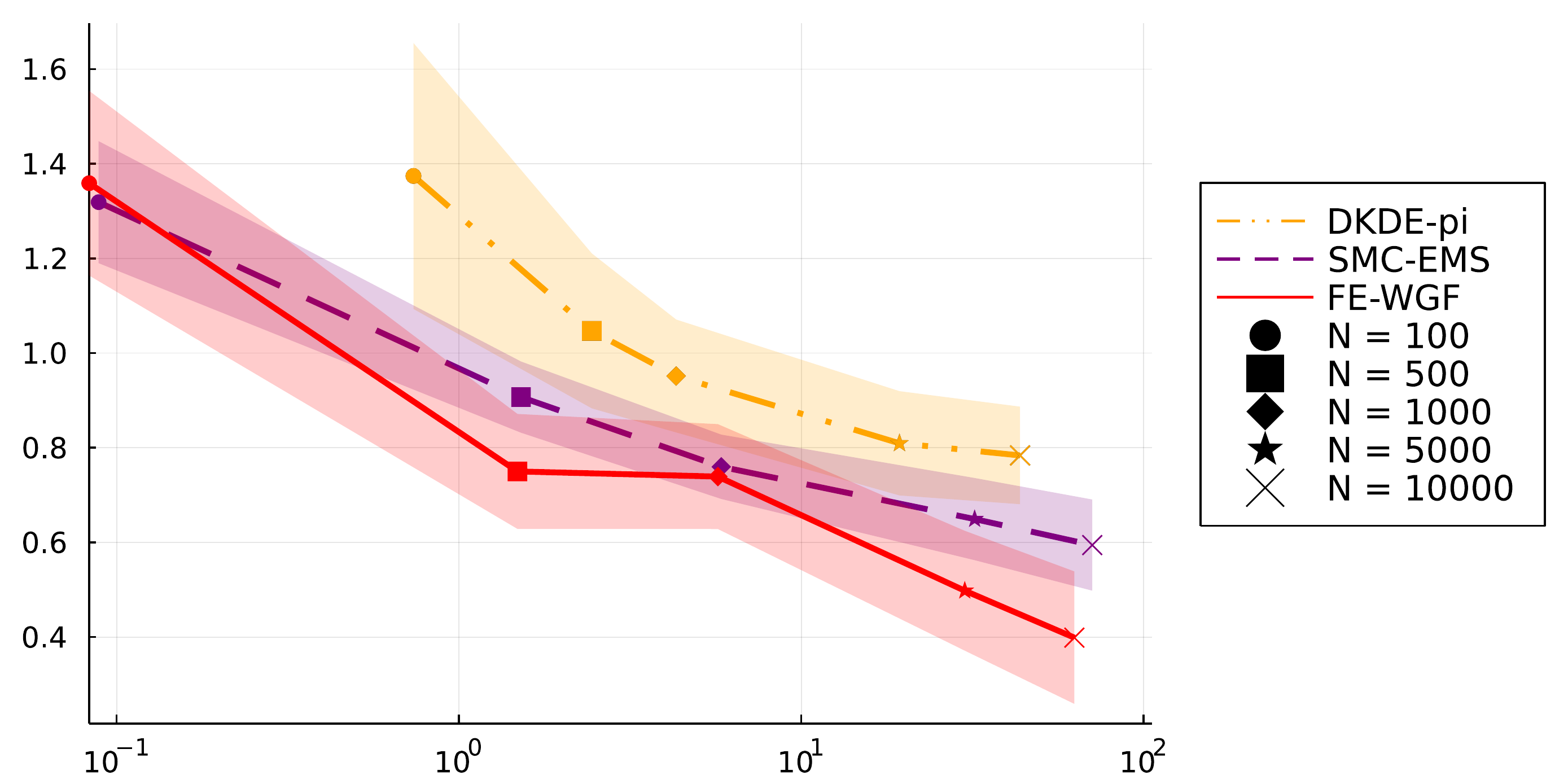}};
\node[below=of img1, node distance = 0, yshift = 1cm] (label1) {Runtime / s};
\node[left=of img1, node distance = 0, rotate=90, anchor = center, yshift = -0.8cm] {$\ise(\hat{\pi})$};
  \end{tikzpicture}
  \caption{Average accuracy and runtime for FE-WGF, SMC-EMS and DKDE with number of particles $N$
    ranging between $10^2$ and $10^4$. The shaded regions represent a interval of two standard deviations over 100 repetitions centred at the average $\ise$.}
\label{fig:mixture_smcems_vs_wgf}
\end{figure}
Both SMC-EMS and the estimator~\eqref{eq:kde} require specification of a number of parameters: we fix the number of time steps for both algorithms to $n_T = 100$ as we observed that convergence occurs in less than 100 iterations, the initial distribution is that of the samples from $Y$, since we expect the distribution of $Y$ to be close to that of $X$.  The regularization parameter $\alpha$ for FE-WGF and the smoothing parameter for SMC-EMS are chosen through cross validation, using~\eqref{eq:G_eta} as target functional for the former and the Kullback--Leibler divergence~\eqref{eq:kl} for the latter.
To implement FE-WGF we also need to specify the reference measure $\pi_0$, a Gaussian with mean and variance given by the empirical mean and variance of $Y$, and the time step $\gamma=10^{-3}$.

We consider different particle sizes (from $10^2$ to $10^4$) and compare reconstruction accuracy and total runtime of SMC-EMS, FE-WGF and DKDE
(Figure~\ref{fig:mixture_smcems_vs_wgf}). In particular, we compare the cost per
iteration since we run both algorithms for a fixed number of steps, however, we found that SMC-EMS and FE-WGF converge in a similar number of steps for this example. The
computational cost could be reduced by considering stopping criteria using
approximations of the functional $\Gun_{\alpha}^{\eta}$.  The performances of
FE-WGF and SMC-EMS are similar both in terms of runtime and in terms of accuracy,
with FE-WGF generally giving more accurate reconstructions for the same computational cost.

The DKDE-pi estimator has runtime comparable to that of SMC-EMS and FE-WGF, but poorer accuracy. Despite being better than SMC-EMS on average, the accuracy of FE-WGF has a slightly higher variance (about twice that of SMC-EMS), while DKDE-pi has still higher variance (about three times that of SMC-EMS).

As both SMC-EMS
and FE-WGF are regularized versions of the inconsistent maximum likelihood estimator for $\pi$, it is natural to compare the smoothness of the reconstructions. To characterize smoothness, we take 100 points $x_c$ in the support of $\pi$ and approximate (with 100 replicates) the mean squared error
\begin{equation}
  \textstyle{
    \mse(x_c) = \expe{\left( \pi(x_c) - \hat{\pi}(x_c)\right)^2}.
    }
\end{equation}
The $\mse$ measures locally the fit of the reconstructions to the solution $\pi$, which is known to be smooth, providing information on the (relative) smoothness of the reconstructions.

The distribution of the $\mse$ over the 100 points
(Figure~\ref{fig:mixture_smcems_vs_wgf_smoothness}) shows that
SMC-EMS and FE-WGF can achieve considerably lower $\mse$ than DKDE-pi, and therefore that the smoothness of the reconstructions provided by these two methods is closer to that of the true
density $\pi$.

\begin{figure}[t]
\centering
\resizebox{0.8\textwidth}{!}{%
\begin{tikzpicture}[every node/.append style={font=\normalsize}]
\node (img1) {\includegraphics[width=0.8\textwidth]{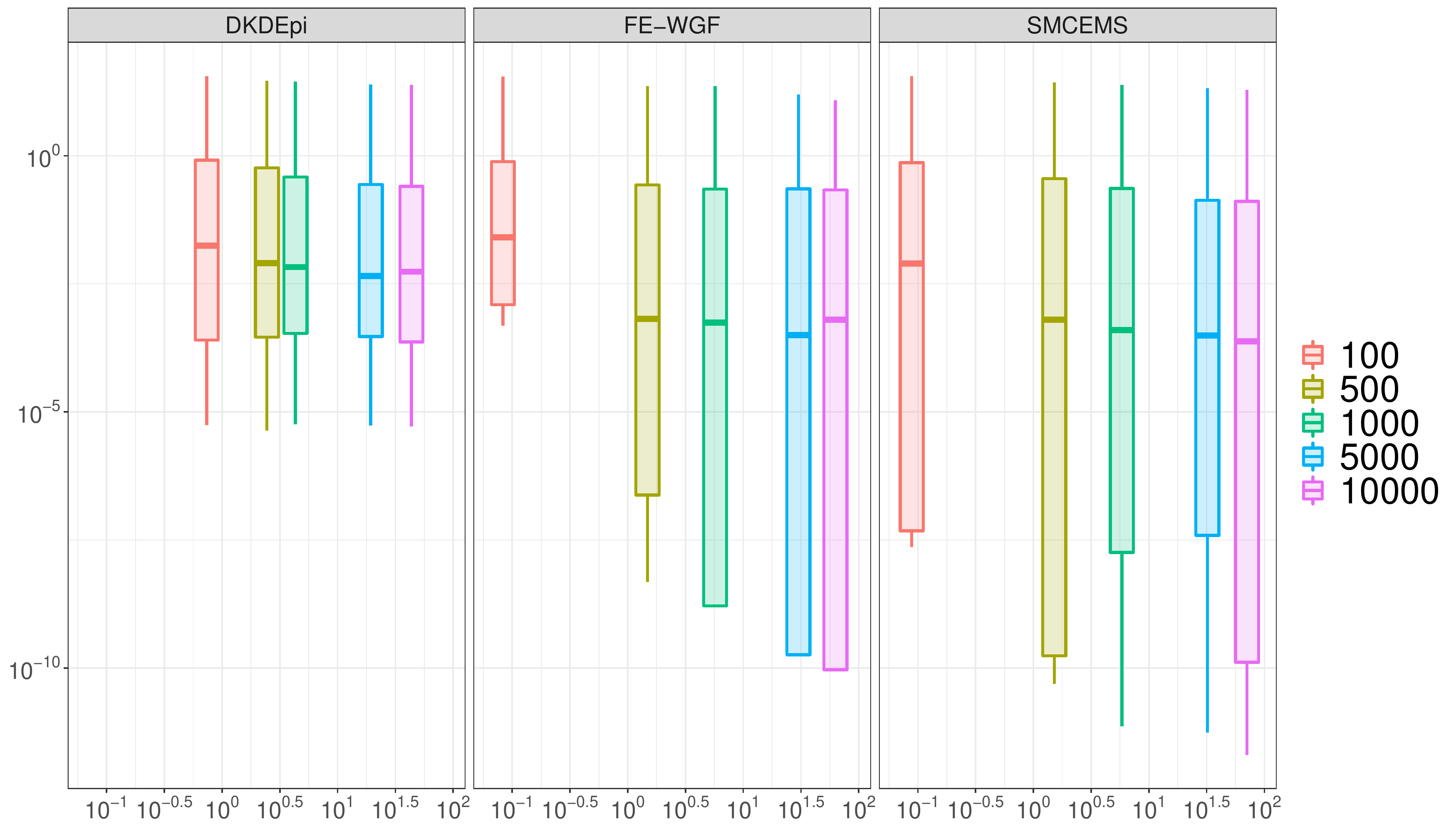}};
\node[left=of img1, node distance = 0, rotate=90, anchor = center, yshift = -1cm] {$\widehat{\mse}(x_c)$};
\node[below=of img1, node distance = 0, yshift = 1cm] (label1) {Runtime / s};
\end{tikzpicture}
}
\caption{
Distribution of $\mse$ as a function of runtime (in $\log$ seconds) for FE-WGF, SMC-EMS and DKDE. The number of particles $N$ ranges between 100 and 10,000.}
\label{fig:mixture_smcems_vs_wgf_smoothness}
\end{figure}

\subsection{Epidemiology}
\label{sec:epidem}
In epidemiology, Fredholm integral equations link the unknown infection
incidence curve $\pi$ to the number of reported cases, deaths or
hospitalizations. This is a particular instance of the deconvolution problem, in
which the kernel $\kker(x, y)= \kker(y-x)$ describes the delay distribution
between time of infection and time of death or hospitalization.  Deconvolution
techniques have been used to infer $\pi$ in the case of HIV
\cite{becker1991method} and influenza \cite{goldstein2009reconstructing} and
have been recently applied to estimate the incidence curve of COVID-19
\cite{marschner2020back, wang2020bayesian, chau2020construction,
  miller2020statistical}.

As an example, we take the spread of the pandemic influenza in the city of Philadelphia between the end of September and the beginning of October 1918 \cite{goldstein2009reconstructing}. The count of daily deaths and the distribution of delay between infection and death are available through the \texttt{R} package \texttt{incidental} \cite{miller2020statistical}.

To obtain a parametric form for $\kker$ we fit a mixture of Gaussians to the delay data using the expectation maximization algorithm (\texttt{normalmixEM} function in \texttt{R}; \cite{mixtools})
\begin{align}
\label{eq:k_epidem}
\kker(x, y) = 0.595\mathcal{N}(y - x; 8.63, 2.56^2) + 0.405\mathcal{N}(y - x; 15.24, 5.39^2).
\end{align}
Although this choice assigns $\approx 10^{-3}$ mass to the negative reals, we
found that a mixture of Gaussians fits the observed delay distribution better
than other commonly used distributions (\eg \  Gamma, log-normal, see
\cite{obadia2012r0}) and also benefits from a bounded
derivative and therefore satisfies~\Cref{assum:general_kker} and \Cref{assum:pi0}.

We compare the reconstructions obtained using FE-WGF with the
robust incidence deconvolution estimator (RIDE) of \cite{miller2020statistical}, the Richardson-Lucy (RL) deconvolution described in
\cite{goldstein2009reconstructing}, an iterative algorithm popular in the image
processing literature which aims at minimising the Kullback--Leibler
divergence~\eqref{eq:kl}, and SMC-EMS. To benchmark the performances of the four methods we consider synthetic data obtained simulating from the slow decay incidence curve in \cite{miller2020statistical}
\begin{align}
\label{eq:incidence}
\pi(x)\propto
\left\{
\begin{array}{ll}
\exp(-0.05(8-x)^2) & 0\leq x\leq 8, \\
\exp(-0.001(x-8)^2)&  8\leq x\leq 100 \\
\end{array},\right.
\end{align}
whose shape is similar to that of the incidence of the 1918 pandemic influenza, rescaled to have 5,000 infections distributed over 100 time steps.

We consider both the case in which
the model is well-specified, and each case on the incidence curve is randomly
propagated forward $s$ days according to $\kker$, and the case in which data are
noisy and the delay model is misspecified: every sixth and seventh day a uniform
random number between $0.3$ and $0.5$ is drawn and that proportion of cases are
recorded two days later (e.g. cases from day six move to day eight).  This approximates reporting delays for testing and death records
\cite{miller2020statistical}.

\begin{figure}
\centering
\begin{tikzpicture}[every node/.append style={font=\normalsize}]
\node (img1) {\includegraphics[width=0.4\textwidth]{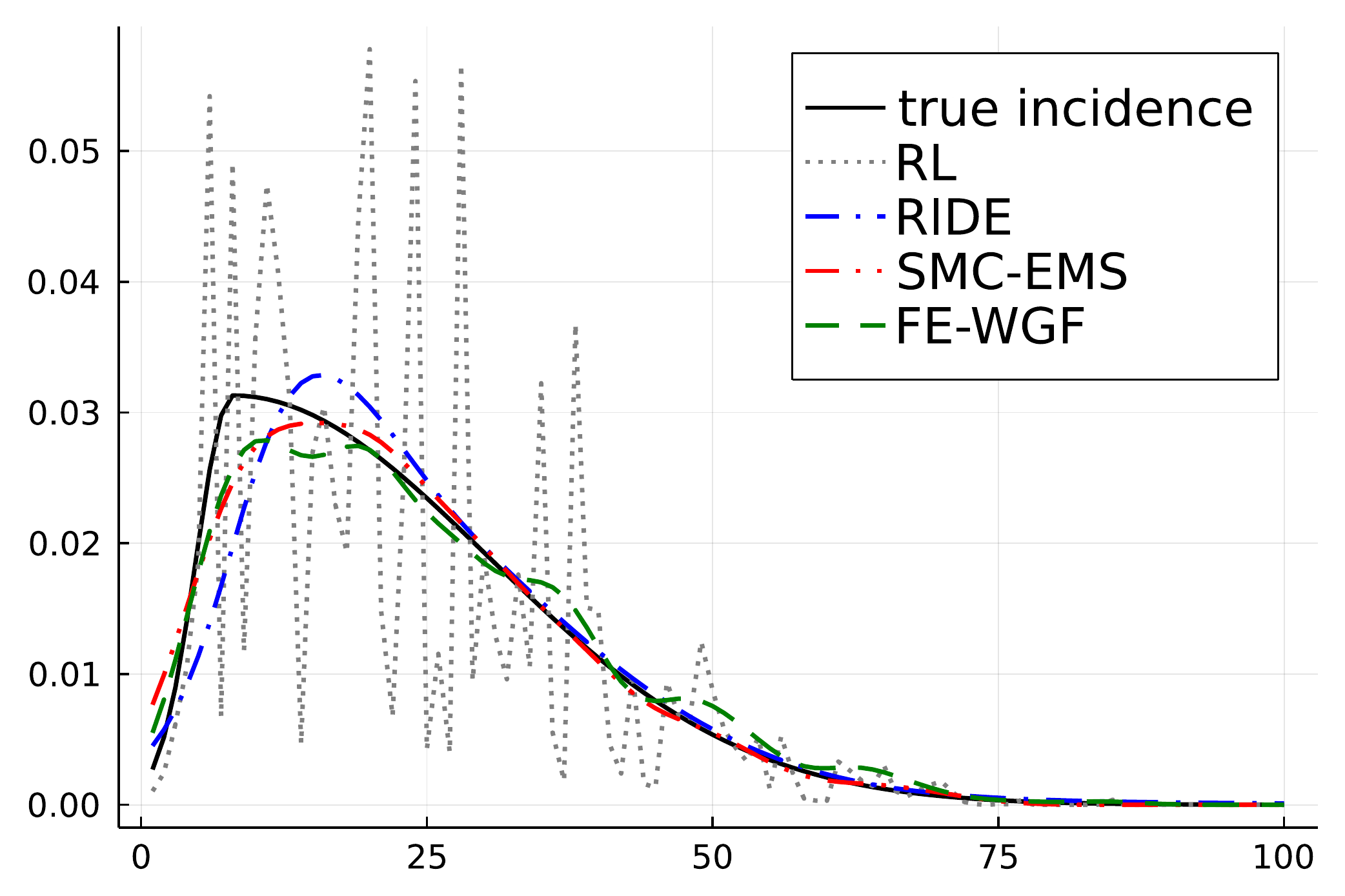}};
\node[below=of img1, node distance = 0, yshift = 1cm] (label1) {Time / days};
  \node[left=of img1, node distance = 0, rotate=90, anchor = center, yshift = -0.8cm] {Incidence};
\node[right=of img1, node distance = 0, xshift = -0.5cm] (img2) {\includegraphics[width=0.4\textwidth]{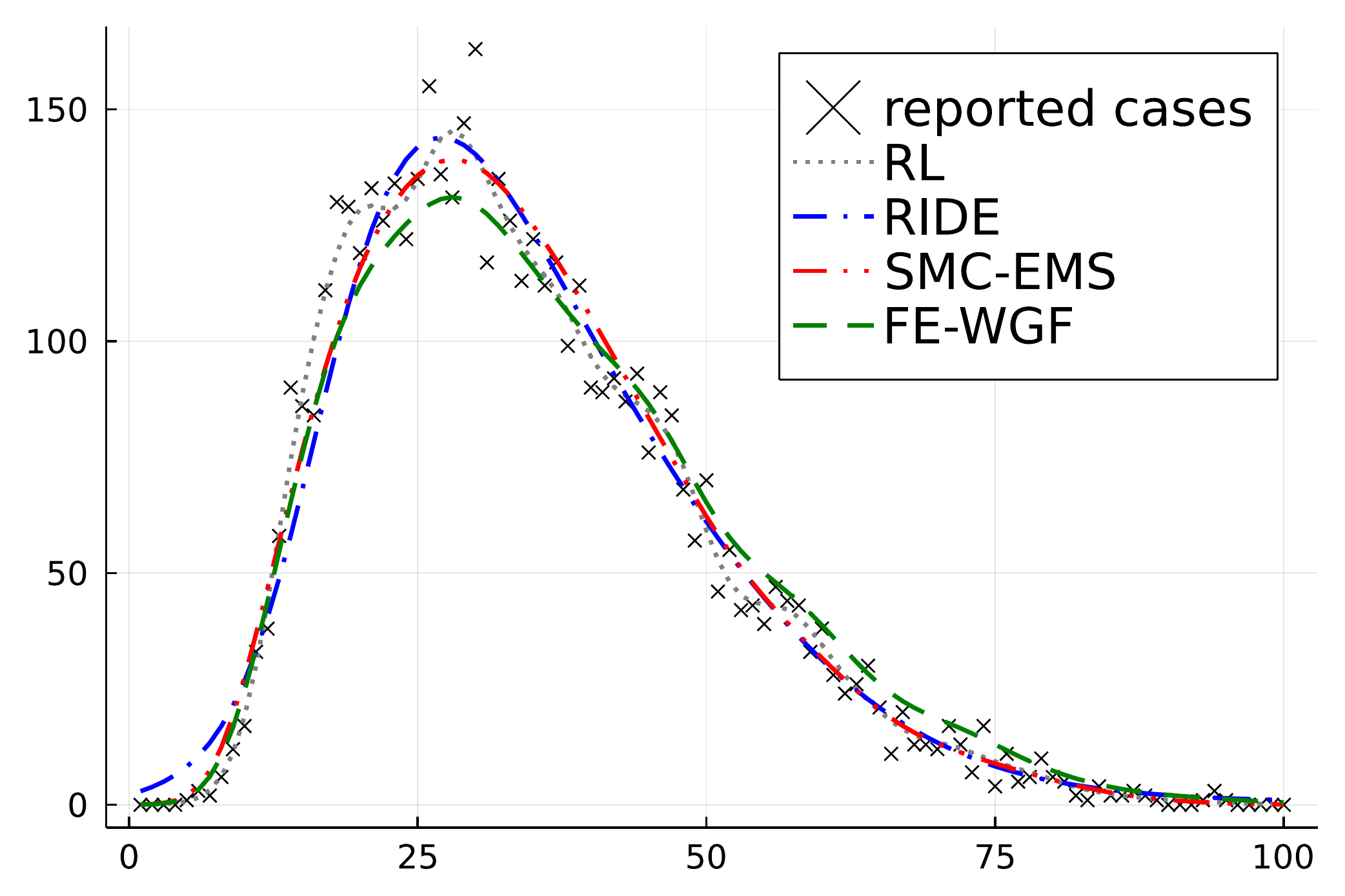}};
\node[below=of img2, node distance = 0, yshift = 1cm] {Time / days};
  \node[left=of img2, node distance = 0, rotate=90, anchor = center, yshift = -0.8cm] {Cases};
  \end{tikzpicture}
\caption{Example fit of the reconstructions of the synthetic incidence curve~\eqref{eq:incidence} and of the corresponding reconstruction of the number of cases. 
}
\label{fig:epidemics}
\end{figure}
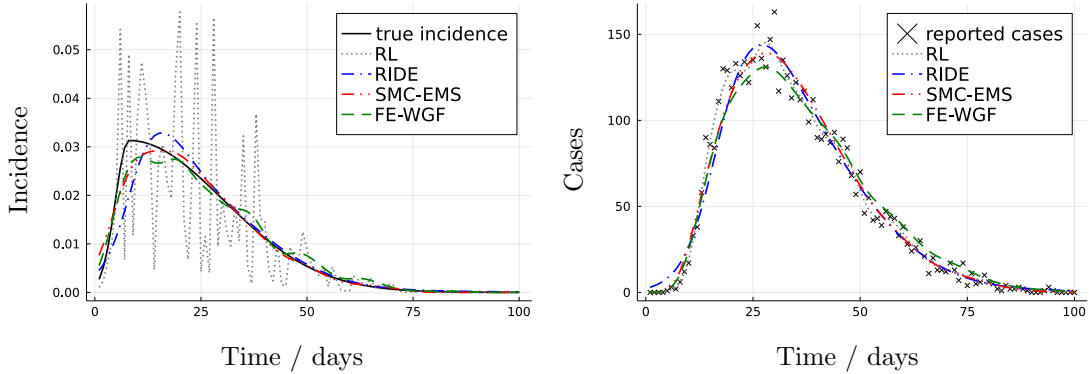

\begin{table}
\centering
\footnotesize{
\begin{tabular}{l|ccc|ccc}
 & \multicolumn{3}{c}{Well-specified} & \multicolumn{3}{c}{Misspecified}\\
\hline\noalign{\smallskip}
Method &$\ise(\hat{\pi})$ & $\ise(\mu^{rec})$ & runtime (s) & $\ise(\hat{\pi})$ & $\ise(\mu^{rec})$ & runtime (s)\\
\hline\noalign{\smallskip}
RL & $7.7\cdot10^{-3}$ & $\mathbf{1.5\cdot10^{-4}}$&$\mathbf{<1}$ & $7.9\cdot10^{-3}$ & $\mathbf{1.5\cdot10^{-4}}$&$\mathbf{<1}$\\
RIDE & $9.0\cdot10^{-4}$ & $3.4\cdot10^{-4}$&58& $1.0\cdot10^{-3}$ &$3.4\cdot10^{-4}$ &58\\
SMC-EMS & $3.3\cdot10^{-4}$ & $2.5\cdot10^{-4}$&3& $3.7\cdot10^{-4}$ & $2.5\cdot10^{-4}$ &3\\
  FE-WGF & $\mathbf{2.7\cdot10^{-4}}$ & $2.5\cdot10^{-4}$&96& $\mathbf{3.1\cdot10^{-4}}$ & $2.5\cdot10^{-4}$ &95\\
\end{tabular}
}
\caption{Comparison of reconstructions of the synthetic incidence curve~\eqref{eq:incidence}, the average reconstruction accuracy and fit to data over 100 repetitions are reported.}
\label{tab:sim_epidem}
\end{table}

Since the average delay between infection and death is
9 days \cite{goldstein2009reconstructing}, we set the reference measure $\pi_0$
to be a Gaussian distribution with mean equal to the mean of $\mu$ ($\approx 34.6$ days) shifted
backwards by 9 and variance equal to the variance of $\mu$ ($\approx 228$ days). 
Estimates of the unknown delay distribution are usually obtained by previous studies or from similar diseases (see, e.g. \cite{becker1991method, chau2020construction} and references therein), in this case, influenza-like infections.
Following similar
considerations, we set the initial distribution to be the death curve shifted
back by 9 days for both the RL algorithm, SMC-EMS and FE-WGF.
For FE-WGF we take $N=500$,
$\gamma=10^{-1}$ and iterate for $3000$ time steps (necessary to achieve convergence).

As discussed in Section~\ref{sec:choice} we set $\eta=0$
precision, we pick $\alpha=1\times 10^{-3}$ and the parameter $\varepsilon=2\times 10^{-4}$ for SMC-EMS by cross-validation over $L=5$
dataset for which the delay distribution $\kker$ is misspecified and use this
value for both the well-specified and the misspecified case. To approximate
$\mu$ we draw $m=N=500$ samples without repetition from the $M=5000$ observed
cases at each iteration in \Cref{alg:wgf_algoo}. The RIDE estimator does not require specification of any parameter while
for RL and SMC-EMS we set the maximum number of iterations to 100.

The estimators' quality is evaluated using the integrated squared
error~\eqref{eq:ise} and the fit to the observed number of cases measured
through the $\ise$ between the observed cases and the reconstructed
cases given by reconvolution of the estimates of $\pi$ with $\kker$
\begin{align}
  \textstyle{
\mu^{rec}(y) = \int_{\rset} \kker(y-x)\rmd\hat{\pi}(x).}
\end{align}

The RL algorithm gives the best fit the the observed death counts and has the lowest runtime ($<1$ second), however, the reconstruction of the incidence curve is not smooth and sensitive to noise \cite{miller2020statistical}. 
RIDE, SMC-EMS and FE-WGF address the lack of smoothness
introducing regularization and provide smoother reconstructions (Figure~\ref{fig:epidemics}) which are considerably closer to the incidence curve both in the well-specified and in the misspecified setting (Table~\ref{tab:sim_epidem}) with FE-WGF achieving the best fit to $\pi$. The runtime of RIDE
and FE-WGF are comparable, while SMC-EMS is cheaper as convergence is achieved in only 100 steps, and the reconstructions do not change considerably after that.
This is not unexpected,
as the reconstructions provided by RIDE, SMC-EMS and FE-WGF are a regularized version of the measure $\pi$ minimizing the Kullback--Leibler divergence~\eqref{eq:kl}, while that given by RL minimizes~\eqref{eq:kl} with no regularization and inherits the well-known inconsistency of the maximum
likelihood estimator in the infinite dimensional setting
\cite{laird1978nonparametric}.

Comparing SMC-EMS and FE-WGF, we find that the latter always achieves better accuracy, but the increase in accuracy is more pronounced for the misspecified model. In fact, even if SMC-EMS and FE-WGF both aim at minimizing a regularized Kullback--Leibler divergence~\eqref{eq:kl},
the latter minimizes~\eqref{eq:minimisation} which allows us to impose that the reconstructions are close (in the Kullback--Leibler sense) to the reference measure $\pi_0$, which for this example is particularly informative, since the corresponding term in the drift~\eqref{eq:b} pushes the particles towards the mode of the distribution of the number of cases shifted back by 9 days. On the contrary, SMC-EMS does not allow us to exploit this information, and $\pi_0$ is only used to initialize the particle system for SMC-EMS.

\subsection{Computed Tomography}
\label{sec:ct}

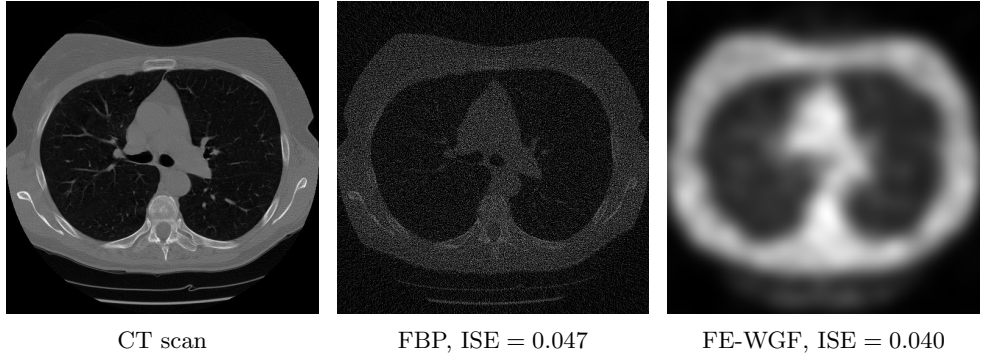
\begin{figure}
\centering
\begin{tikzpicture}[every node/.append style={font=\small}]
\node (img1) {\includegraphics[width=0.25\textwidth]{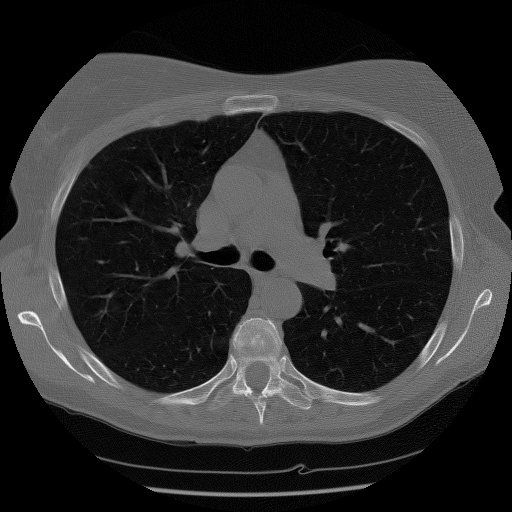}};
\node[below=of img1, yshift=1cm] {CT scan};
\node[right=of img1, xshift=-1cm] (img2) {\includegraphics[width=0.25\textwidth]{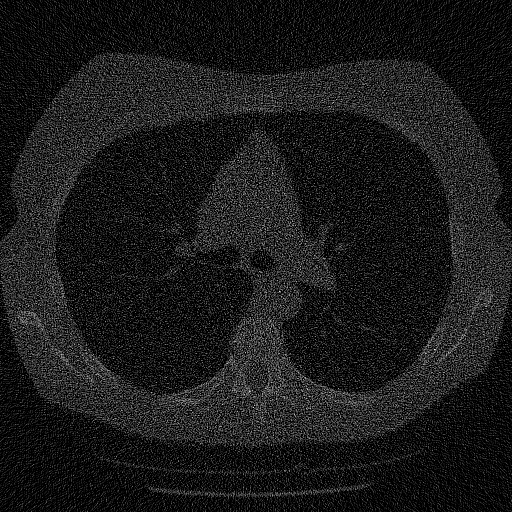}};
\node[below=of img2, yshift=1cm] {FBP, $\ise = 0.047$};
\node[right=of img2, xshift=-1cm] (img3) {\includegraphics[width=0.25\textwidth]{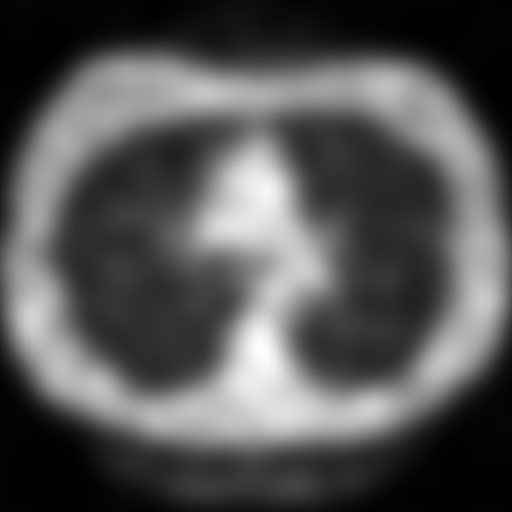}};
\node[below=of img3, yshift=1cm] {FE-WGF, $\ise = 0.040$};
\end{tikzpicture}
\caption{Reconstruction of a lung CT scan via FE-WGF and FBP. FBP provides reconstructions which preserve sharp edges but present speckle noise, while the reconstructions obtained with FE-WGF are smooth but with blurry edges.} 
\label{fig:ct}
\end{figure}

Fredholm integral equations of the first kind find wide application in medical
imaging \cite{webb2017introduction}; in this context they model the
reconstruction of cross-section images of the organ of interest from the noisy
measurements provided by positron emission tomography (PET) and CT scanners.

CT scanners provide noisy measurements by mapping each point of the organ's
cross-section $\pi$ onto its radial
projection $\mu$ defined for
angles $\phi$ between $0$ and $2\uppi$ and depths $\xi\in\rset$ with 0
corresponding to the centre of the scanner.  The distribution of the radial
projections $\mu$ is linked to the cross-section image of the organ
of interest $\pi$ through the Radon transform
\cite{radon1986determination}, \ie \ for any $\phi \in \ccintLigne{0, 2\uppi}$ and $\xi \in \rset$
\begin{align}
  \label{eq:radon}
  \textstyle{
  \mu(\phi, \xi) = \int_{\rset} \pi(\xi\cos \phi-t\sin\phi, \xi\sin \phi+t\cos\phi)\rmd t,
  }
\end{align}
(with a slight abuse of notation we denote both a measure and its density w.r.t. the Lebesgue measure with the same symbol) where the right hand side is the line integral along the line with equation $x\cos\phi +y \sin\phi =\xi$.
We model the alignment between the projections onto $(\phi, \xi)$ and the corresponding location $(x, y)$ in the reference image in~\eqref{eq:radon} using a Gaussian distribution with small variance, renormalized to integrate to 1 for fixed $(x,y)\in\rset^2$ and obtain the following Fredholm integral equation~\eqref{eq:fe}  
\begin{align}
  \textstyle{
\mu(\phi, \xi)
  = \int_{\rset^2}(2\uppi \sigma^2)^{-1} \exp[(x\cos\phi +y \sin\phi -\xi)^2/ (2\sigma^2)] \pi(x, y)\rmd x\rmd y,
  }
\end{align}
We test FE-WGF on a $512\times 512$ pixels lung CT scan from
the LIDC-IDRI database \cite{armato2011lung} (in particular, scan LIDC\_IDRI\_0683\_1\_048). The data image is obtained using the
ASTRA toolbox \cite{van2015astra} by projecting the CT scan at 512 equally
spaced angles $\phi$ in $[0, 2\uppi]$ and at 729 depths, and then corrupting the
obtained projections with Poisson noise.  We compare
the reconstructions given by~FE-WGF with those obtained by
filtered back projection (FBP), one of the most common methods for analytic image reconstruction \cite{tong2010image}, available in the ASTRA toolbox \cite{van2015astra}.

The reference measure is a Gaussian distribution with mean corresponding
to the centre of the image and variance $s^2\Id$ with $s >0$ large enough to
approach a uniform distribution over the image (in our case $s^2=0.35^2$), this
guarantees that~\Cref{assum:lip_U} is
satisfied. The number of particles $N = 10^4$ and the time step
$\gamma=10^{-3}$ are chosen to achieve high enough resolution, the
regularization parameter $\alpha=7\times 10^{-3}$ is selected by cross validation
over $L=5$ replicates.  Convergence is measured by numerically approximating
$\Gun_\alpha^\eta$ and occurs in $200$ steps.

The reconstructions provided by FBP are less robust to noise in the data, which
reflects in a higher $\ise$, on the contrary FE-WGF is more robust to noise and
achieves lower $\ise$. However, the reconstructions obtained with FE-WGF rarely result in sharp edges due to the inherent features of the
estimator~\eqref{eq:kde} (\ie\ the diffusive behaviour of~\eqref{eq:particle}
and the kernel density estimator).
\subsection{Scaling with dimension}
\label{sec:hd}

To explore the scaling with the dimension $d$ of the support of $\pi$ we take the generalization of the Gaussian mixture model in \Cref{sec:dd} described in \cite{crucinio2020particle} (with a slight abuse of notation we denote both a measure and its density with the same symbol)
\begin{align}
\label{eq:mixture_hd}
&\pi(x) = (1/3)\mathcal{N}(x;0.3\cdot\mathbf{1}_{d}, 0.07^2 \Id_d) + (2/3)\mathcal{N}(x; 0.7\cdot\mathbf{1}_{d}, 0.1^2\Id_d),\\
&\kker(x, y)  = \mathcal{N}(y; x, 0.15^2 \Id_d),\\
&\mu(y)  =  (1/3)\mathcal{N}(y;0.3 \cdot\mathbf{1}_{p}, (0.07^2+0.15^2)\Id_p) \\
&\qquad\qquad+ (2/3)\mathcal{N}(y; 0.7\cdot\mathbf{1}_{p}, (0.1^2+0.15^2) \Id_p),
\end{align}
where $p=d$ and
$\mathbf{1}_{d}, \Id_{d}$ denote the unit function in $\rset^{d}$ and the
$d \times d$ identity matrix, respectively.

We compare the reconstructions given by FE-WGF with those obtained
with the one-step-late expectation maximization (OSL-EM) algorithm
\cite{green1990use}, an iterative algorithm for maximum penalized likelihood
estimation, and SMC-EMS \cite{crucinio2020particle}. 
First, we focus on the
comparison with OSL-EM since this algorithm can be implemented to minimize $\Gun_\alpha^{\eta}$ in~\eqref{eq:G_eta} as does FE-WGF. In particular, FE-WGF is a stochastic approach to the problem of
minimizing~\eqref{eq:minimisation}, while OSL-EM relies on a deterministic
discretization of $\pi$.  Secondly, we compare the proposed approach with
SMC-EMS \cite{crucinio2020particle}.

\subsubsection{Comparison with OSL-EM}

\begin{figure}[t]
\centering
\resizebox{0.6\textwidth}{!}{%
\begin{tikzpicture}[every node/.append style={font=\normalsize}]
\node (img1) {\includegraphics[width=0.6\textwidth]{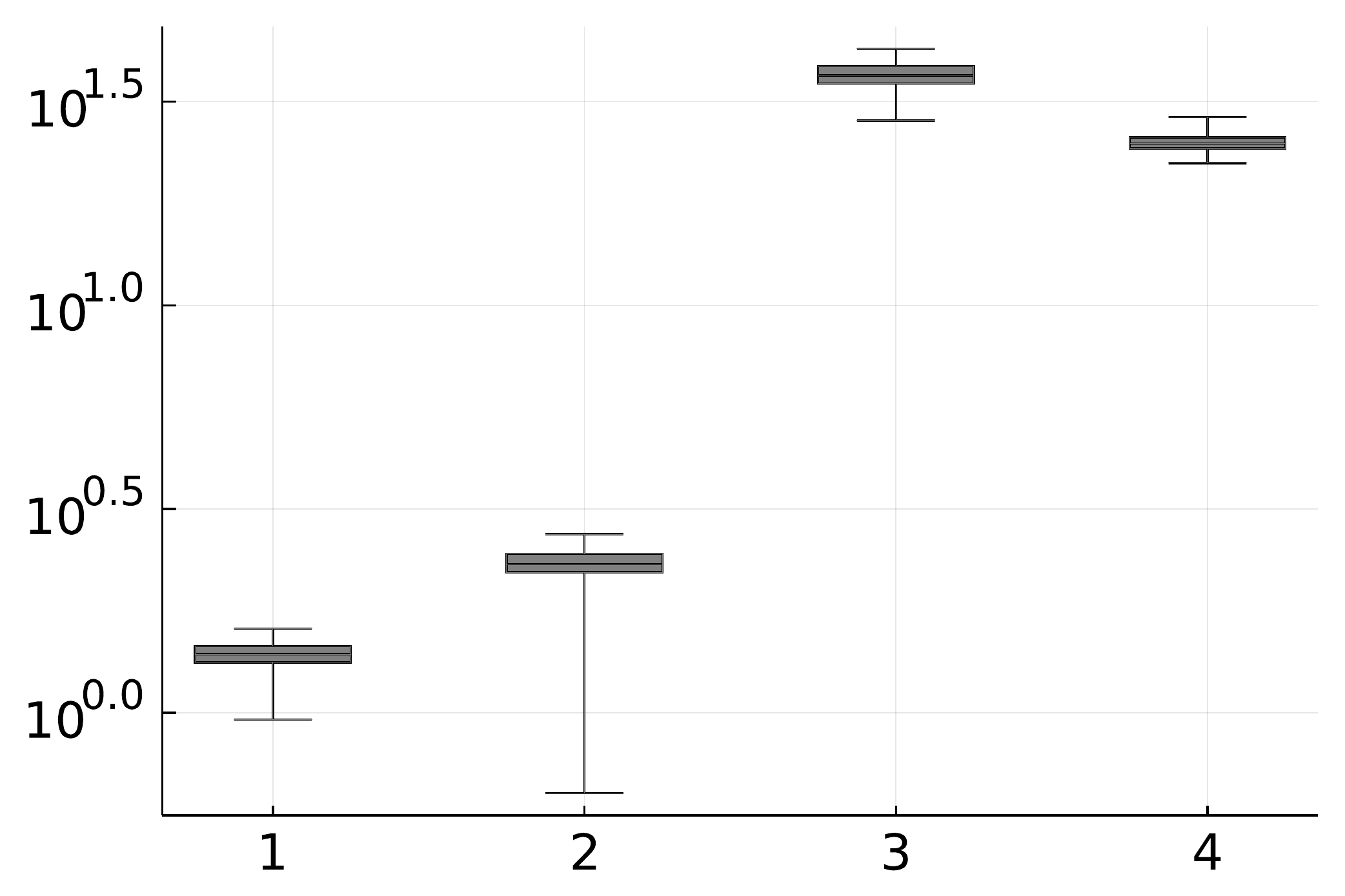}};
\node[below=of img1, node distance = 0, yshift = 1cm] {$d$};
  \node[left=of img1, node distance = 0, rotate=90, anchor = center, yshift = -0.5cm] {Gain in $\ise(\hat{\pi}_1)$};
\end{tikzpicture}
}
\caption{Distribution of $\ise$ ratios ($\ise$ of OSL-EM divided by $\ise$ of FE-WGF) over 100 repetitions. The runtime of FE-WGF is 1.3 times that of OSL-EM on average, while the average gain ranges from $1.3$ for $d=1$ to $\approx 30$ for $d=3, 4$.}
\label{fig:wgf_vs_em}
 \end{figure}
As briefly discussed in \Cref{sec:variants}, the first term in the definition of
the functional $\Gun_{\alpha}$ in~\eqref{eq:G} (with $\mu$ replaced by $\mu^M$)
is the likelihood associated with the incomplete-data $\{y^{k,M}\}_{k=1}^M$. Minimizing
$\Gun_{\alpha}$ is equivalent to maximizing a penalised likelihood with penalty
$\KL{\pi}{\pi_0}$. The one-step-late expectation maximization (OSL-EM) algorithm
\cite{green1990use} is a variant of the well-known EM algorithm
\cite{dempster1977maximum} to maximize any penalized likelihood.

In order to implement OSL-EM we need to discretize the support of $\pi$, which
essentially requires $\pi$ to have known compact support. This is not the case
for the measure $\pi$ considered in this example, however,
$\vert \int_{[0,1]^d}\pi(x)\rmd x-1\vert<10^{-2}$ for $d\leq 10$; therefore we
define a discretization grid for OSL-EM over the $d$-dimensional hypercube
$[0, 1]^d$ (similar considerations apply to $\mu$).  We consider $B \in \nset$
equally spaced bins for each dimension and obtain the discretizations of $\pi$
and $\mu$, $\{\pi_b\}_{b=1}^{B^d}, \{\mu_b\}_{b=1}^{B^d}$,  setting
$\pi_b = \pi(x_b)$, where $x_b$ is the centre of the $b$-th
bin. The same discretization mechanism is used for $\kker$ to obtain
$\{\kker_{b, c}\}_{b,c=1}^{B^d}$, where $\kker_{b, c}=\kker(x_b, x_c)$.  The
one-dimensional OSL-EM iteration is given for any $n \in \nset$ by
\begin{equation}
  \textstyle{
    \pi_{b}^{(n+1)} = \pi_{b}^{(n)}/( 1 + \alpha(1+\log \pi_{b}^{(n)} - \log\pi_{0,b})) \sum_{c=1}^B \{\mu_c \kker_{b,c}/ \sum_{d=1}^B \pi_d^{(n)}\kker_{d,c}\},
    }
  \end{equation}
  where $\pi_{0,b}=\pi_0(x_b)$, for $b\in \{1, \dots,
  B^d\}$. Contrary to the estimator provided by FE-WGF, the
  reconstructions provided by OSL-EM are deterministic and might not converge to
  the unique minimizer of $\Gun^\eta_{\alpha}$. 
  For the mixture model in~\eqref{eq:mixture_hd}, we observe
  empirical convergence to a unique fixed point for reasonable values of
  $\alpha$, \eg\ $\alpha<1$.

To compare OSL-EM and FE-WGF we fix $\alpha=0.01$  (as both algorithms aim at minimizing $\Gun_{\alpha}^\eta$ the value of $\alpha=0.01$ is chosen to compare them in a regime of practical interest, with the aim of obtaining good approximations of $\pi$), set the reference measure $\pi_0$ to be a Gaussian distribution with mean $(0.5, \dots, 0.5)$ and
covariance matrix $\sigma_0^2 \Id_d$ with $\sigma_0^2=0.25^2$ and use
$\pi_0$ as initial condition for~FE-WGF. With this choice of
parameters, convergence (assessed empirically by approximating the value of $\Gun_{\alpha}^\eta$) occurs within 50 steps for both algorithms and $d\in \{1, \dots, 4\}$.  The number of particles is set to $N=10^4$, and we obtain the number of bins per dimension, $B$, so that $B^d\approx N$, and the two algorithms have similar runtime. This corresponds to coarser discretizations as $d$
  increases which give low resolution reconstructions, for instance, for $d=4$
  each dimension is discretized into $B=10$ bins while for $d=1$ we have
  $B=N=10^4$ bins; however, such fine discretization schemes are impractical for
  $d\gg 1$ since the number of points at which $\pi$ needs to be approximated
  grows exponentially with $d$.

  To assess the quality of the reconstructions we use the integrated squared
  error~\eqref{eq:ise} of the first marginal, $\ise(\hat{\pi}_1)$, where
  $\hat{\pi}$ is the estimated density. Since OSL-EM is a deterministic algorithm
  we take its $\ise$ as reference and investigate the gains obtained by using
  FE-WGF for dimension $d=1, 2, 3, 4$, see \Cref{fig:wgf_vs_em}. For $d=1, 2$ the
  gain with respect to OSL-EM is moderate, but for $d=3, 4$ the reconstructions obtained with   FE-WGF are at least 10 times more accurate that those obtained by
  OSL-EM with an average runtime which is only 1.3 times that of OSL-EM.

\subsubsection{Comparison with SMC-EMS}
We now turn to comparing the reconstructions obtained with~FE-WGF with those
given by SMC-EMS \cite{crucinio2020particle}. Both methods approximate $\pi$
through a family of interacting particles from which smooth estimates can be
obtained via standard kernel density estimation as in~\eqref{eq:kde}.
Since typical applications of Fredholm
integral equations are low-dimensional, mostly one to three dimensional with
some notable exceptions \cite{signoroni2019deep}, we focus on $d$ up to 10.

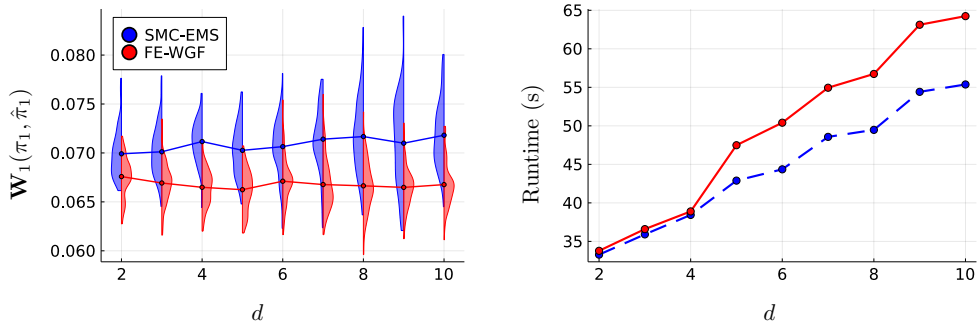
\begin{figure}
\centering
\resizebox{0.8\textwidth}{!}{%
\begin{tikzpicture}[every node/.append style={font=\normalsize}]
\node (img1) {\includegraphics[width=0.4\textwidth]{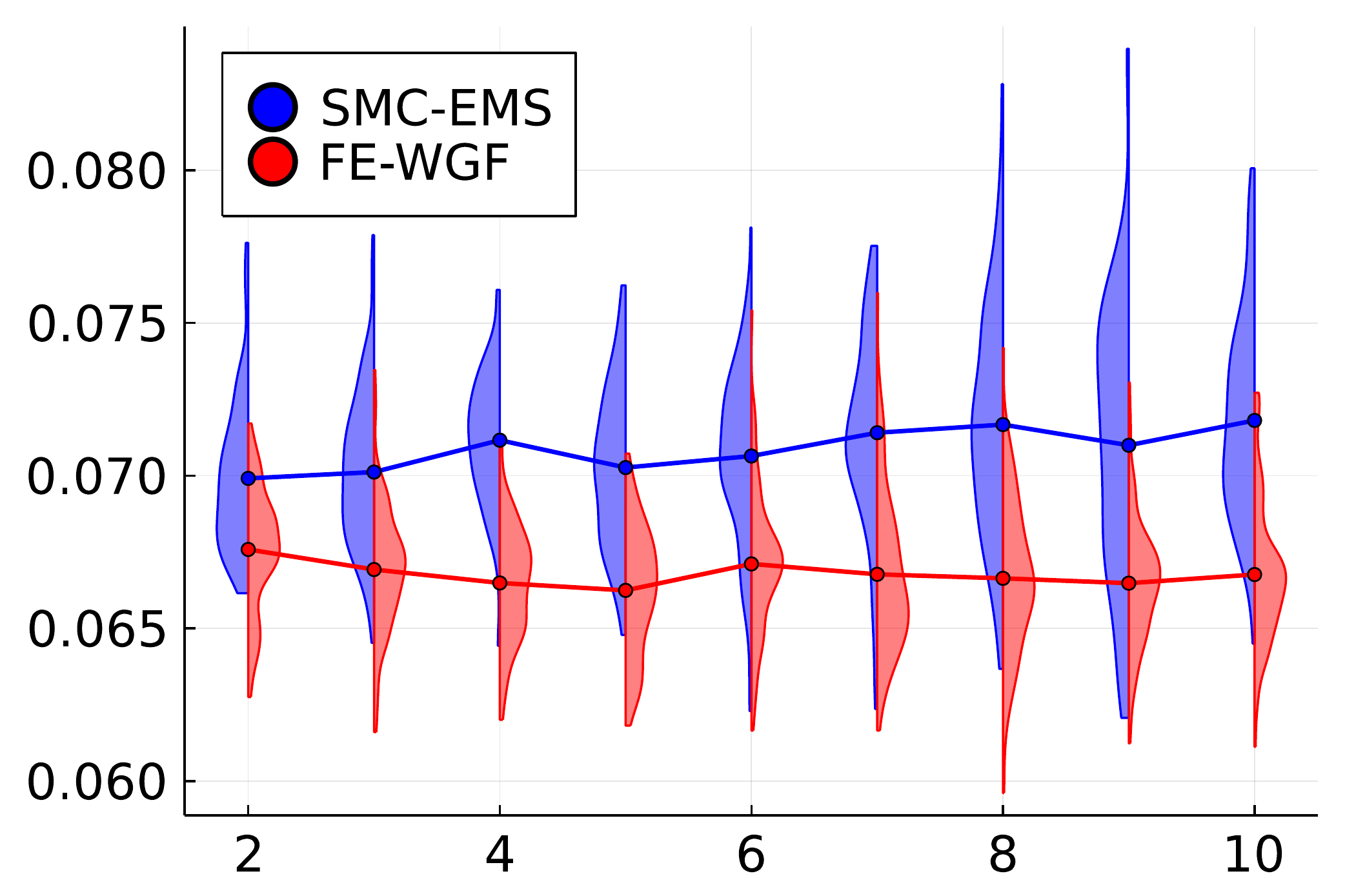}};
\node[below=of img1, node distance = 0, yshift = 1cm] {$d$};
  \node[left=of img1, node distance = 0, rotate=90, anchor = center, yshift = -0.8cm] {$\wassersteinD[1](\pi_1, \hat{\pi}_1)$};
 \node[right=of img1] (img2) {\includegraphics[width=0.4\textwidth]{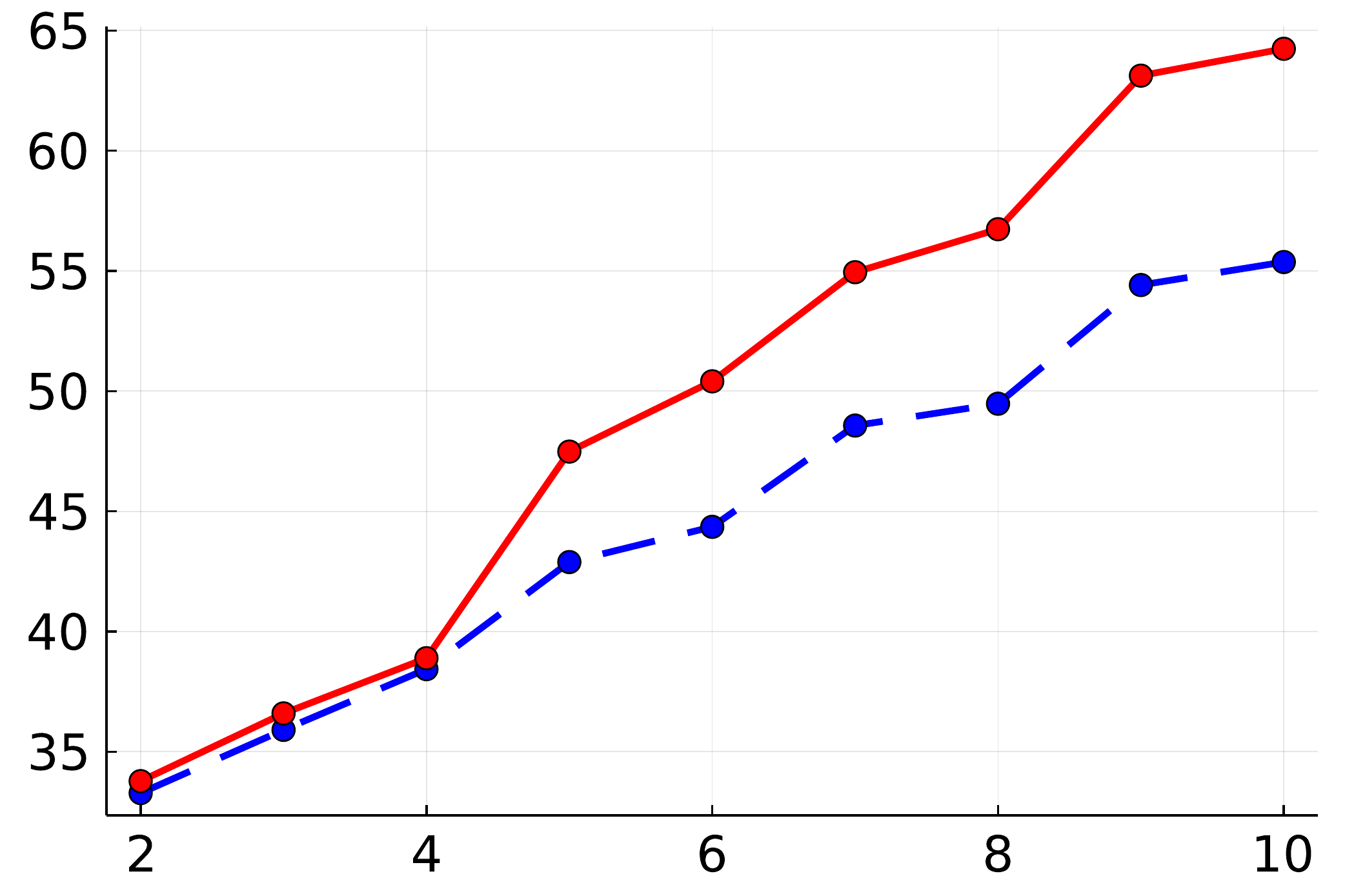}};
\node[below=of img2, node distance = 0, yshift = 1cm] {$d$};
  \node[left=of img2, node distance = 0, rotate=90, anchor = center, yshift = -0.8cm] {Runtime (s)};
\end{tikzpicture}
}
\caption{Reconstruction accuracy over 100 of SMC-EMS and FE-WGF for the Gaussian mixture model in~\eqref{eq:mixture_hd}. The left panel shows the distribution of the Wasserstein-1 distance between the first marginal of $\pi$ and its particle approximations for $d=2, \dots, 10$, the solid lines highlight the average accuracy. The right panel reports average runtime.}
\label{fig:wgf_vs_smc}
\end{figure}

To compare the two algorithms we consider $N=10^3$ and we assume that $\mu$ is
not known but a sample of size $M=10^6$ is available and resample without
replacement $m=N=10^3$ times from it at each iteration of both SMC-EMS and FE-WGF.
SMC-EMS converges quickly to an approximation of $\pi$
\cite{crucinio2020particle}, for the model in~\eqref{eq:mixture_hd} we observe
convergence in less than 50 steps (measured by numerically approximating the
Kullback--Leibler divergence~\eqref{eq:kl} using the samples from $\mu$ and $\hat{\pi}$, the estimated distribution). 
For FE-WGF we set $\gamma = 10^{-2}$ and select the
reference measure $\pi_0$ to be a Gaussian distribution with mean
$(0.5, \dots, 0.5)$ and covariance matrix $\sigma_0^2 \Id_d$ with $\sigma_0^2=0.25^2$; 
with this choice of parameters convergence occurs in less than 50 iterations also for FE-WGF.
To specify the smoothing parameter for SMC-EMS and $\alpha$ for FE-WGF we use the cross-validation approach in \Cref{sec:choice} with $L=10$.

Since kernel density estimators do not perform well for $d\gg 1$ \cite{silverman1986density}, we do not build the kernel density estimator~\eqref{eq:kde} 
and compare the reconstructions obtained by SMC-EMS and FE-WGF using the particle population given by each method as an approximation for $\pi$ (as justified by \Cref{prop:particles_min}), and compute the Wasserstein-1 distance between the first marginal $\pi_1$ and its particle approximation (Figure~\ref{fig:wgf_vs_smc}).

SMC-EMS and FE-WGF give similar results for all $d$, with the average runtime of
SMC-EMS being slightly lower than that of FE-WGF for $d=2, 3, 4$ and about
$15\%$ lower for $d\geq 5$. 
The increase in runtime for $d\geq 5$ is caused by the computation of the norm of the drift for the tamed Euler scheme~\eqref{eq:particle_M} and is likely due to the increased algebraic complexity of this operation with dimension.
In terms of accuracy, FE-WGF performs slightly
better than SMC-EMS with $\wassersteinD[1]$ distance $5\%$ smaller on average
for FE-WGF.

\section{Discussion}

We consider a probabilistic approach to the solution of Fredholm integral equations of the first kind~\eqref{eq:fe} by introducing the minimization problem~\eqref{eq:minimisation} and its surrogate~\eqref{eq:G_eta}.
Under mild smoothness assumptions on the kernel $\Kker$ and the reference measure $\pi_0$, we show that the surrogate functional admits a unique minimizer which is stable w.r.t. $\mu$.
Leveraging a Wasserstein gradient flow construction, we introduce an interacting particle system~\eqref{eq:particle} which approximates the minimizer of the surrogate functional and is ergodic for any finite $N$.
Combining standard propagation of chaos results and strong convergence of tamed Euler--Maruyama schemes we obtain the bound~\eqref{eq:guideline} which provides a practical guideline for the selection of the time discretization step $\gamma$, the number of particles $N$ and the number of samples $m$ from $\mu$ to use at each iteration.

In Section~\ref{sec:ex} we compare FE-WGF with both problem specific and state-of-the-art methods to solve~\eqref{eq:fe}; specifically, we compare our algorithm with SMC-EMS, a particle method to solve~\eqref{eq:fe} which also aims at minimizing a regularized Kullback--Leibler divergence. This particle method has been shown to achieve state-of-the-art performances in a number of settings and to scale with the dimension of the state space $d$ better than standard discretization based algorithms \cite{crucinio2020particle}.
Our results show that FE-WGF has always comparable or superior performances to problem specific algorithms and outperforms SMC-EMS in high dimensional setting.
Moreover, FE-WGF allows the introduction of a reference measure $\pi_0$, which can be particularly beneficial when some features of the measure $\pi$ to be reconstructed are known (as, for instance, in epidemiology applications as that in Section~\ref{sec:epidem}).

From a theoretical perspective, \Cref{prop:convergence_minimum} and the convergence properties in \Cref{sec:part-appr-mcke} guarantee that the particle system~\eqref{eq:particle} approximates the unique minimizer of~\eqref{eq:G_eta}; an equivalent uniqueness property has not yet been established for SMC-EMS.
However, the results presented in Section~\ref{sec:ex} also provide some additional
confidence in the use of SMC-EMS, demonstrating empirically that it recovers
comparable solutions to the provably convergent method developed in this paper
in a range of settings.

Moreover, while both algorithms minimize a regularized Kullback--Leibler divergence, and for small values of the regularization parameters the reconstruction they provide are similar, the methods through which they approximate the minimizer present some notable differences. FE-WGF is based on a numerical approximation of the MKVSDE~\eqref{eq:nonlinearSDE} and therefore the associated particle system has a diffusive behaviour, while SMC-EMS uses a resampling mechanism to eliminate the particles which are not in the support of $\pi$ and replicate those close to high-probability regions. The effect of this different algorithmic approach can be observed by considering different initial distributions for $\{\bfX_0^{k,N}\}_{k=1}^N$: SMC-EMS benefits from overly diffused initial distributions, since the resampling mechanism allows to quickly move the particles towards high probability regions, for concentrated distributions (e.g. Dirac delta) convergence still occurs but in a larger number of iterations \cite[Appendix E.1]{crucinio2020particle}. The convergence speed of FE-WGF is less influenced by the initial state $\{\bfX_0^{k,N}\}_{k=1}^N$ as shown in \Cref{app:ex}.
In contrast with SMC-EMS, FE-WGF requires specification of a reference measure $\pi_0$. While this might seem a limitation of our method, we found that in practice we often have sufficient knowledge about the problem at hand to build $\pi_0$. For instance, in image processing applications usually a large number of images needs to be processed and it is possible to build reference measures using previously processed images (e.g. \cite{kingma2010regularized}), while in epidemiology, one can exploit previous studies on similar diseases, since they often present incidence curves with the same characteristics (e.g. \cite{goldstein2009reconstructing}).
In fact, it is in this latter setting that we find that FE-WGF outperforms SMC-EMS in terms of accuracy (Section~\ref{sec:epidem}).

FE-WGF describes one possible numerical approximation of the particle system~\eqref{eq:particle}, however, we anticipate that more sophisticated time discretization schemes will prove beneficial in some settings. For instance, in our experiments we observed that in the first few iterations the drift component in~\eqref{eq:particle} has larger values, and pushes the particles towards high $\pi$-probability regions; after the first iterations, the magnitude of the drift component decreases and the shape of the approximated solution is refined. In terms of the functional $\Gun_\alpha^\eta$, the first few iterations correspond to a steepest decrease in $\Gun_{\alpha}^\eta$, which then stabilizes around the minimum (see e.g., \Cref{app:ex}). To speed up convergence to the minimizer, one could consider having a larger time step $\gamma$ for earlier iterations, and then gradually reduce $\gamma$ to refine the approximation. More sophisticated adaptive strategies for MKVSDE have been recently studied in \cite{reisinger2022adaptive}.

In conclusion, FE-WGF offers a general recipe to obtain regularized solutions to Fredholm integral equations of the first kind, which can be refined exploiting specific knowledge of the problem at hand; for example, choosing informative reference measures $\pi_0$ or adapting the time discretization step $\Delta t$ to the magnitude of the drift~\eqref{eq:b}.

\section*{Acknowledgments}
FRC, VDB, AD and AMJ acknowledge support from the EPSRC (grant \#  EP/R034710/1). AMJ acknowledges further support from the  EPSRC (grant \# EP/T004134/1) and the Lloyd's Register Foundation Programme on Data-Centric Engineering at the Alan Turing Institute. 
For the purpose of open access, the authors have applied a Creative Commons Attribution (CC BY) licence to any Author Accepted Manuscript version arising from this submission.

\medskip\noindent\textbf{Data access statement:} No new data was created during this research. 

\bibliographystyle{abbrv}  
\bibliography{wgf_biblio}

\begin{thebibliography}{100}

\bibitem{amato1991maximum}
U.~Amato and W.~Hughes.
\newblock Maximum entropy regularization of {Fredholm} integral equations of
  the first kind.
\newblock {\em Inverse Problems}, 7(6):793, 1991.

\bibitem{ambrosio2008gradient}
L.~Ambrosio, N.~Gigli, and G.~Savar{\'e}.
\newblock {\em Gradient flows: in metric spaces and in the space of probability
  measures}.
\newblock Springer Science \& Business Media, 2008.

\bibitem{andreassen2021scaffolding}
A.~Andreassen, P.~T. Komiske, E.~M. Metodiev, B.~Nachman, A.~Suresh, and
  J.~Thaler.
\newblock Scaffolding simulations with deep learning for high-dimensional
  deconvolution.
\newblock {\em ICLR 2021 SimDL Workshop}, 2021.

\bibitem{antonelli2002rate}
F.~Antonelli and A.~Kohatsu-Higa.
\newblock Rate of convergence of a particle method to the solution of the
  {McKean--Vlasov} equation.
\newblock {\em Annals of Applied Probability}, 12(2):423--476, 2002.

\bibitem{arbel2019maximum}
M.~Arbel, A.~Korba, A.~Salim, and A.~Gretton.
\newblock Maximum mean discrepancy gradient flow.
\newblock In {\em Advances in Neural Information Processing Systems}, pages
  6481--6491, 2019.

\bibitem{armato2011lung}
S.~G. Armato, G.~McLennan, L.~Bidaut, M.~F. McNitt-Gray, C.~R. Meyer, A.~P.
  Reeves, B.~Zhao, D.~R. Aberle, C.~I. Henschke, E.~A. Hoffman, et~al.
\newblock The lung image database consortium {(LIDC)} and image database
  resource initiative {(IDRI)}: a completed reference database of lung nodules
  on {CT} scans.
\newblock {\em Medical Physics}, 38(2):915--931, 2011.

\bibitem{aster2018parameter}
R.~C. Aster, B.~Borchers, and C.~H. Thurber.
\newblock {\em Parameter Estimation and Inverse Problems}.
\newblock Elsevier, 2018.

\bibitem{bao2020first}
J.~Bao, C.~Reisinger, P.~Ren, and W.~Stockinger.
\newblock First order convergence of {Milstein} schemes for {McKean-Vlasov}
  equations and interacting particle systems.
\newblock {\em Proceedings of the Royal Society A: Mathematical, Physical and
  Engineering Sciences}, 2020.

\bibitem{becker1991method}
N.~G. Becker, L.~F. Watson, and J.~B. Carlin.
\newblock A method of non-parametric back-projection and its application to
  {AIDS} data.
\newblock {\em Statistics in Medicine}, 10:1527--1542, 1991.

\bibitem{mixtools}
T.~Benaglia, D.~Chauveau, D.~R. Hunter, and D.~Young.
\newblock {mixtools}: An {R} package for analyzing finite mixture models.
\newblock {\em Journal of Statistical Software}, 32(6):1--29, 2009.

\bibitem{billingsley1995measure}
P.~Billingsley.
\newblock {\em Probability and Measure}.
\newblock John Wiley \& Sons, 1995.

\bibitem{bissiri2016general}
P.~G. Bissiri, C.~C. Holmes, and S.~G. Walker.
\newblock A general framework for updating belief distributions.
\newblock {\em Journal of the Royal Statistical Society: Series B (Statistical
  Methodology)}, 78(5):1103--1130, 2016.

\bibitem{bogachev2019convergence}
V.~I. Bogachev, M.~R{\"o}ckner, and S.~V. Shaposhnikov.
\newblock Convergence in variation of solutions of nonlinear
  {Fokker--Planck--Kolmogorov} equations to stationary measures.
\newblock {\em Journal of Functional Analysis}, 276(12):3681--3713, 2019.

\bibitem{bossy1997stochastic}
M.~Bossy and D.~Talay.
\newblock A stochastic particle method for the {McKean-Vlasov and the Burgers}
  equation.
\newblock {\em Mathematics of Computation}, 66(217):157--192, 1997.

\bibitem{brosse2019tamed}
N.~Brosse, A.~Durmus, {\'E}.~Moulines, and S.~Sabanis.
\newblock The tamed unadjusted {Langevin }algorithm.
\newblock {\em Stochastic Processes and Their Applications},
  129(10):3638--3663, 2019.

\bibitem{buckdahn2017mean}
R.~Buckdahn, J.~Li, S.~Peng, and C.~Rainer.
\newblock Mean-field stochastic differential equations and associated {PDEs}.
\newblock {\em Annals of Probability}, 45(2):824--878, 2017.

\bibitem{burger2019entropic}
M.~Burger, E.~Resmerita, and M.~Benning.
\newblock An entropic {Landweber} method for linear ill-posed problems.
\newblock {\em Inverse Problems}, 36(1):015009, 2019.

\bibitem{butkovsky2014ergodic}
O.~Butkovsky.
\newblock On ergodic properties of nonlinear {Markov} chains and stochastic
  {McKean--Vlasov} equations.
\newblock {\em Theory of Probability \& Its Applications}, 58(4):661--674,
  2014.

\bibitem{byrne2015algorithms}
C.~Byrne and P.~P. Eggermont.
\newblock {EM} algorithms.
\newblock In {\em Handbook of {M}athematical {M}ethods in {I}maging}, pages
  305--388. Springer, 2015.

\bibitem{cardaliaguet2010notes}
P.~Cardaliaguet.
\newblock Notes on mean field games.
\newblock Technical report, Coll\'ege de France, 2010.

\bibitem{Carmona2016}
R.~Carmona.
\newblock {\em Lectures on BSDEs, Stochastic Control, and Stochastic
  Differential Games with Financial Applications}.
\newblock SIAM, 2016.

\bibitem{carmona2013probabilistic}
R.~Carmona and F.~Delarue.
\newblock Probabilistic analysis of mean-field games.
\newblock {\em SIAM Journal on Control and Optimization}, 51(4):2705--2734,
  2013.

\bibitem{carmona2018probabilistic}
R.~Carmona and F.~Delarue.
\newblock {\em Probabilistic theory of mean field games with applications
  {I-II}}.
\newblock Springer, 2018.

\bibitem{chae2018convergence}
M.~Chae, R.~Martin, and S.~G. Walker.
\newblock Convergence of an iterative algorithm to the nonparametric {MLE} of a
  mixing distribution.
\newblock {\em Statistics \& Probability Letters}, 140:142--146, 2018.

\bibitem{chae2018algorithm}
M.~Chae, R.~Martin, and S.~G. Walker.
\newblock On an algorithm for solving {Fredholm} integrals of the first kind.
\newblock {\em Statistics \& Computing}, 29:645–654, 2018.

\bibitem{chau2020construction}
P.~H. Chau, W.~Y. Li, and P.~S. Yip.
\newblock Construction of the infection curve of local cases of {COVID-19 in
  Hong Kong} using back-projection.
\newblock {\em International Journal of Environmental Research and Public
  Health}, 17(18):6909, 2020.

\bibitem{chizat2022mean}
L.~Chizat.
\newblock {Mean-Field Langevin Dynamics: Exponential Convergence and
  Annealing}.
\newblock {\em Transactions on Machine Learning Research}, 2022.

\bibitem{clason2019regularization}
C.~Clason, B.~Kaltenbacher, and E.~Resmerita.
\newblock Regularization of ill-posed problems with non-negative solutions.
\newblock In {\em Splitting Algorithms, Modern Operator Theory, and
  Applications}, pages 113--135. Springer, 2020.

\bibitem{colton2012inverse}
D.~Colton and R.~Kress.
\newblock {\em Inverse acoustic and electromagnetic scattering theory},
  volume~93.
\newblock Springer, 2012.

\bibitem{crucinio2021some}
F.~R. Crucinio.
\newblock {\em Some Interacting Particle Methods with Non-Standard
  Interactions}.
\newblock PhD thesis, University of Warwick, 2021.

\bibitem{crucinio2020particle}
F.~R. Crucinio, A.~Doucet, and A.~M. Johansen.
\newblock A particle method for solving {Fredholm} equations of the first kind.
\newblock {\em Journal of the American Statistical Association},
  118(542):937--947, 2023.

\bibitem{csiszar1975divergence}
I.~Csisz\'{a}r.
\newblock {$I$}-divergence geometry of probability distributions and
  minimization problems.
\newblock {\em Annals of Probability}, 3:146--158, 1975.

\bibitem{dalmaso1993introduction}
G.~Dal~Maso.
\newblock {\em An introduction to {$\Gamma$}-convergence}, volume~8 of {\em
  Progress in Nonlinear Differential Equations and their Applications}.
\newblock Birkh\"{a}user Boston, Inc., Boston, MA, 1993.

\bibitem{datta2018unfolding}
K.~Datta, D.~Kar, and D.~Roy.
\newblock Unfolding with generative adversarial networks.
\newblock {\em arXiv preprint arXiv:1806.00433}, 2018.

\bibitem{debortoli2019convergence}
V.~De~Bortoli and A.~Durmus.
\newblock Convergence of diffusions and their discretizations: from continuous
  to discrete processes and back.
\newblock {\em arXiv preprint arXiv:1904.09808}, 2019.

\bibitem{delaigle2008alternative}
A.~Delaigle.
\newblock An alternative view of the deconvolution problem.
\newblock {\em Statistica Sinica}, pages 1025--1045, 2008.

\bibitem{delaigle2004practical}
A.~Delaigle and I.~Gijbels.
\newblock Practical bandwidth selection in deconvolution kernel density
  estimation.
\newblock {\em Computational Statistics \& Data Analysis}, 45(2):249--267,
  2004.

\bibitem{dempster1977maximum}
A.~P. Dempster, N.~M. Laird, and D.~B. Rubin.
\newblock Maximum likelihood from incomplete data via the {EM} algorithm.
\newblock {\em Journal of the Royal Statistical Society: Series B (Statistical
  Methodology)}, 39:2--38, 1977.

\bibitem{dockhorn2020density}
T.~Dockhorn, J.~A. Ritchie, Y.~Yu, and I.~Murray.
\newblock {Density Deconvolution with Normalizing Flows}.
\newblock {\em ICML Workshop on Invertible Neural Networks, Normalizing Flows,
  and Explicit Likelihood Models 2020}, 2020.

\bibitem{dupuis1997weak}
P.~Dupuis and R.~S. Ellis.
\newblock {\em A weak convergence approach to the theory of large deviations}.
\newblock Wiley Series in Probability and Statistics. John Wiley \& Sons, Inc.,
  New York, 1997.
\newblock A Wiley-Interscience Publication.

\bibitem{durmus2019analysis}
A.~Durmus, S.~Majewski, and B.~Miasojedow.
\newblock Analysis of {Langevin Monte Carlo} via convex optimization.
\newblock {\em The Journal of Machine Learning Research}, 20(1):2666--2711,
  2019.

\bibitem{eberle2016reflection}
A.~Eberle.
\newblock Reflection couplings and contraction rates for diffusions.
\newblock {\em Probability Theory \& Related Fields}, 166(3-4):851--886, 2016.

\bibitem{eggermont1993maximum}
P.~P. Eggermont.
\newblock Maximum entropy regularization for {Fredholm} integral equations of
  the first kind.
\newblock {\em SIAM Journal on Mathematical Analysis}, 24(6):1557--1576, 1993.

\bibitem{garbuno2020interacting}
A.~Garbuno-Inigo, F.~Hoffmann, W.~Li, and A.~M. Stuart.
\newblock Interacting {Langevin} diffusions: {Gradient} structure and ensemble
  {Kalman} sampler.
\newblock {\em SIAM Journal on Applied Dynamical Systems}, 19(1):412--441,
  2020.

\bibitem{goldstein2009reconstructing}
E.~Goldstein, J.~Dushoff, J.~Ma, J.~B. Plotkin, D.~J. Earn, and M.~Lipsitch.
\newblock Reconstructing influenza incidence by deconvolution of daily
  mortality time series.
\newblock {\em Proceedings of the National Academy of Sciences},
  106(51):21825--21829, 2009.

\bibitem{gostic2020practical}
K.~M. Gostic, L.~McGough, E.~B. Baskerville, S.~Abbott, K.~Joshi, C.~Tedijanto,
  R.~Kahn, R.~Niehus, J.~A. Hay, P.~M. De~Salazar, et~al.
\newblock Practical considerations for measuring the effective reproductive
  number, {Rt}.
\newblock {\em PLoS Computational Biology}, 16(12):e1008409, 2020.

\bibitem{gottlieb2000markov}
A.~D. Gottlieb.
\newblock Markov transitions and the propagation of chaos.
\newblock {\em arXiv preprint math/0001076}, 2000.

\bibitem{green1990use}
P.~J. Green.
\newblock On use of the {EM} for penalized likelihood estimation.
\newblock {\em Journal of the Royal Statistical Society: Series B (Statistical
  Methodology)}, 52(3):443--452, 1990.

\bibitem{groetsch1984theory}
C.~W. Groetsch.
\newblock {\em The theory of Tikhonov regularization for Fredholm equations}.
\newblock Pitman Advanced Publishing Program, 1984.

\bibitem{grunwald2012safe}
P.~Gr{\"u}nwald.
\newblock The safe {B}ayesian.
\newblock In N.~H. Bshouty, G.~Stoltz, N.~Vayatis, and T.~Zeugmann, editors,
  {\em Algorithmic Learning Theory}, pages 169--183, Berlin, Heidelberg, 2012.
  Springer Berlin Heidelberg.

\bibitem{hall2005nonparametric}
P.~Hall and J.~L. Horowitz.
\newblock Nonparametric methods for inference in the presence of instrumental
  variables.
\newblock {\em Annals of Statistics}, 33(6):2904--2929, 2005.

\bibitem{hammersley2021mckean}
W.~R. Hammersley, D.~{\v{S}}i{\v{s}}ka, and {\L}.~Szpruch.
\newblock {McKean--Vlasov SDEs under measure dependent Lyapunov conditions}.
\newblock In {\em Annales de l'Institut Henri Poincar{\'e}, Probabilit{\'e}s et
  Statistiques}, volume~57, pages 1032--1057. Institut Henri Poincar{\'e},
  2021.

\bibitem{hu2019mean}
K.~Hu, Z.~Ren, D.~{\v{S}}i{\v{s}}ka, and {\L}.~Szpruch.
\newblock Mean-field {Langevin} dynamics and energy landscape of neural
  networks.
\newblock In {\em Annales de l'Institut Henri Poincar{\'e}, Probabilit{\'e}s et
  Statistiques}, volume~57, pages 2043--2065. Institut Henri Poincar{\'e},
  2021.

\bibitem{hutzenthaler2015numerical}
M.~Hutzenthaler and A.~Jentzen.
\newblock {\em Numerical approximations of stochastic differential equations
  with non-globally Lipschitz continuous coefficients}.
\newblock American Mathematical Society, 2015.

\bibitem{hutzenthaler2012strong}
M.~Hutzenthaler, A.~Jentzen, and P.~E. Kloeden.
\newblock Strong convergence of an explicit numerical method for {SDEs with
  nonglobally Lipschitz} continuous coefficients.
\newblock {\em Annals of Applied Probability}, 22(4):1611--1641, 2012.

\bibitem{islam2020approximating}
M.~S. Islam and A.~Smith.
\newblock Approximating solutions of {Fredholm} integral equations via a
  general spline maximum entropy method.
\newblock {\em International Journal of Applied and Computational Mathematics},
  6(3), 2020.

\bibitem{iusem1994new}
A.~N. Iusem and B.~Svaiter.
\newblock A new smoothing-regularization approach for a maximum-likelihood
  estimation problem.
\newblock {\em Applied Mathematics and Optimization}, 29(3):225--241, 1994.

\bibitem{jaynes1957information}
E.~T. Jaynes.
\newblock Information theory and statistical mechanics.
\newblock {\em Physical Review}, 106(4):620--630, 1957.

\bibitem{jin2016solving}
C.~Jin and J.~Ding.
\newblock Solving {Fredholm} integral equations via a piecewise linear maximum
  entropy method.
\newblock {\em Journal of Computational and Applied Mathematics}, 304:130--137,
  2016.

\bibitem{jordan1998variational}
R.~Jordan, D.~Kinderlehrer, and F.~Otto.
\newblock The variational formulation of the {Fokker--Planck} equation.
\newblock {\em SIAM Journal on Mathematical Analysis}, 29(1):1--17, 1998.

\bibitem{karatzas1991brownian}
I.~Karatzas and S.~E. Shreve.
\newblock {\em Brownian motion and stochastic calculus}, volume 113 of {\em
  Graduate Texts in Mathematics}.
\newblock Springer-Verlag, New York, second edition, 1991.

\bibitem{kingma2010regularized}
D.~P. Kingma and Y.~LeCun.
\newblock Regularized estimation of image statistics by score matching.
\newblock In {\em Proceedings of the 23rd International Conference on Neural
  Information Processing Systems-Volume 1}, pages 1126--1134, 2010.

\bibitem{kloeden1992stochastic}
P.~E. Kloeden and E.~Platen.
\newblock {\em Numerical solution of stochastic differential equations}.
\newblock Springer, 1992.

\bibitem{kopec1993application}
S.~Kope{\'c}.
\newblock On application of maxent to solving {Fredholm} integral equations.
\newblock In {\em Maximum Entropy and Bayesian Methods}, pages 63--66.
  Springer, 1993.

\bibitem{kress2014linear}
R.~Kress.
\newblock {\em Linear Integral Equations}, volume~82 of {\em Applied
  Mathematical Sciences}.
\newblock Springer, 2014.

\bibitem{laird1978nonparametric}
N.~Laird.
\newblock Nonparametric maximum likelihood estimation of a mixing distribution.
\newblock {\em Journal of the American Statistical Association},
  73(364):805--811, 1978.

\bibitem{liu2019image}
J.~Liu, Y.~Sun, X.~Xu, and U.~S. Kamilov.
\newblock Image restoration using total variation regularized deep image prior.
\newblock In {\em International Conference on Acoustics, Speech and Signal
  Processing (ICASSP)}, pages 7715--7719. IEEE, 2019.

\bibitem{liu2017stein}
Q.~Liu.
\newblock Stein variational gradient descent as gradient flow.
\newblock {\em Advances in Neural Information Processing Systems}, 30, 2017.

\bibitem{ma2011indirect}
J.~Ma.
\newblock Indirect density estimation using the iterative {Bayes} algorithm.
\newblock {\em Computational Statistics \& Data Analysis}, 55(3):1180--1195,
  2011.

\bibitem{malrieu2001logarithmic}
F.~Malrieu.
\newblock Logarithmic {Sobolev} inequalities for some nonlinear {PDEs}.
\newblock {\em Stochastic Processes and Their Applications}, 95(1):109--132,
  2001.

\bibitem{marschner2020back}
I.~Marschner.
\newblock Back-projection of {COVID-19} diagnosis counts to assess infection
  incidence and control measures: analysis of {Australian} data.
\newblock {\em Epidemiology and Infection}, 148:e97, 2020.

\bibitem{mckean1966class}
H.~McKean.
\newblock A class of {M}arkov processes associated with nonlinear parabolic
  equations.
\newblock {\em Proceedings of the National Academy of Sciences of the United
  States of America}, 56(6):1907--1911, 1966.

\bibitem{mead1986approximate}
L.~R. Mead.
\newblock Approximate solution of {Fredholm} integral equations by the
  maximum-entropy method.
\newblock {\em Journal of Mathematical Physics}, 27(12):2903--2907, 1986.

\bibitem{meleard1996asymptotic}
S.~M{\'e}l{\'e}ard.
\newblock Asymptotic behaviour of some interacting particle systems;
  {McKean-Vlasov and Boltzmann models}.
\newblock In {\em Probabilistic models for nonlinear partial differential
  equations}, pages 42--95. Springer, 1996.

\bibitem{meleard1987propagation}
S.~M{\'e}l{\'e}ard and S.~Roelly-Coppoletta.
\newblock A propagation of chaos result for a system of particles with moderate
  interaction.
\newblock {\em Stochastic Processes and Their Applications}, 26:317--332, 1987.

\bibitem{miao2018identifying}
W.~Miao, Z.~Geng, and E.~J. Tchetgen~Tchetgen.
\newblock Identifying causal effects with proxy variables of an unmeasured
  confounder.
\newblock {\em Biometrika}, 105(4):987--993, 2018.

\bibitem{miller2020statistical}
A.~C. Miller, L.~Hannah, J.~Futoma, N.~J. Foti, E.~B. Fox, A.~D'Amour,
  M.~Sandler, R.~A. Saurous, and J.~Lewnard.
\newblock Statistical deconvolution for inference of infection time series.
\newblock Technical Report~4, 2022.

\bibitem{molina1992bayesian}
R.~Molina, A.~del Olmo, J.~Perea, and B.~D. Ripley.
\newblock Bayesian deconvolution in optical astronomy.
\newblock {\em The Astronomical Journal}, 103:666--675, 1992.

\bibitem{molina1993using}
R.~Molina and B.~D. Ripley.
\newblock Using spatial models as priors in astronomical image analysis.
\newblock {\em Journal of Applied Statistics}, 20(5-6):281--298, 1993.

\bibitem{moore1920reciprocal}
E.~H. Moore.
\newblock On the reciprocal of the general algebraic matrix.
\newblock {\em Bulletin of the American Mathematical Society}, 26:394--395,
  1920.

\bibitem{multhei1987iterative}
H.~M{\"u}lthei, B.~Schorr, and W.~T{\"o}rnig.
\newblock On an iterative method for a class of integral equations of the first
  kind.
\newblock {\em Mathematical Methods in the Applied Sciences}, 9(1):137--168,
  1987.

\bibitem{nitanda2022convex}
A.~Nitanda, D.~Wu, and T.~Suzuki.
\newblock {Convex analysis of the mean field Langevin dynamics}.
\newblock In {\em International Conference on Artificial Intelligence and
  Statistics}, pages 9741--9757. PMLR, 2022.

\bibitem{obadia2012r0}
T.~Obadia, R.~Haneef, and P.-Y. Bo{\"e}lle.
\newblock The{ R0 }package: a toolbox to estimate reproduction numbers for
  epidemic outbreaks.
\newblock {\em BMC Medical Informatics and Decision Making}, 12(1):1--9, 2012.

\bibitem{oelschlager1984martingale}
K.~Oelschlager.
\newblock A martingale approach to the law of large numbers for weakly
  interacting stochastic processes.
\newblock {\em Annals of Probability}, pages 458--479, 1984.

\bibitem{otto2001geometry}
F.~Otto.
\newblock The geometry of dissipative evolution equations: the porous medium
  equation.
\newblock {\em Communications in Partial Differential Equations}, 2001.

\bibitem{pensky2017minimax}
M.~Pensky et~al.
\newblock Minimax theory of estimation of linear functionals of the
  deconvolution density with or without sparsity.
\newblock {\em Annals of Statistics}, 45(4):1516--1541, 2017.

\bibitem{radon1986determination}
J.~Radon.
\newblock On the determination of functions from their integral values along
  certain manifolds.
\newblock {\em IEEE Transactions on Medical Imaging}, 5(4):170--176, 1986.

\bibitem{reisinger2022adaptive}
C.~Reisinger and W.~Stockinger.
\newblock An adaptive {Euler--Maruyama scheme for McKean--Vlasov SDEs with
  super-linear growth and application to the mean-field FitzHugh--Nagumo
  model}.
\newblock {\em Journal of Computational and Applied Mathematics}, 400:113725,
  2022.

\bibitem{resmerita2007joint}
E.~Resmerita and R.~S. Anderssen.
\newblock Joint additive {Kullback--Leibler} residual minimization and
  regularization for linear inverse problems.
\newblock {\em Mathematical Methods in the Applied Sciences},
  30(13):1527--1544, 2007.

\bibitem{revuz1999continuous}
D.~Revuz and M.~Yor.
\newblock {\em Continuous martingales and {B}rownian motion}, volume 293 of
  {\em Grundlehren der Mathematischen Wissenschaften [Fundamental Principles of
  Mathematical Sciences]}.
\newblock Springer-Verlag, Berlin, third edition, 1999.

\bibitem{santambrogio2017euclidean}
F.~Santambrogio.
\newblock $\{$Euclidean, metric, and Wasserstein$\}$ gradient flows: an
  overview.
\newblock {\em Bulletin of Mathematical Sciences}, 7(1):87--154, 2017.

\bibitem{signoroni2019deep}
A.~Signoroni, M.~Savardi, A.~Baronio, and S.~Benini.
\newblock Deep learning meets hyperspectral image analysis: a multidisciplinary
  review.
\newblock {\em Journal of Imaging}, 5(5):52, 2019.

\bibitem{silverman1986density}
B.~W. Silverman.
\newblock {\em Density {E}stimation for {S}tatistics and {D}ata {A}nalysis},
  volume~26 of {\em Monographs on Statistics and Applied Probability}.
\newblock Chapman \& Hall, 1986.

\bibitem{silverman1990smoothed}
B.~W. Silverman, M.~C. Jones, J.~D. Wilson, and D.~W. Nychka.
\newblock A smoothed {EM} approach to indirect estimation problems, with
  particular, reference to stereology and emission tomography.
\newblock {\em Journal of the Royal Statistical Society: Series B (Statistical
  Methodology)}, 52(2):271--324, 1990.

\bibitem{snyder1992deblurring}
D.~L. Snyder, T.~J. Schulz, and J.~A. O'Sullivan.
\newblock Deblurring subject to nonnegativity constraints.
\newblock {\em IEEE Transactions on Signal Processing}, 40(5):1143--1150, 1992.

\bibitem{sznitman1991topics}
A.-S. Sznitman.
\newblock Topics in propagation of chaos.
\newblock In {\em Ecole d'{\'e}t{\'e} de probabilit{\'e}s de Saint-Flour
  XIX—1989}, volume 1464 of {\em Lecture Notes in Mathematics}, pages
  165--251. Springer, Berlin, 1991.

\bibitem{tanaka1984limit}
H.~Tanaka.
\newblock Limit theorems for certain diffusion processes with interaction.
\newblock In {\em North-Holland Mathematical Library}, volume~32, pages
  469--488. Elsevier, 1984.

\bibitem{tanana2016approximate}
V.~P. Tanana, E.~Y. Vishnyakov, and A.~I. Sidikova.
\newblock An approximate solution of a {Fredholm} integral equation of the
  first kind by the residual method.
\newblock {\em Numerical Analysis and Applications}, 9(1):74--81, 2016.

\bibitem{tong2010image}
S.~Tong, A.~M. Alessio, and P.~E. Kinahan.
\newblock Image reconstruction for {PET/CT} scanners: past achievements and
  future challenges.
\newblock {\em Imaging in Medicine}, 2(5):529, 2010.

\bibitem{van2015astra}
W.~Van~Aarle, W.~J. Palenstijn, J.~De~Beenhouwer, T.~Altantzis, S.~Bals, K.~J.
  Batenburg, and J.~Sijbers.
\newblock The {ASTRA Toolbox: A platform for advanced algorithm development in
  electron tomography}.
\newblock {\em Ultramicroscopy}, 157:35--47, 2015.

\bibitem{villani2009optimal}
C.~Villani.
\newblock {\em Optimal Transport: Old and New}, volume 338.
\newblock Springer Science \& Business Media, 2008.

\bibitem{wahba1977practical}
G.~Wahba.
\newblock Practical approximate solutions to linear operator equations when the
  data are noisy.
\newblock {\em SIAM Journal on Numerical Analysis}, 14(4):651--667, 1977.

\bibitem{wang2020bayesian}
S.~Wang, X.~Yang, L.~Li, P.~Nadler, R.~Arcucci, Y.~Huang, Z.~Teng, and Y.~Guo.
\newblock {A Bayesian Updating Scheme for Pandemics: Estimating the Infection
  Dynamics of COVID-19}.
\newblock {\em IEEE Computational Intelligence Magazine}, 15(4):23--33, 2020.

\bibitem{wang2013tamed}
X.~Wang and S.~Gan.
\newblock The tamed {Milstein method for commutative stochastic differential
  equations with non-globally Lipschitz} continuous coefficients.
\newblock {\em Journal of Difference Equations and Applications},
  19(3):466--490, 2013.

\bibitem{webb2017introduction}
A.~G. Webb.
\newblock {\em Introduction to Biomedical Imaging}.
\newblock John Wiley \& Sons, 2017.

\bibitem{yang2020density}
R.~Yang, D.~W. Apley, J.~Staum, and D.~Ruppert.
\newblock Density deconvolution with additive measurement errors using
  quadratic programming.
\newblock {\em Journal of Computational and Graphical Statistics},
  29(3):580--591, 2020.

\bibitem{jin2019expectation}
C.~Zhang, S.~Arridge, and B.~Jin.
\newblock Expectation propagation for {Poisson data}.
\newblock {\em Inverse Problems}, 35(8):085006, 2019.

\bibitem{zhang2014multi}
H.~Zhang, D.~Wipf, and Y.~Zhang.
\newblock Multi-observation blind deconvolution with an adaptive sparse prior.
\newblock {\em IEEE Transactions on Pattern Analysis and Machine Intelligence},
  36(8):1628--1643, 2014.

\end{thebibliography}
\cleardoublepage
\appendix
\section{Proofs of \Cref{sec:vari-point-view}}

We start this section by recalling some basic facts on $\Gamma$-convergence which will be used in the proof of  \Cref{prop:convergence_minimum} and \Cref{prop:Gun_approx_mu} below. We refer to \cite{dalmaso1993introduction} for a more complete introduction to $\Gamma$-convergence.

\subsection{Basics on $\Gamma$-convergence}
First, we recall that a function $f: \ \msx \to \rset$ (where $\msx$ is a metric
space) is coercive if for any $t \in \rset$, $f^{-1}(\ocint{-\infty, t})$ is
relatively compact. This definition can be extended to the case where $\msx$ is
only a topological space, see \cite[Definition 1.12]{dalmaso1993introduction}.

Let $\msf = \ensembleLigne{f_\alpha}{\alpha \in \msa}$
where $\msa$ is a topological space and for any $\alpha \in \msa$,
$f_\alpha: \msx \to \rset \cup \{+\infty\}$ with $\msx$ a metric space.  We say
that $\msf$ is a $\Gamma$-system if the following hold:
\begin{enumerate}[wide, labelindent=0pt, label=(\alph*)]
\item For any $x \in \msx$, $\alpha^\star \in \msa$, 
  $(\alpha_n)_{n \in \nset} \in \msa^\nset$ with
  $\lim_{n \to +\infty} \alpha_n = \alpha^\star$ and 
  $(x_n)_{n \in \nset} \in \msx^\nset$ with $\lim_{n \to +\infty} x_n =x$
  we have $\liminf_{n \to +\infty} f_{\alpha_n}(x_n) \geq f_{\alpha^\star}(x)$.
\item For any $x \in \msx$, $\alpha^\star \in \msa$,
  $(\alpha_n)_{n \in \nset} \in \msa^\nset$ with
  $\lim_{n \to+\infty} \alpha_n = \alpha^\star$ there exists
  $(x_n)_{n \in \nset} \in \msx^\nset$ with $\lim_{n \to +\infty} x_n =x$ such
  that $\lim_{n \to +\infty} f_{\alpha_n}(x_n) = f_{\alpha^\star}(x)$.
\end{enumerate}

Under some additional coercivity assumption (see \Cref{sec:notation} for a definition), we have the following result, see \cite[Corollary 7.24]{dalmaso1993introduction} for a proof.

\begin{proposition}
  \label{prop:gamma_cv}
  Let $\msf = \ensembleLigne{f_\alpha}{\alpha \in \msa}$ where $\msa$ is a
  topological space and for any $\alpha \in \msa$,
  $f_\alpha: \msx \to \rset \cup \{+\infty\}$ with $\msx$ a metric space. Assume
  that $\msf$ is a $\Gamma$-system, is equicoercive and is such that for any
  $\alpha \in \msa$, $f_\alpha$ admits a unique minimizer. Then for any
  $(\alpha_n)_{n \in \nset} \in \msa^\nset$ and
  $(x_n)_{n \in \nset} \in \msx^\nset$ such that for any $n \in \nset$, $x_n$ is
  the minimizer of $f_{\alpha_n}$ and
  $\lim_{n \to +\infty} \alpha_n =\alpha^\star$, we have that
  $\lim_{n \to +\infty} x_n =x^\star$ the minimizer of $f_{\alpha^\star}$.
\end{proposition}

The equicoercivity condition can be removed if we assume that the sequence
$(x_n)_{n \in \nset}$ converges, see \cite[Corollary
7.20]{dalmaso1993introduction}.

\begin{proposition}
  \label{prop:gamma_cv_without}
  Let $\msf = \ensembleLigne{f_\alpha}{\alpha \in \msa}$ where $\msa$ is a
  topological space and for any $\alpha \in \msa$,
  $f_\alpha: \msx \to \rset \cup \{+\infty\}$ with $\msx$ a metric space. Assume
  that $\msf$ is a $\Gamma$-system  and is such that for any
  $\alpha \in \msa$, $f_\alpha$ admits a unique minimizer. Then
  for any $(\alpha_n)_{n \in \nset} \in \msa^\nset$ and
  $(x_n)_{n \in \nset} \in \msx^\nset$ such that for any $n \in \nset$, $x_n$ is
  the minimizer of $f_{\alpha_n}$ and
  $\lim_{n \to +\infty} \alpha_n =\alpha^\star$, if there exists
  $x^\star \in \msx$ such that $\lim_{n \to +\infty} x_n =x^\star$ then
  $x^\star$ is a minimizer of $f_{\alpha^\star}$.
\end{proposition}

\subsection{Proof of \Cref{prop:g_prop}}
\label{proof:prop:g_prop}

We recall that for any $\eta \geq 0$ and $\pi \in \Pens(\rset^d)$ 
\begin{equation}
  \textstyle{
  \Gun^{\eta}(\pi) = -\int_{\rset^p} \log(\pi[\kker(\cdot, y)] + \eta) \rmd \mu(y) .
}
\end{equation}

  We divide the proof into two parts.
  \begin{enumerate}[wide, labelindent=0pt, label=(\alph*)]
  \item Let $\eta \geq 0$. We first show that $\Gun^\eta$ is lower bounded. For
    any $\pi \in \Pens(\rset^d)$ we have using \Cref{assum:general_kker}
    \begin{equation}
      \textstyle{
        \int_{\rset^p} \1_{\coint{1, +\infty}}(\pi[\kker(\cdot, y)] + \eta) \log(\pi[\kker(\cdot, y)] + \eta) \rmd \mu(y) \leq \max(\log(\Mtt + \eta),0).
        }
  \end{equation}
  Therefore,
  $\Gun^\eta: \ \Pens(\rset^d) \to \ccint{-\max(\log(\Mtt + \eta), 0),
    +\infty}$. $\Gun^\eta$ is convex since $t \mapsto -\log(t + \eta)$ is convex and $\pi\mapsto \pi[\kker(\cdot, y)]$ is a linear function of $\pi$ (and therefore convex too).
We now show that $\Gun^\eta$ is
  lower semi-continuous. Let $(\pi_n)_{n \in \nset}$ be such that
  $\lim_{n \to +\infty} \pi_n = \pi \in \Pens(\rset^d)$. Then, since for any
  $(x, y) \in \rset^d\times \rset^p$, $0 \leq \kker(x,y) \leq \Mtt$ and
  $\kker \in \rmc(\rset^d \times \rset^p, \rset)$ we have for any
  $y \in \rset^p$,
  $\lim_{n \to +\infty} \pi_n[\kker(\cdot, y)] = \pi[\kker(\cdot, y)]$. For any
  $n \in \nset$ and $y \in \rset^d$ we have
  \begin{equation}
    \abs{\1_{\coint{1, +\infty}}(\pi_n[\kker(\cdot, y)] + \eta) \log(\pi_n[\kker(\cdot, y)] + \eta)} \leq \max(\log(\Mtt + \eta),0), 
  \end{equation}
  Therefore, using the reverse Fatou's lemma and the fact that
  $t \mapsto \1_{\coint{1,+\infty}}(t)$ is upper semi-continuous on
  $\coint{0, +\infty}$ we obtain that
  \begin{align}
    \label{eq:limsup}
  \textstyle{\limsup_{n \to +\infty} \int_{\rset^p} \1_{\coint{1,
      +\infty}}(\pi_n[\kker(\cdot, y)] + \eta) \log(\pi_n[\kker(\cdot, y)] + \eta) \rmd \mu(y)}
  \\ \textstyle{\leq \int_{\rset^p} \1_{\coint{1, +\infty}}(\pi[\kker(\cdot, y)] + \eta)
  \log(\pi[\kker(\cdot, y)] + \eta) \rmd \mu(y).}
\end{align}
Using Fatou's lemma and the fact that $t \mapsto -\1_{\ccint{0,1}}(t)$ is lower
semi-continuous on $\coint{0, +\infty}$, we have 
\begin{align}
  &\label{eq:liminf}
  \textstyle{\liminf_{n \to +\infty} -\int_{\rset^p} \1_{\ccint{0,1}}(\pi_n[\kker(\cdot, y)]) \log(\pi_n[\kker(\cdot, y)]) \rmd \mu(y) \geq} \\
  &\qquad\qquad\textstyle{- \int_{\rset^p} \1_{\ccint{0,1}}(\pi[\kker(\cdot, y)]) \log(\pi[\kker(\cdot, y)]) \rmd \mu(y). }
\end{align}
Hence, combining~\eqref{eq:limsup},~\eqref{eq:liminf}, we get that
$\underset{n \to +\infty}{\liminf}\ \Gun^\eta(\pi_n) \geq \Gun^\eta(\pi)$ and $\Gun^\eta$
is lower semi-continuous.\\
We now show that $\Gun^\eta$ is not coercive. Let
$\beta = \log(2) - \int_{\rset^p} \log(\kker(0, y)) \rmd \mu(y)$ and
$\Pi = \ensembleLigne{\pi_R = (\updelta_0 + \updelta_R)/2 \in \Pens(\rset^d)}{R
  > 0}$. Then, for any $R \geq 0$, $\pi_R(\cball{0}{R/2}^\complementary) = 1/2$,
where $\cball{0}{R/2}$ denotes the ball of radius $R/2$ centred at $0$. Hence
$\Pi$ is not relatively compact in $\Pens(\rset^d)$. However, for any
$R \geq 0$, we have
\begin{align}
 & \textstyle{
  \Gun^{\eta}(\pi_R) = - \int_{\rset^p} \log((\kker(0,  y) + \kker(R, y))/2 + \eta ) \rmd \mu(y) }\\
  &\qquad\qquad\textstyle{\leq \log(2) - \int_{\rset^p} \log(\kker(0,  y)) \rmd \mu(y) \leq \beta.}
\end{align}
Hence, $\Pi \subset (\Gun^\eta)^{-1}(\ocint{-\infty, \beta})$ and
$(\Gun^\eta)^{-1}(\ocint{-\infty, \beta})$ is not relatively compact which
implies that $\Gun^\eta$ is not coercive.
\item Second, let $\eta > 0$. Then, we have for any $y \in \rset^p$ and $\pi \in \Pens(\rset^d)$
  \begin{equation}
    \label{eq:dominated}
    \abs{\log(\pi[\kker(\cdot, y)] + \eta )} \leq \max(\abs{\log(\eta)}, \abs{\log(\Mtt + \eta)}).
  \end{equation}
  Hence, $\Gun^\eta$ is proper. Let $(\pi_n) \in (\Pens(\rset^d))^\nset$ be such
  that $\lim_{n \to +\infty} \pi_n = \pi \in \Pens(\rset^d)$. Since for any
  $(x, y )\in \rset^d\times \rset^p$, $\abs{\kker(x,y)} \leq \Mtt$ and
  $\kker \in \rmc(\rset^d \times \rset^p, \rset)$, we have that
  $\lim_{n \to +\infty} \pi_n[\kker(\cdot, y)] = \pi[\kker(\cdot,
  y)]$. Combining this result,~\eqref{eq:dominated} and the dominated convergence theorem we get that
  $\lim_{n \to +\infty} \Gun^\eta(\pi_n) = \Gun^\eta(\pi)$. Therefore, we have
  that $\Gun^\eta \in \rmc(\Pens(\rset^d), \rset)$.
\end{enumerate}

\subsection{Proof of \Cref{prop:convergence_minimum}}
\label{proof:prop:convergence_minimum}

We recall that for any $\alpha, \eta \geq 0$ and $\pi \in \Pens(\rset^d)$ 
\begin{equation}
  \textstyle{
  \Gun_{\alpha}^{\eta}(\pi) = -\int_{\rset^p} \log(\pi[\kker(\cdot, y)] + \eta) \rmd \mu(y) + \alpha \KLLigne{\pi}{\pi_0}.
}
\end{equation}

Since, for $\eta>0$, $\Gun^\eta$ is proper as shown in \Cref{prop:g_prop}--(b) and $\KL{\pi}{\pi_0}\geq 0$ with equality for $\pi=\pi_0$, $\Gun_\alpha^\eta$ is proper too.
\cite[Lemma
1.4.3-(b)]{dupuis1997weak} guarantees that $\KL{\pi}{\pi_0}$ is strictly convex and lower semi-continuous.
Combining this with the results on $\Gun^\eta$ in \Cref{prop:g_prop} we find that $\Gun_\alpha^\eta$ is strictly convex and lower semi-continuous.

To see that $\Gun^\eta_\alpha$ is coercive observe that $\Gun^\eta_\alpha$ is the sum of the lower bounded lower semi-continuous functional $\Gun^\eta$ and of the coercive functional \cite[Lemma
 1.4.3-(c)]{dupuis1997weak} $\KL{\pi}{\pi_0}$, then for any $\beta\in\rset$
\begin{align}
    S&:=\left\lbrace\pi\in \Pens(\rset^d): \Gun^\eta(\pi)+\alpha\KL{\pi}{\pi_0}\leq \beta\right\rbrace \\
    &\subseteq \left\lbrace\pi\in \Pens(\rset^d): \alpha\KL{\pi}{\pi_0}\leq \beta-\max(\abs{\log(\eta)}, \abs{\log(\Mtt + \eta)})\right\rbrace:=\tilde{S},
\end{align}
since $\vert \Gun^\eta(\pi)\vert \leq \max(\abs{\log(\eta)}, \abs{\log(\Mtt + \eta)})$.
$\tilde{S}$ is relatively compact since $\KL{\pi}{\pi_0}$ is coercive and thus $S$ is also relatively compact, showing that $\Gun^\eta_\alpha$ is coercive.
Hence, for any $\alpha, \eta > 0$, $\Gun_{\alpha}^\eta$ admits a unique minimizer denoted $\pi_{\alpha, \eta}^\star$.

  For any $\tilde{\alpha}, \tilde{\eta} >0$, $\pi \in \Pens(\rset^d)$, let
  $(\alpha_n)_{n \in \nset} \in \ooint{0, +\infty}^\nset$ be such that
  $\lim_{n \to +\infty} \alpha_n = 	\tilde{\alpha}$ and $(\eta_n) \in \ooint{0, +\infty}^\nset$
  such that $\lim_{n \to +\infty} \eta_n = \tilde{\eta}$, then we have
  $\lim_{n \to +\infty} \Gun_{\alpha_n}^{\eta_n}(\pi) =
  \Gun_{\tilde{\alpha}}^{\tilde{\eta}}(\pi)$ using the dominated convergence
  theorem. In addition, using the dominated convergence theorem and \cite[Lemma
  1.4.3]{dupuis1997weak} we get that for any
  $(\pi_n)_{n \in \nset} \in \Pens(\rset^d)^\nset$ such that
  $\lim_{n \to +\infty} \pi_n = \pi \in \Pens(\rset^d)$ we have
  $\liminf_{n \to +\infty} \Gun_{\alpha_n}^{\eta_n}(\pi_n) \geq
  \Gun_{\tilde{\alpha}}^{\tilde{\eta}}(\pi)$.  Therefore,
  $\ensembleLigne{\Gun_{\alpha}^\eta}{\alpha, \eta > 0}$ is a
  $\Gamma$-system. In addition, for any $\tilde{\alpha} > 0$ and $\eta > 0$ we
  have that for any $\alpha \in \ooint{\tilde{\alpha}/2, +\infty}$,
  $\Gun_{\alpha}^\eta \geq \Gun_{\tilde{\alpha} / 2}^\eta$. Hence the
  $\Gamma$-system
  $\ensembleLigne{\Gun_{\alpha}^\eta}{\alpha > \tilde{\alpha} /2, \eta > 0}$ is
  equicoercive and using \Cref{prop:gamma_cv} we have that $\mtt$ and $\dtt$ are
  continuous at $(\alpha, \eta)$ for any $\alpha, \eta > 0$.

  \subsection{Proof of \Cref{prop:Gun_approx_mu}}
\label{proof:prop:Gun_approx_mu}

  Let $\alpha, \eta > 0$, let $(\mu_n)_{n \in \nset} \in (\Pens_1(\rset^d))^\nset$
  be such that $\lim_{n \to +\infty} \wassersteinD[1](\mu_n,\mu) = 0$ with $\mu \in \Pens_1(\rset^d)$, and $(\pi_n)_{n \in \nset} \in (\Pens_1(\rset^d))^\nset$ such that
  $\lim_{n \to +\infty} \wassersteinD[1](\pi_n, \pi) = 0$ with
  $\pi \in \Pens_1(\rset^d)$.  For any $\pi \in \Pens(\rset^d)$ denote
  $f_\pi(y) = \log(\pi[\kker(\cdot, y)] +\eta)$. For any
  $\nu \in \Pens(\rset^p)$ define $\Gun_\alpha^\eta(\cdot, \nu)$ as
  \begin{equation}
    \textstyle{\Gun_\alpha^\eta(\pi, \nu) := -\int_{\rset^p} \log(\pi[\kker(\cdot, y)] + \eta) \rmd \nu(x) + \alpha \KLLigne{\pi}{\pi_0}}
  \end{equation}
for any $\pi \in \Pens(\rset^d)$.
  Using
  the dominated convergence theorem, for any $\pi \in \Pens(\rset^d)$,
  $f_\pi \in \rmc^1(\rset^p, \rset)$ is differentiable and we have that for any
  $y \in \rset^p$
  \begin{equation}
    \norm{\nabla f_\pi(y)} \leq \pi[ \norm{\nabla_2 \kker(\cdot, y)}] / (\pi[\kker(\cdot, y)] + \eta) \leq \Mtt / \eta. 
  \end{equation}
  Therefore, for any $\pi \in \Pens(\rset^d)$ we have for any $y_1, y_2 \in \rset^p$
  \begin{equation}
    \abs{\log(\pi[\kker(\cdot, y_1)] + \eta) - \log(\pi[\kker(\cdot, y_2)] + \eta)} \leq (\Mtt / \eta) \norm{y_1 - y_2}.  
  \end{equation}
  Using this result we have that
  \begin{align}
    \abs{\Gun^\eta_\alpha(\pi_n, \mu_n) - \Gun^\eta_\alpha(\pi, \mu)} &\leq \abs{\Gun^\eta_\alpha(\pi_n, \mu_n) - \Gun^\eta_\alpha(\pi_n, \mu)} + \abs{\Gun^\eta_\alpha(\pi_n, \mu) - \Gun^\eta_\alpha(\pi, \mu)} \\
    &\textstyle{\leq (\Mtt / \eta) \wassersteinD[1](\mu_n, \mu) }\\
    &\textstyle{+ \int_{\rset^p} \abs{\log(\pi_n[\kker(\cdot, y)] +\eta) - \log(\pi[\kker(\cdot, y)] +\eta)} \rmd \mu(y)}\\
    &\textstyle{+ \alpha\vert \KL{\pi_n}{\pi_0}-\KL{\pi}{\pi_0}\vert,}
  \end{align}
  where the second inequality follows using the dual representation of $\wassersteinD[1]$.
  
  Combining the result above, the dominated convergence theorem, the lower semi-continuity of the KL \cite[Lemma
1.4.3-(b)]{dupuis1997weak} and the fact that
  $\underset{n \to +\infty}{\lim} \wassersteinD[1](\mu_n,\mu) = 0$, we get that
  $\underset{n \to +\infty}{\lim\inf}\ \Gun^\eta_\alpha(\pi_n, \mu_n) \geq \Gun^\eta_\alpha(\pi,
  \mu)$. Thus, $\ensembleLigne{\Gun^{\eta}_\alpha(\cdot, \mu_n)}{n \in \nset}$
  is a $\Gamma$-system. Since we have that for any $\pi \in \Pens(\rset^d)$ and
  $n \in \nset$, $\Gun^{\eta}_\alpha(\pi, \mu_n) \geq -\max(\abs{\log(\eta)}, \abs{\log(\Mtt + \eta)})+\alpha \KLLigne{\pi}{\pi_0}$, we have that
  $\ensembleLigne{\Gun^{\eta}_\alpha(\cdot, \mu_n)}{n \in \nset}$ is equicoercive. We conclude using
  \Cref{prop:gamma_cv}.

\subsection{Proof of \Cref{prop:tikhonov_reg}}
\label{app:proof_tikhonov}

  Take $\alpha > 0$ and  $\eta \geq 0$. Since by \Cref{prop:convergence_minimum},
  $\Gun_{\alpha}^\eta$ is proper,
  $\Gun_{\alpha}^\eta(\pi_{\alpha, \eta}^\star) < +\infty$ and therefore,
  $\KLLigne{\pi_{\alpha, \eta}^{\star}}{\pi_0} < +\infty$. Hence,
  $\pi_{\alpha, \eta}^\star$ admits a density w.r.t to $\pi_0$. Since $\pi_0$ is
  equivalent to the Lebesgue measure, we get that $\pi_{\alpha, \eta}^\star$
  admits a density with respect to the Lebesgue measure. In addition, let
  $\pi_1 \in \Pens(\rset^d)$ be such that for any $x \in \rset^d$ we have for some $\tau > 0$
  \begin{equation}
    (\rmd \pi_1 / \rmd \pi_0)(x) = f(x) / \pi_0[f], \qquad f(x) = \exp[(\tau / 2) \norm{x}^2]. 
  \end{equation}
  Since $\KLLigne{\pi_{\alpha, \eta}^\star}{\pi_0} < +\infty$ we have
  $\int_{\rset^d} \abs{\log((\rmd \pi_{\alpha, \eta}^\star / \rmd \pi_0)(x))} \rmd
  \pi_{\alpha, \eta}^\star(x) < +\infty$ and
  \begin{align}
    \label{eq:kl1_max}
    &\textstyle{\int_{\rset^d} \log((\rmd \pi_{\alpha, \eta}^\star / \rmd \pi_1)(x)) \1_{\coint{1,+\infty}}((\rmd \pi_{\alpha, \eta}^\star / \rmd \pi_1)(x)) \rmd
    \pi_{\alpha, \eta}^\star(x) }\\
    & \qquad \quad \textstyle{= \int_{\rset^d} \defEns{\log((\rmd \pi_{\alpha, \eta}^\star / \rmd \pi_0)(x)) - (\tau/2)\norm{x}^2 +\log(\pi_0[f])} }\\
    &\qquad\qquad\times \textstyle{\1_{\coint{1,+\infty}}((\rmd \pi_{\alpha, \eta}^\star / \rmd \pi_1)(x)) \rmd
      \pi_{\alpha, \eta}^\star(x)} \\
        & \qquad \quad \textstyle{\leq \int_{\rset^d} \defEns{\abs{\log((\rmd \pi_{\alpha, \eta}^\star / \rmd \pi_0)(x))}  +\log(\pi_0[f])} }\\
        &\qquad\qquad\times \textstyle{\1_{\coint{1,+\infty}}((\rmd \pi_{\alpha, \eta}^\star / \rmd \pi_1)(x)) \rmd
  \pi_{\alpha, \eta}^\star(x) < +\infty. }
  \end{align}
  Since $t \mapsto t \log(t)$ is bounded on $\ccint{0,1}$ we have that
  \begin{equation}
    \label{eq:kl1_min}
    - \textstyle{\int_{\rset^d} \log((\rmd \pi_{\alpha, \eta}^\star / \rmd \pi_1)(x)) \1_{\ccint{0,1}}((\rmd \pi_{\alpha, \eta}^\star / \rmd \pi_1)(x)) \rmd
    \pi_{\alpha, \eta}^\star(x) < +\infty. }
  \end{equation}
  Combining~\eqref{eq:kl1_min} and~\eqref{eq:kl1_max} we get that
  $\KLLigne{\pi_{\alpha, \eta}^\star}{\pi_1} < +\infty$. Therefore, using \cite[Equation
  2.6]{csiszar1975divergence}, we get that
  \begin{align}
    &\textstyle{(\tau/2) \int_{\rset^d} \norm{x}^2 \rmd \pi_{\alpha, \eta}^\star(x) - \log(\pi_0[f])} \\ & \qquad \qquad \qquad \textstyle{= \int_{\rset^d} \log((\rmd \pi_1 / \rmd \pi_0)(x)) \rmd \pi_{\alpha, \eta}^\star(x) = \KLLigne{\pi_{\alpha, \eta}^\star}{\pi_0} - \KLLigne{\pi_{\alpha, \eta}^\star}{\pi_1} < +\infty. }
  \end{align}
  Therefore, $\pi_{\alpha, \eta}^\star \in \Pens_2(\rset^d)$.

  Let
  $(\alpha_n)_{n \in \nset} \ooint{0, +\infty}^\nset, (\eta_n)_{n \in \nset} \in
  \coint{0, +\infty}^\nset$ such that $\underset{n \to +\infty}{\lim} \alpha_n = 0$,
  $\underset{n \to +\infty}{\lim} \eta_n = 0$ and
  $\underset{n \to +\infty}{\lim} \pi_{\alpha_n, \eta_n}^{\star} = \pi^\star$ with
  $\pi^{\star} \in \Pens_2(\rset^d)$. For any
  $(\pi_n)_{n \in \nset} \in \Pens_2(\rset^d)^{\nset}$ such that
  $\lim_{n \to +\infty} \pi_n =\pi$ we have using the monotonicity in $\alpha$
  of $\Gun_\alpha^\eta$ and Fatou's lemma
  \begin{equation}
    \liminf_{n \to +\infty} \Gun_{\alpha_n}^{\eta_n}(\pi_n) \geq \liminf_{n \to +\infty} \Gun^{\eta_n}(\pi_n) \geq \Gun(\pi).
  \end{equation}

  Let
  $\pi \in \Pens_2(\rset^d)$ and for any $n \in \nset$ define
  $\pi_n \in \Pens_2(\rset^d)$ with density w.r.t. the Lebesgue measure given for any $x \in \rset^d$ by
  \begin{equation}
    \textstyle{
       \pi_n(x) = \int_{\rset^d} (2\uppi \alpha_n)^{-d/2} \exp[-\norm{x - \tilde{x}}^2/(2\alpha_n)] \rmd \pi(\tilde{x}).
      }
    \end{equation}
    Let $X$ be a random variable with distribution $\pi$ and for any
    $n \in \nset$, let $X_n = X + \alpha_n^{1/2} Z$ where $Z$ is a Gaussian
    random variable with zero mean and identity covariance matrix. For
    any $n \in \nset$, $X_n$ has distribution $\pi_n$. Therefore, we have
    $\lim_{n \to +\infty} \wassersteinD[2](\pi_n, \pi)^2 \leq \lim_{n \to
      +\infty}\expeLigne{\normLigne{X - X_n}^2} =0$.
    
    In what follows, we show that
    $\lim_{n \to +\infty} \Gun^{\eta_n}(\pi_n) = \Gun(\pi)$. Define
    $h: \ \coint{0, +\infty} \to \rset$ as
    $h(t) = t\log(t)$. For any $n \in \nset$, using Jensen's inequality and the
    Fubini-Tonelli theorem we have
   \begin{align}
     \textstyle{\int_{\rset^d} h( \pi_n(x)) \rmd x} & \textstyle{\leq \int_{\rset^d} \int_{\rset^d} h((2\uppi \alpha_n)^{-d/2} \exp[-\norm{x - \tilde{x}}^2/(2\alpha_n)]) \rmd \pi(\tilde{x}) \rmd x }\\
    &\leq \textstyle{-(d/2)\log(2\uppi \alpha_n)\int_{\rset^d}\int_{\rset^d} (2\uppi \alpha_n)^{-d/2} \exp[-\norm{x-\tilde{x}}^2/(2\alpha_n)] \rmd x\rmd \pi(\tilde{x}) }\\
    &\leq \textstyle{-(d/2)\log(2\uppi \alpha_n).}
 \label{eq:ent_alpha}
   \end{align}
   Using \tup{\Cref{assum:pi0}} there exists $C_1 \geq 0$ and $\tau > 0$ such that
   for any $x \in \rset^d$, $\abs{U(x)} \leq C_1 + \tau \norm{x}^2$.  Using this
   result, the Fubini-Tonelli theorem, the fact that for any $\tilde{x}\in\rset^d$,
   $\norm{x}^2 \leq 2\norm{x-\tilde{x}}^2 + 2\norm{\tilde{x}}^2$ and that
   $\int_{\rset^d} \norm{\tilde{x}}^2 \rmd \pi(\tilde{x}) < +\infty$ we have for any $n \in \nset$
   \begin{align}
     \textstyle{\int_{\rset^d} \abs{U(x)} \rmd \pi_n(x)} & \textstyle{\leq C_1 + \tau \int_{\rset^d} \int_{\rset^d} \norm{x}^2 (2\uppi \alpha_n)^{-d/2} \exp[-\norm{x - \tilde{x}}^2/(2\alpha_n)]  \rmd \tilde{x} \rmd \pi(x) }\\                                            &\leq\textstyle{ C_1 + 2 \tau d \sup_{n \in \nset} \alpha_n + 2 \tau \int_{\rset^d} \norm{\tilde{x}}^2 \rmd \pi(\tilde{x}) < +\infty.}
                                               \label{eq:U_alpha}
   \end{align}
   Combining~\eqref{eq:ent_alpha} and~\eqref{eq:U_alpha} we have
   $\lim_{n \to +\infty} \alpha_n \KLLigne{\pi_n}{\pi_0} = 0$.

For any $n \in \nset$ and $y \in \rset^p$,
   $\log(\pi_n[\kker(\cdot, y)] + \eta_n) \leq \log(\Mtt + \sup_{n \in \nset}
   \eta_n)$. Define $\Phi: \ \rset^d \to \rset$, $\Phi: x \mapsto \norm{x}^2$ and $\Psi: \ \rset^p \to \rset$, $\Psi: y \mapsto \norm{y}^2$.
   Using Jensen's inequality and the fact that under \tup{\Cref{assum:pi0}} there
   exists $C_2 \geq 0$ such that for any $(x, y) \in \rset^d\times\rset^p$,
   $\exp[-\Phi(x) - \Psi(y)] / \kker(x,y)\leq C_2\exp[C_2]$, we have for any
   $y \in \rset^p$ and $n \in \nset$ that
   \begin{equation}
     \log(\pi_n[\kker(\cdot, y)] + \eta_n) \geq \log(\pi_n[\kker(\cdot, y)]) \geq -\Psi(y) - \sup_{n \in \nset} \pi_n[\Phi] - C_2 -\log C_2.
   \end{equation}
   Since, $(\pi_n)_{n \in \nset}$ is relatively compact in $\Pens_2(\rset^d)$
   and $\mu \in \Pens_2(\rset^p)$, there exists $C \geq 0$ such that for any
   $n \in \nset$, $\pi_n[\Phi] \leq C$, see \cite[Definition
   6.8]{villani2009optimal} and $\mu[\Psi] < +\infty$.  Hence, there exists
   $M \geq 0$ such that for any $n \in \nset$ and $y \in \rset^p$ such that
   $\absLigne{\log(\pi_n[\kker(\cdot, y)] + \eta_n)} \leq M$. Combining this
   result and the dominated convergence theorem, we get that
   $\lim_{n \to +\infty} \Gun^{\eta_n}(\pi_n) = \Gun(\pi)$.

   Therefore, we have that
   $\Gun(\pi) = \lim_{n \to +\infty} \Gun_{\alpha_n}^{\eta_n}(\pi_n)$, and $\ensembleLigne{\Gun_{\alpha}^\eta}{\alpha, \eta \geq 0}$ is a
   $\Gamma$-system.  Using \Cref{prop:gamma_cv_without}, we get that $\pi^\star$ is a
   minimizer of $\Gun$.

   Assume that there exists $\pi^{\dagger} \in \argmin_{\Pens(\rset^d)} \Gun$
   such that $\KLLigne{\pi^\dagger}{\pi_0} < \KLLigne{\pi^\star}{\pi_0}$. In what follows
   we assume that $\KLLigne{\pi^\star}{\pi_0} < +\infty$, the case where
   $\KLLigne{\pi^\star}{\pi_0} = +\infty$ is similar and omitted. There
   exists $\vareps > 0$ such that
   $\KLLigne{\pi^\dagger}{\pi_0} = \KLLigne{\pi^\star}{\pi_0} - \vareps$. Since
   $\liminf_{n \to +\infty} \KLLigne{\pi_{\alpha_n, \eta_n}^\star}{\pi_0} \geq
   \KLLigne{\pi^\star}{\pi_0}$, there exists $n_0 \in \nset$ such that for any
   $n \in \nset$ with $n \geq n_0$ we have
   \begin{equation}
     \label{eq:kl_alpha_dagger}
    \KLLigne{\pi_{\alpha_n, \eta_n}^\star}{\pi_0} \geq \KLLigne{\pi^\star}{\pi_0} -\vareps / 2 \geq \KLLigne{\pi^\dagger}{\pi_0} + \vareps / 2. 
   \end{equation}
   Therefore we have that for any $n \geq n_0$
   \begin{align}
     \Gun_{\alpha_n}^{\eta_n}(\pi_{\alpha_n, \eta_n}^\star) &= \Gun^{\eta_n}(\pi_{\alpha_n, \eta_n}^\star) + \alpha_n \KLLigne{\pi_{\alpha_n, \eta_n}^\star}{\pi_0} \\
     &\geq \Gun^{\eta_n}(\pi^\dagger) + \alpha_n \KLLigne{\pi^\dagger}{\pi_0} + \alpha_n \vareps / 2 > \Gun_{\alpha_n}^{\eta_n}(\pi^\dagger),
   \end{align}
   which is a contradiction since for any $n \in \nset$, $\pi_{\alpha_n, \eta_n}^\star$ is a
   minimizer of $\Gun_{\alpha_n}^{\eta_n}$. Hence, for any
   $\pi \in \argmin_{\Pens(\rset^d)} \Gun$,
   $\KLLigne{\pi^\star}{\pi_0} \leq \KLLigne{\pi}{\pi_0}$, which concludes the proof.


\section{Proofs of \Cref{sec:part-appr-mcke}}
\label{sec:proof-cref-sec:part}

\subsection{Basics on Wasserstein gradient flows and subdifferential of $\Gun_\alpha^\eta$}
\label{app:subdifferential}
In this section we recall basic results on gradient flows in Wasserstein spaces
and provide a slight extension of the classical variational formula. We refer to \cite{ambrosio2008gradient} for further details.  We denote by $\rmL^2(\rset^d, \nu)$
the set of square integrable functions w.r.t. a measure $\nu$, 
by $\rmc_c^\infty(\rset^d)$ the set of infinitely
many times differentiable functions with compact support and by $\rmW^{1,1}_{\mathrm{loc}}(\rset^d)$ the set of functions which are in $\rmW^{1,1}_{\mathrm{loc}}(\msx)$ for all compact sets $\msx\subset \rset^d$, where $\rmW^{1,1}_{\mathrm{loc}}(\msx)$ denotes the Sobolev space of functions whose zero-th and first derivative are bounded in $\rmL^1$.  For any
$\nu\in\Pens(\rset^d)$ and $t: \ \rset^d \to \rset^d$ measurable, we denote by
$t_{\#}\nu$ the push-forward measure, $t_{\#}\nu(A)=\nu(t^{-1}(A))$ for all
$A\in\mcb {\rset^d}$.  We start by recalling the definition of the
subdifferentiability in Wasserstein spaces, see \cite[Definition
10.1.1]{ambrosio2008gradient}. We denote $\Pens_{2}^r(\rset^d)$ the space of
probability distributions in $\Pens_2(\rset^d)$ which are absolutely continuous
w.r.t the Lebesgue measure.

\begin{definition}[Fr\'echet subdifferential]
  Let $\Phi: \ \Pens_2(\rset^d) \to \ooint{-\infty, +\infty}$. Let
  $\pi \in \Pens_{2}^r(\rset^d)$, then
  $\xi \in \rmL^2(\rset^d, \pi)$ belongs to the strong Fr\'{e}chet
    subdifferential $\partial_{\rms} \Phi(\pi)$ of $\Phi$ at $\pi$ if for any $t \in \rmL^2(\rset^d, \pi)$
    \begin{equation}
      \textstyle{
        \liminf_{t \to \Id}\left. \defEns{\Phi(t_\# \pi) - \Phi(\pi) - \int_{\rset^d} \langle \xi(x), t(x) - x \rangle \rmd \pi(x)}\middle/ \norm{t - \Id}_{\rmL^2(\rset^d, \pi)} \right. = 0.
        }
  \end{equation}
\end{definition}

The following lemma is a specific case of the chain-rule formula.

\begin{lemma}
  \label{lemma:chain_rule}
  Let $g \in \rmc^1(\rset, \rset)$ be Lipschitz, convex and lower-bounded. For any
  $y \in \rset^p$, let $\Lun : \ \Pens_2(\rset^d) \times \rset^p \to \rset$,
  $\pi \in \Pens_2(\rset^d)$, $\mu \in \Pens_2(\rset^p)$ and
  $\xi \in \rmL^2(\rset^d \times \rset^p, \pi \otimes \mu)$ be such that for any
  $y \in \rset^p$ and $\pi \in \Pens_{2}^r(\rset^d)$,
  $\xi(\pi, y) \in \partial_{\rms} \Lun(\pi, y)$ and
  $-\xi(\pi, y) \in \partial_{\rms} (-\Lun)(\pi, y)$.
  In addition, assume
  that there exists $h \in \rmL^1(\rset^p, \mu)$ such that for any
  $\pi_1, \pi_2 \in \Pens_2(\rset^d)$ and $y \in \rset^p$ we have
  \begin{equation}
    \abs{\Lun(\pi_1, y) - \Lun(\pi_2, y)} \leq h(y) \wassersteinD[2](\pi_1, \pi_2). 
  \end{equation}
Let $\Mun$ be given for any $\pi \in \Pens(\rset^d)$ by
\begin{equation}
  \textstyle{
    \Mun(\pi) = \int_{\rset^p} g(\Lun(\pi, y)) \rmd \mu(y).
    }
  \end{equation}  
  Then, we have that $\partial_{\rms} \Mun(\pi) \neq \emptyset$ and
  $\defEnsLigne{x \mapsto \int_{\rset^p} g'(\Lun(\pi, y)) \xi(\cdot, y) \rmd
    \mu(y)} \in \partial_{\mathrm{s}} \Mun(\pi)$.
\end{lemma}

\begin{proof}
  Let $\pi \in \Pens(\rset^d)$, $t \in \rmL^2(\rset^d, \pi)$ and $y \in \rset^d$.  Since $g$ is convex we have
  \begin{equation}
    g(\Lun(t_\# \pi, y)) - g(\Lun(\pi, y)) \geq g'(\Lun(\pi, y)) (\Lun(t_\# \pi, y) - \Lun(\pi, y)). 
  \end{equation}
  We have 
  \begin{align}
    \label{eq:inter_convx}
    &\textstyle{g(\Lun(t_\# \pi, y)) - g(\Lun(\pi, y)) - g'(\Lun(\pi, y)) \int_{\rset^d} \langle \xi(x,y), t(x) - x \rangle \rmd \pi(x)} \\
    &\qquad\textstyle{\geq g'(\Lun(\pi, y)) \parenthese{\Lun(t_\# \pi, y) - \Lun(\pi, y) - \int_{\rset^d} \langle \xi(x, y), t(x) - x \rangle \rmd \pi(x) }.}
  \end{align}
  In addition, using that $\xi(\pi, y) \in \partial_{\rms} \Lun(\pi, y)$ and
  $-\xi(\pi, y) \in \partial_{\rms} (-\Lun)(\pi, y)$ we have
  \begin{align}
    \textstyle{
    \liminf_{t \to \Id} \parenthese{\Lun(t_\# \pi, y) - \Lun(\pi, y) - \int_{\rset^d} \langle \xi(x, y), t(x) - x \rangle \rmd \pi(x) } / \normLigne{t - \Id}_2  \geq 0,} \\
    \textstyle{\limsup_{t \to \Id} \parenthese{\Lun(t_\# \pi, y) - \Lun(\pi, y) - \int_{\rset^d} \langle \xi(x, y), t(x) - x \rangle \rmd \pi(x) } / \normLigne{t - \Id}_2  \leq 0.}
  \end{align}
  Hence, we have
  \begin{equation}
    \textstyle{\lim_{t \to \Id} \parenthese{\Lun(t_\# \pi, y) - \Lun(\pi, y) - \int_{\rset^d} \langle \xi(x, y), t(x) - x \rangle \rmd \pi(x) } / \normLigne{t - \Id}_2  = 0.}
  \end{equation}
  We also have that
\begin{align}
  &\textstyle{\abs{\Lun(t_\# \pi, y) - \Lun(\pi, y) - \int_{\rset^d} \langle \xi(x, y), t(x) - x \rangle \rmd \pi(x) } / \normLigne{t - \Id}_2 }\\
  &\qquad\qquad\qquad\textstyle{\leq h(y) + \int_{\rset^d} \norm{\xi(x, y)} \rmd \pi(x). }
\end{align}
We conclude upon using this result, the fact that $g$ is Lipschitz,
\eqref{eq:inter_convx} and the dominated convergence theorem.
\end{proof}

We are now ready to derive the subdifferential of $\Gun_{\alpha}^{\eta}$ for any
$\eta > 0$.

\begin{lemma}[Strong subdifferential of $\Gun^{\eta}$]
  \label{lemma:subdiff_g_eta}
  Assume \tup{\Cref{assum:general_kker}}. Then for any
  $\pi \in \Pens_{2}^r(\rset^d)$ and $\eta > 0$, we have that 
  $\partial_{\rms} \Gun^\eta(\pi) \neq \emptyset$ and
  \begin{equation}
    \textstyle{
      \defEnsLigne{x \mapsto -\int_{\rset^p} (\eta + \pi[\kker(x, y)])^{-1} \nabla_1
        \kker(x, y) \rmd \mu(y)} \in \partial_{\rms} \Gun^\eta(\pi). }
  \end{equation}

\end{lemma}

\begin{proof}
  First, using that $\normLigne{\nabla^2 \kker} \leq \Mtt$ we get that for any $x_1, x_2 \in \rset^d$, $y\in\rset^p$
  \begin{equation}
    \label{eq:lambda_cvx}
    \kker(x_2, y) \geq \kker(x_1, y) + \langle \nabla_1 \kker(x_1, y), x_2 -x_1 \rangle - \Mtt \norm{x_1 - x_2}^2.
  \end{equation}
  Let $\Lun: \ \Pens_2(\rset^d) \times \rset^p \to \coint{0, +\infty}$ be given for
  any $\pi \in \Pens_2(\rset^d)$ and $y \in \rset^p$ by
  \begin{equation}
    \textstyle{
      \Lun(\pi, y) = \int_{\rset^d} \kker(x, y) \rmd \pi(x).
      }
  \end{equation}
  Using~\eqref{eq:lambda_cvx} and \cite[Proposition
  10.4.2]{ambrosio2008gradient} we get that for any $y \in \rset^p$ and
  $\pi \in \Pens_2(\rset^d)$,
  $\partial_{\rms} \Lun(\pi, y) = \{\nabla_1 \kker(\cdot, y)\}$.  Since there
  exists $\norm{\nabla_1 \kker}_{\infty} \leq \Mtt < +\infty$ we get that for
  any $\pi_1, \pi_2 \in \Pens_2(\rset^d)$ and $y \in \rset^p$
  \begin{equation}
    \label{eq:lip}
    \abs{\Lun(\pi_1, y) - \Lun(\pi_2, y)} \leq \Mtt \wassersteinD[2](\pi_1, \pi_2). 
  \end{equation}
  In addition, let $g : \ \ooint{0, +\infty} \to \rset$ given for any $t \geq 0$
  by $g(t) = -\log(t + \eta)$. Using that $g$ is convex,~\eqref{eq:lip} and that
  $\norm{\nabla_1 \kker}_{\infty} \leq \Mtt < +\infty$ in \Cref{lemma:chain_rule}, we
  get that for any $\pi \in \Pens_2(\rset^d)$,
  $\partial_{\rms} \Gun^\eta(\pi) \neq \emptyset$ and
  $\defEnsLigne{x \mapsto -\int_{\rset^d} (\eta + \pi[\kker(x, y)])^{-1} \nabla_1
    \kker(x, y) \rmd \mu(y)} \in \partial_{\rms} \Gun^\eta(\pi)$.
\end{proof}

Let $\Hun: \ \Pens(\rset^d) \to \ccint{0, +\infty}$ be defined for any
$\pi \in \Pens(\rset^d)$ by $\Hun(\pi) =
\KL{\pi}{\pi_0}$. \Cref{lemma:relative_entropy_sub} is a restatement of
\cite[Theorem 10.4.9]{ambrosio2008gradient}.

\begin{lemma}[Strong subdifferential of $\KL{\pi}{\pi_0}$]
  \label{lemma:relative_entropy_sub}
  Assume that 
  $\pi_0$
  admits a density, also denoted $\pi_0$, w.r.t. the Lebesgue measure such that for any $x \in \rset^d$,
  $\pi_0(x) \propto \exp[-U(x)]$ for some convex $U:\rset^d\to\rset$.
  Let $\pi \in \Pens_{2}^r(\rset^d)$
  such that $\Hun(\pi) < +\infty$ and
  $(\rmd \pi / \rmd \pi_0) \in \rmW^{1,1}_{\mathrm{loc}}(\rset^d, \ooint{0, +\infty})$. Then,
 \begin{align}
 \partial_{\rms} \Hun(\pi) = \defEns{x \mapsto \nabla \log(\rmd \pi / \rmd \pi_0)(x)}.
 \end{align}
  
\end{lemma}

Combining these two lemmas we are now ready to derive the subdifferential of
$\Gun_{\eta}^\alpha$ given by~\eqref{eq:G_eta} as well as one associated
Wasserstein gradient flow.

\begin{proposition}[Gradient flow]
\label{prop:total_subdifferential}
  Assume \tup{\Cref{assum:general_kker}} and that
  $\pi_0$
  admits a density, also denoted $\pi_0$, w.r.t. the Lebesgue measure such that for any $x \in \rset^d$,
  $\pi_0(x) \propto \exp[-U(x)]$ for some convex $U:\rset^d\to\rset$.
   Let $\pi \in \Pens_{2}^r(\rset^d)$ be
  such that $\KL{\pi}{\pi_0} < +\infty$ and
  $(\rmd \pi / \rmd \pi_0) \in \rmW^{1,1}_{\mathrm{loc}}(\rset^d, \ooint{0,
    +\infty})$. Then for any $\alpha, \eta > 0$
  \begin{equation}
    \textstyle{
   \defEns{x \mapsto -\int_{\rset^p} (\eta + \pi[\kker(x, y)])^{-1} \nabla_1
    \kker(x, y) \rmd \mu(y) + \alpha \nabla \log(\rmd
    \pi / \rmd \pi_0)(x)} \in \partial_{\rms} \Gun_{\alpha}^\eta(\pi). }
  \end{equation}
  In addition, if $U \in \rmc^1(\rset^d, \rset)$ and there exists
  $(\nu_t)_{t \geq 0}$ such that for any $t \in \ooint{0, +\infty}$, $\nu_t$ admits a density w.r.t. the Lebesgue measure also denoted (with a slight abuse of notation) by 
  $\nu_t \in \rmc^1(\rset^d, \rset)$, and for any $\varphi \in \rmc_c^\infty(\rset^d)$
  \begin{multline}
    \textstyle{\int_{\rset} \varphi(x) \rmd \nu_t(x) - \int_{\rset} \varphi(x) \rmd
    \nu_0(x)} \\= \textstyle{\int_0^t \int_{\rset^d} \left[  -\int_{\rset^d} \defEns{ \left. \langle
      \nabla_1 \kker(x,y), \nabla \varphi(x) \rangle \middle/ \parenthese{\int_{\rset^d} \kker(x,
      y) \rmd \nu_s(x) + \eta}} \rmd \mu(y) \right. \right.} \\ \textstyle{\left. + \alpha \langle \nabla U(x), \nabla
      \varphi(x) \rangle + \alpha \Delta \varphi(x) \vphantom{\int_{\rset^d}}\right] \rmd \nu_s(x),}
  \end{multline}
  then $(\nu_t)_{t \geq 0}$ is a Wasserstein gradient flow associated with $\Gun_{\alpha}^\eta$.
\end{proposition}

\begin{proof}
  The first part of the proof is a straightforward application of
  \Cref{lemma:subdiff_g_eta} and \Cref{lemma:relative_entropy_sub}. The second
  part of the proof follows from the definition of a Wasserstein gradient flow
  \cite[Definition 11.1.1]{ambrosio2008gradient}.
\end{proof}

\subsection{Proof of \Cref{prop:existence_uniqueness}}
\label{app:proof_eu}

  Let $\alpha, \eta > 0$. First, we show that $\bmketa$ is Lipschitz
  continuous. Let $f(t) = (\eta + t)^{-1}$. Using that for any $t \geq 0$,
  $\abs{f'(t)} \leq \eta^{-2}$ we have for any $x_1, x_2\in \rset^d$,
  $y\in\rset^p$ and $\pi_1, \pi_2 \in \Pens_1(\rset^d)$
  \begin{align}
    \label{eq:lip_b_eta}
    &\normLigne{\bmketa(x_1, \pi_1, y) - \bmketa(x_2, \pi_2, y)} \\
    & \qquad \leq \normLigne{\nabla_1 \kker(x_1, y) / (\pi_1[\kker(\cdot, y)] + \eta) - \nabla_1 \kker(x_2, y) / (\pi_2[\kker(\cdot, y)] + \eta)} \\
    & \qquad \leq \normLigne{\nabla_1 \kker(x_1, y) / (\pi_1[\kker(\cdot, y)] + \eta) - \nabla_1 \kker(x_2, y) / (\pi_1[\kker(\cdot, y)] + \eta)} \\
    & \qquad \qquad + \normLigne{\nabla_1 \kker(x_2, y) / (\pi_1[\kker(\cdot, y)] + \eta) - \nabla_1 \kker(x_2, y) / (\pi_2[\kker(\cdot, y)] + \eta)} \\
    &\qquad \leq (\Mtt/\eta) \normLigne{x_1 - x_2} + \Mtt \abs{(\pi_1[\kker(\cdot, y)] + \eta)^{-1} - (\pi_2[\kker(\cdot, y)] + \eta)^{-1}} \\
    &\qquad \leq (\Mtt/\eta)(1 + (1/\eta)) \defEns{\normLigne{x_1 - x_2} + \abs{\pi_1[\kker(\cdot, y)] - \pi_2[\kker(\cdot, y)]}}, 
  \end{align}
  where we have used the fact that, under \tup{\rref{assum:general_kker}}, $\kker$ is Lipschitz continuous.
  The rest of the proof is classical, see for instance \cite[Theorem
  1.1]{sznitman1991topics}, but is given for completeness.  Define
  $b:\rset^d\times\Pens_1(\rset^d) \to \rset^d$ such that for any $x \in \rset^d$ and $\pi \in \Pens_1(\rset^d)$ 
  \begin{equation}
    \textstyle{
      b(x, \pi) = \int_{\rset^p} \bmketa(x, \pi, y) \rmd \mu(y) - \alpha \nabla U(x).
      }
  \end{equation}
  Using \tup{\rref{assum:lip_U}-\ref{item:lip}} and~\eqref{eq:lip_b_eta} we have that for any
  $x_1, x_2 \in \rset^d$ and $\pi_1, \pi_2 \in \Pens_1(\rset^d)$\
  \begin{align}
    \label{eq:lipo_lip}
&    \textstyle{ \norm{b(x_1, \pi_1) - b(x_2, \pi_2)} \leq ((\Mtt/\eta)(1 + (1/\eta))+ \alpha \Ltt) }\\
    &\qquad\qquad\times\textstyle{\parenthese{ \norm{x_1 - x_2} + \int_{\rset^p} \abs{\pi_1[\kker(\cdot, y)] - \pi_2[\kker(\cdot, y)]} \rmd \mu(y)}. }
  \end{align}
  Combining this result and \rref{assum:general_kker}, we
  get that for any $x_1, x_2 \in \rset^d$ and  $\pi_1, \pi_2 \in \Pens_1(\rset^d)$
  \begin{equation}
  \label{eq:lip_b}
    \norm{b(x_1, \pi_1) - b(x_2, \pi_2)} \leq((\Mtt/\eta)(1 + (1/\eta))+ \alpha \Ltt)(1 + \Mtt) \parenthese{ \norm{x_1 - x_2} + \wassersteinD[1](\pi_1, \pi_2)},
  \end{equation}
  where the inequality follows using the dual representation of $\wassersteinD[1]$.
  For any $T \geq 0$ and any $\bfnu \in \rmc(\ccint{0,T}, \Pens_1(\rset^d))$ denote
  by $(\bfX_t^{\bfnu})_{t \in \ccint{0,T}}$ the unique strong solution to~\eqref{eq:nonlinearSDE}  with
  initial condition $\bfnu \in \Pens_1(\rset^d)$ (see \cite[Chapter 5, Theorem 2.9 and 2.5]{karatzas1991brownian}) given for any $t \in \ccint{0,T}$ by
  \begin{equation}
    \label{eq:xtmu}
    \textstyle{
      \bfX_t^{\bfnu} = \bfX_0^{\bfnu} + \int_0^t b(\bfX_s^{\bfnu}, \bfnu_s) \rmd s + \sqrt{2 \alpha} \bfB_t,
      }
    \end{equation}
    where $(\bfB_t)_{t \geq 0}$ is a $d$-dimensional Brownian motion with
    $\bfX_0^{\bfnu} = X_0 \in \rset^d$. Using \rref{assum:general_kker}
    and~\eqref{eq:lip_b}, there exists $C \geq 0$ such that for any
    $x \in \rset^d$ and $\pi \in \Pens_1(\rset^d)$
    \begin{equation}
      \label{eq:upper_brift}
    \norm{b(x, \pi)} \leq C \parenthese{1 + \norm{x}}.
  \end{equation}
  Denote $(\mathcal{F}_t)_{t \geq 0}$ the filtration associated with
  $(\bfB_t)_{t \geq 0}$. Using Jensen's inequality we have that
  $(\normLigne{\bfB_t})_{t \geq 0}$ is a
  $(\mathcal{F}_t)_{t \geq 0}$-supermartingale. Therefore, using \cite[Chapter
  2, Corollary 1.6]{revuz1999continuous}, we get that for any $t \geq 0$,
  \begin{equation}
    \textstyle{
      \expeLigne{\sup_{s \in \ccint{0,t}}\normLigne{\bfB_s}} \leq 2 \expesqLigne{\normLigne{\bfB_t}^2} \leq 2 (dt)^{1/2}.
      }
  \end{equation}
  Using this result and~\eqref{eq:upper_brift} we get that for any $t \geq 0$
  \begin{align}
    &\textstyle{
      \expeLigne{ \sup_{s \in \ccint{0,t}} \normLigne{\bfX_s^{\bfnu}}} }\\
      &\qquad\textstyle{\leq \expeLigne{\normLigne{\bfX_0^{\bfnu}}} + 2(2\alpha d)^{1/2} + (C+ 2(2 \alpha d)^{1/2}) t + C \int_0^t \expeLigne{\sup_{u \in \ccint{0,s}}\normLigne{\bfX_u^{\bfnu}}} \rmd s.
      }
  \end{align}
  Using Gr\"{o}nwall's lemma we get that
  $\expeLigne{\sup_{t \in \ccint{0,T}} \normLigne{\bfX_t^{\bfnu}}} <
  +\infty$. In particular, for any $t \in \ccint{0,T}$,
  $\bflambda^{\bfnu}_t \in \Pens_1(\rset^d)$, where $\bflambda^{\bfnu}$ is the
  distribution of $(\bfX_t^{\bfnu})_{t \in \ccint{0,T}}$. Similarly there exists
  $C \geq 0$ such that for any $t, s \in \ccint{0,T}$ we have
  \begin{equation}
    \wassersteinD[1](\bflambda^{\bfnu}_t, \bflambda^{\bfnu}_s) \leq \expe{\norm{\bfX_t^{\bfnu} - \bfX_s^{\bfnu}}} \leq C (t - s). 
  \end{equation}
  Therefore, $\bflambda^{\bfnu} \in \rmc(\ccint{0,T}, \Pens_1(\rset^d))$.  In
  addition, using~\eqref{eq:lip_b} we get that for any
  $\bfnu_1, \bfnu_2 \in \rmc(\ccint{0,T}, \Pens_1(\rset^d))$ and
  $t \in \ccint{0,T}$
  \begin{align}
    &\textstyle{
      \expeLigne{\normLigne{\bfX_t^{\bfnu_1} - \bfX_t^{\bfnu_2}}} \leq((\Mtt/\eta)(1 + (1/\eta))+ \alpha \Ltt)(1 + \Mtt)}\\
      &\qquad\qquad\times\textstyle{ \int_0^t \expeLigne{ \defEns{\normLigne{\bfX_s^{\bfnu_1}- \bfX_s^{\bfnu_2}} + \wassersteinD[1](\bfnu_{1,s}, \bfnu_{2,s})} \rmd s}.
      }
  \end{align}
  By Gr\"{o}nwall's lemma, there exists $C \geq 0$ such that for any
  $\bfnu_1, \bfnu_2 \in \rmc(\ccint{0,T}, \Pens_1(\rset^d))$ and $t \geq 0$
  \begin{multline}
    \label{eq:contrac}
    \expe{\norm{\bfX_t^{\bfnu_1} - \bfX_t^{\bfnu_2}}} \leq T ((\Mtt/\eta)(1 + (1/\eta))+ \alpha \Ltt)(1 + \Mtt) \\ \times \exp[T ((\Mtt/\eta)(1 + (1/\eta))+ \alpha \Ltt)(1 + \Mtt)] \sup_{s \in \ccint{0,T}}\wassersteinD[1](\bfnu_{1,s}, \bfnu_{2,s}). 
  \end{multline}
  Let
  $F_T: \ \rmc(\ccint{0,T}, \Pens_1(\rset^d)) \to \rmc(\ccint{0,T},
  \Pens_1(\rset^d))$ denote the map which associates to $\bfnu$ to law of $(\bfX_t^{\bfnu})_{t \in \ccint{0,T}}$, $\bflambda^{\bfnu}_t \in \Pens_1(\rset^d)$. Then there exists $T_0 \geq 0$ such that for any
  $T_0 \in \ccint{0,T}$ and
  $\bfnu_1, \bfnu_2 \in \rmc(\ccint{0,T}, \Pens_1(\rset^d))$
  \begin{equation}
    \sup_{t \in \ccint{0,T_0}} \wassersteinD[1](F_T(\bfnu_1)_t, F_T(\bfnu_2)_t) \leq (1/2) \sup_{t \in \ccint{0,T_0}} \wassersteinD[1](\bfnu_{1,t}, \bfnu_{2,t}). 
  \end{equation}
  Since $\Pens_1(\rset^d)$ is complete, see \cite[Theorem
  6.18]{villani2009optimal}, $F_T$ admits a unique fixed point using Picard's
  theorem (see, e.g., \cite[Theorem 1.7]{Carmona2016}). Denote $\bflambda^\star$ this fixed point and
  $(\bfX_t^\star)_{t \in \ccint{0,T}}$ the solution of~\eqref{eq:xtmu} with
  $\bfnu \leftarrow \bflambda^\star$. Then $(\bfX_t^\star)_{t \in \ccint{0,T}}$
  is a solution of~\eqref{eq:nonlinearSDE} up to time $T \geq 0$. By recursion,
  we obtain a solution $(\bfX_t^\star)_{t \geq 0}$.

  Assume that there exist two solutions $(\bfX_{1,t}^\star)_{t \geq 0}$ and
  $(\bfX_{2,t}^\star)_{t \geq 0}$ and denote $\bflambda_1$ and $\bflambda_2$
  their associated distribution. Since there exists a unique fixed point to
  $F_T$ we get that $\bflambda_1 = \bflambda_2$. Since there exists a unique
  strong solution to~\eqref{eq:xtmu}, we get that
  $(\bfX_{1,t}^\star)_{t \geq 0} = (\bfX_{2,t}^\star)_{t \geq 0}$, which
  concludes the proof.

\subsection{Proof of \Cref{prop:la_convergence_star}}
\label{app:proof_invariant}
Under \tup{\rref{assum:general_kker}}, Proposition~\ref{prop:g_prop} ensures that the functional $\Gun^\eta$ is convex and lower bounded.
In addition, \tup{\rref{assum:pi0}} guarantees that $U$ is sufficiently regular to satisfy \cite[Assumption 2.2]{hu2019mean}.
To see this, we can use \tup{\rref{assum:lip_U}-\ref{item:dissi}} and the Cauchy-Schwarz inequality to obtain
\begin{align}
    \langle \nabla U(x) , x \rangle & = \langle \nabla U(x) -\nabla U(0), x \rangle + \langle \nabla U(0) , x \rangle\geq \mtt\norm{x}^2-\ctt - \norm{x}\norm{\nabla U(0)}.
\end{align}
Then Young's inequality yields, $\alpha\beta \leq \alpha^2/(2\epsilon)+\beta^2\epsilon/2$ for all $\epsilon>0$, and thus that:
\begin{align}
    -\norm{x}\norm{\nabla U(0)} \geq -\norm{x}^2/(2\epsilon)-\epsilon \norm{\nabla U(0)}^2/2.
\end{align}
It follows that for all $\epsilon > 1/(2\mtt)$ we have, for any $x \in \rset^d$, $\langle \nabla U(x) , x \rangle \geq \mathtt{a} \norm{x}^2 + \mathtt{b}$, where $\mathtt{a}:= \mtt-1/(2\epsilon) >0, \mathtt{b}=-\ctt-\epsilon \norm{\nabla U(0)}^2/2$.

As shown in~\eqref{eq:lip_b}, the drift of the MKVSDE is Lipschitz continuous.

In addition, using the fact that, under \tup{\rref{assum:general_kker}}, $\kker\in\rmc^\infty(\rset^d \times \rset^d, [0, +\infty))$ and Leibniz integral rule for differentiation under the
integral sign (e.g. \cite[Theorem 16.8]{billingsley1995measure}), we have that $b(x, \nu)$ in~\eqref{eq:b} is $\rmc^\infty(\rset^d , \rset^d)$ for all fixed $\nu\in\Pens(\rset^d)$.
Finally, we need to show that $\nabla b(x, \nu)$ is jointly continuous in $(x, \nu)$.
Similarly to~\eqref{eq:lipo_lip}, we have
  for any $x_1,x_2 \in \rset^d$ and $\nu_1, \nu_2 \in \Pens_2(\rset^d)$
  \begin{align}
  \label{eq:a3_part_one}
    &\normLigne{\nabla b(x_1, \nu_1) - \nabla b(x_2, \nu_2)} \\
    &\quad \textstyle{ \leq \normLigne{\int_{\rset^p} \nabla_1^2 \kker(x_1,y) / (\nu_1[\kker(\cdot, y)] + \eta)  - \int_{\rset^p} \nabla_1^2 \kker(x_2,y) / (\nu_2[\kker(\cdot, y)] + \eta)}} \\
    & \qquad +\alpha \textstyle{ \normLigne{\nabla_1^2 U(x_1) - \nabla_1^2 U(x_2) }}\\
    &\quad \leq (\Mtt/\eta)(1+(1/\eta))(1+ \alpha\Ltt_2)(1+\Mtt) \{\normLigne{x_1 - x_2} + \wassersteinD[2](\nu_1, \nu_2)\},  
  \end{align}
  which gives continuity. Then, the result follows directly from \cite[Theorem
2.11]{hu2019mean}.
  \subsection{Proof of \Cref{prop:propagation_chaos}}
  \label{sec:proof-prop-chaos}

  Let $N \in \nsets$.  The existence and strong uniqueness of a solution to~\eqref{eq:particle} is a straightforward consequence of~\eqref{eq:lipo_lip},
  \cite[Chapter 5, Theorem 2.9 and 2.5]{karatzas1991brownian} and the fact that for any
  $\{x_1^{k,N}\}_{k=1}^N, \{x_2^{k,N}\}_{k=1}^N \in (\rset^d)^N$ and
  $\ell \in \{1, \dots, N\}$ we have
  \begin{align}
  \label{eq:lipschitz_bmketa}
    &\textstyle{\normLigne{\int_{\rset^p} \defEnsLigne{\bmketa\parentheseLigne{x_1^{\ell,N}, (1/N) \sum_{k=1}^N \updelta_{x_1^{k,N}}, y} - \bmketa\parentheseLigne{x_2^{\ell,N}, (1/N) \sum_{k=1}^N \updelta_{x_2^{k,N}}, y} } \rmd \mu(y)}} \\ 
    &\quad\textstyle{\leq ((\Mtt/\eta)(1 + (1/\eta))+ \alpha \Ltt)}\\
    &\quad\quad\textstyle{ \times\parentheseLigne{ \normLigne{x_1^{\ell, N} - x_2^{\ell, N}} + (1/N) \sum_{k=1}^N \int_{\rset^p} \absLigne{\kker(x_1^{k,N}, y) -\kker(x_2^{k,N}, y)} \rmd \mu(y)}} \\
    &\quad\textstyle{\leq 2 ((\Mtt/\eta)(1 + (1/\eta))+ \alpha \Ltt)  \normLigne{x_1^{1:N} - x_2^{1:N}}. }
  \end{align}
  We now turn to the quantitative functional propagation of chaos result.  For
  any $x \in \rset^d$ and $\pi \in \Pens_1(\rset^d)$ denote
  $b(x, \pi) = \int_{\rset^p} \bmketa(x, \pi, y) \rmd \mu(y) - \alpha \nabla
  U(x)$. We recall that using~\eqref{eq:lipo_lip} there exists $C \geq 0$
  such that for any $x_1, x_2 \in \rset^d$ and
  $\pi_1, \pi_2 \in \Pens_1(\rset^d)$ we have
  \begin{equation}
    \textstyle{
    \norm{b(x_1, \pi_1) - b(x_2, \pi_2)} \leq C \defEns{\norm{x_1 - x_2} +\int_{\rset^p} \abs{\pi_1[\kker(\cdot, y)] - \pi_2[\kker(\cdot, y)]} \rmd \mu(y)}. }
  \end{equation}
Using this result,  we have for any $t \geq 0 $
  \begin{align}
  \label{eq:poc_proof}
    &\textstyle{\expeLigne{\sup_{s \in \ccint{0,t}}\normLigne{\bfX_s^\star - \bfX_s^{1,N}}} \leq \int_{0}^t \expeLigne{\normLigne{b(\bfX_s^\star, \bflambda_s^\star) - b(\bfX_s^{1, N}, \bflambda^N_s)}} \rmd s}\\
    &\quad \textstyle{ \leq C \int_{0}^t \expeLigne{\sup_{u \in \ccint{0,s}}\normLigne{\bfX_u^\star - \bfX_u^{1,N}}}\rmd s }\\
    &\quad\textstyle{+  C\int_{0}^t \int_{\rset^p} \expeLigne{\absLigne{(1/N)\sum_{k=1}^N \kker(\bfX_s^{k,N}, y) - \lambdabf_s^\star[\kker(\cdot, y)]}} \rmd \mu(y) \rmd s} .
  \end{align}
  Now, consider $N$ independent copies of the nonlinear process $\bfX_s^{\star}$, $\{(\bfX_s^{k,\star})_{t \geq 0}\}_{k=1}^N$. We can bound the second term in the above with
    \begin{align}
    & \expeLigne{\absLigne{(1/N)\sum_{k=1}^N \kker(\bfX_s^{k,N}, y) - \lambdabf_s^\star[\kker(\cdot, y)]}} \leq (1/N)  \expeLigne{\absLigne{ \sum_{k=1}^N \defEns{\kker(\bfX_s^{k,N}, y) - \kker(\bfX_s^{k,\star}, y) }}} \\
    &\qquad\qquad+(1/N)  \expeLigne{\absLigne{ \sum_{k=1}^N \defEns{\kker(\bfX_s^{k,\star}, y) - \lambdabf_s^\star[\kker(\cdot, y)]}}} \\
     &\qquad\qquad\leq \Mtt\expeLigne{\sup_{u \in \ccint{0,s}}\normLigne{\bfX_u^\star - \bfX_u^{1,N}}} +(1/N)  \expeLigne{\absLigne{ \sum_{k=1}^N \defEns{\kker(\bfX_s^{k,\star}, y) - \lambdabf_s^\star[\kker(\cdot, y)]}}} ,
  \end{align}
  where we used the Lipschitz continuity of $\kker$ and the fact that $\{(\bfX_t^{k,N})_{t \geq 0}\}_{k=1}^N$ is exchangeable to obtain the last inequality.
  Plugging the above into~\eqref{eq:poc_proof} and using Jensen's inequality we obtain
    \begin{align}
    &\textstyle{\expeLigne{\sup_{s \in \ccint{0,t}}\normLigne{\bfX_s^\star - \bfX_s^{1,N}}} \leq }\\
    &\quad \textstyle{ \leq C (1 + \Mtt) \int_{0}^t \expeLigne{\sup_{u \in \ccint{0,s}}\normLigne{\bfX_u^\star - \bfX_u^{1,N}}}\rmd s }\\
    &\quad \textstyle{ + C (1 + \Mtt)(1/N) \int_0^t \int_{\rset^p} \expesq{ \sum_{k=1}^N \parenthese{\kker(\bfX_s^{k,\star}, y) - \lambdabf_s^\star[\kker(\cdot, y)]}^2} \rmd\mu(y)\rmd s} \\
    &\quad \textstyle{ \leq C (1 + \Mtt) \{\int_{0}^t \expeLigne{\sup_{u \in \ccint{0,s}}\normLigne{\bfX_u^\star - \bfX_u^{1,N}}}\rmd s + \sqrt{2} \Mtt t N^{-1/2} \}. }
  \end{align}
  
  Using Gr\"{o}nwall's lemma we get that for any $T \geq 0$ there exists
  $C_T \geq 0$ such that for any $N \in \nset$
\begin{equation}
  \textstyle{\expeLigne{\sup_{t \in \ccint{0,T}} \normLigne{\bfX_t^\star - \bfX_t^{1,N}}} \leq C_T N^{-1/2}, }
\end{equation}
which concludes the proof.

\subsection{Proof of \Cref{prop:particle_ergodic}}
\label{proof:prop:particle_ergodic}

First, we recall that $b$ and $\tilde{b}$ are given, for any $x \in \rset^d$, 
$\nu \in \Pens(\rset^d)$ and $y \in \rset^p$, by
\begin{equation}
  \textstyle{
    b(x, \nu) = \int_{\rset^p}\bmketa(x, \nu, y) \rmd \mu(y) - \alpha \nabla U(x), \quad \bmketa(x, \nu, y) = -\nabla_1 \kker(x, y) / (\nu[\kker(\cdot, y)] + \eta).
    }
\end{equation}
Using \tup{\rref{assum:general_kker}} and~\eqref{eq:lipo_lip}, there exist $C_0, C_1 \geq 0$ such that for any
$x_1, x_2 \in \rset^d$, $\pi_1, \pi_2 \in \Pens(\rset^d)$ and $y \in \rset^p$,
\begin{equation}
  \label{eq:big}
 \textstyle{ \normLigne{\bmketa(x_1, \pi_1, y)}  \leq C_0, \qquad   
     \norm{b(x_1, \pi_1) - b(x_2, \pi_2)} \leq C_1 \parenthese{ \norm{x_1 - x_2} + \wassersteinD[1](\pi_1, \pi_2)}. }
 \end{equation}
In addition, note that for any $x_1^{1:N},x_2^{1:N} \in (\rset^d)^N$ we have for any $i \in \{1,2\}$
 \begin{equation}
   \label{eq:w1_empi}
   \textstyle{\wassersteinD[1](\mu_1^N, \mu_2^N) \leq (1/N) \sum_{k=1}^N \normLigne{x_1^{k,N} - x_2^{k,N}}, \qquad \mu_i^N = (1/N) \sum_{k=1}^N \updelta_{x_i^{k,N}}. }
 \end{equation}
Let $N \in \nset$ and denote
$B_N: (\rset^d)^N \to (\rset^d)^N$ given for any $x^{1:N} \in (\rset^d)^N$ by
\begin{equation}
  \textstyle{B_N(x^{1:N}) = \left(b(x^{k,n}, \mu^N)\right)_{k =1}^{N}, \qquad \mu^N = (1/N) \sum_{k=1}^N \updelta_{x^{k,N}}. }
\end{equation}
Using~\eqref{eq:big},~\eqref{eq:w1_empi} and the Cauchy-Schwarz
inequality we have for any $x^{1:N}_1, x_2^{1:N} \in (\rset^d)^N$ 
\begin{align}
  \label{eq:lip_N}
 \norm{B_N(x^{1:N}_1) - B_N(x^{1:N}_2)}
&\leq C_1\left(\sum_{k=1}^N\norm{x_1^{k,N}-x_2^{k,N}}^2+N\wassersteinD[1]\left(\frac{1}{N} \sum_{k=1}^N \updelta_{x_1^{k,N}}, \frac{1}{N} \sum_{k=1}^N \updelta_{x_2^{k,N}}\right)^2\right.\\
&\qquad\qquad\left.+2\sum_{k=1}^N\norm{x_1^{k,N}-x_2^{k,N}}\wassersteinD[1]\left(\frac{1}{N} \sum_{k=1}^N \updelta_{x_1^{k,N}}, \frac{1}{N} \sum_{k=1}^N \updelta_{x_2^{k,N}}\right)\right)^{1/2}\\
& \leq C_1\left(\sum_{k=1}^N\norm{x_1^{k,N}-x_2^{k,N}}^2+ \frac{1}{N}\left(\sum_{k=1}^N \norm{x_1^{k,N} - x_2^{k,N}}\right)^2\right.\notag\\
&\qquad\qquad\left.+\frac{2}{N}\sum_{k=1}^N\norm{x_1^{k,N}-x_2^{k,N}}\sum_{k=1}^N \norm{x_1^{k,N} - x_2^{k,N}}\right)^{1/2}\\
&\leq 2C_1\sum_{k=1}^N\norm{x_1^{k,N}-x_2^{k,N}}\\
 & \leq 2C_1N^{1/2} \norm{x_1^{1:N} - x_2^{1:N}}.
\end{align}

Let $R = \max(4C_0N^{1/2}/(\alpha \mtt), (2 \ctt N/ \mtt)^{1/2})$ with $\mtt$ and $\ctt$ as in \tup{\rref{assum:lip_U}-\ref{item:dissi}}. 
Using~\eqref{eq:big} and the Cauchy-Schwarz inequality we have for
any $x^{1:N}_1, x_2^{1:N} \in (\rset^d)^N$ 
\begin{align}
  \langle B_N(x^{1:N}_1) - B_N(x^{1:N}_2), x_1^{1:N} - x_2^{1:N} \rangle 
  &\leq -\alpha \mtt \norm{x_1^{1:N} - x_2^{1:N}}^2 + \alpha \ctt N + 2C_0 \sum_{k=1}^N \norm{x_1^{k,N} - x_2^{k,N}}\\
 &\leq -\alpha \mtt \norm{x_1^{1:N} - x_2^{1:N}}^2 + \alpha \ctt N + 2C_0 N^{1/2} \norm{x_1^{1:N} - x_2^{1:N}}.
\end{align}
Then for any $x^{1:N}_1, x_2^{1:N} \in (\rset^d)^N$ with $\normLigne{x_1^{1:N} - x_2^{1:N}} \geq R$ we further have
\begin{align}
  \label{eq:diss_N}
  \langle B_N(x^{1:N}_1) - B_N(x^{1:N}_2), x_1^{1:N} - x_2^{1:N} \rangle 
  &\leq  -(\alpha \mtt/2)  \norm{x_1^{1:N} - x_2^{1:N}}^2.
\end{align}
We conclude upon combining~\eqref{eq:lip_N},~\eqref{eq:diss_N} and 
\cite[Corollary 2]{debortoli2019convergence}.

\subsection{Proof of \Cref{prop:particles_min}}
\label{proof:prop:particles_minapp:ps}

First, we start with the following lemma.

\begin{lemma}
  \label{lemma:stabi_init}
  Assume \tup{\rref{assum:general_kker}}, \tup{\rref{assum:lip_U}-\ref{item:lip}} and \tup{\rref{assum:lip_U}-\ref{item:dissi}}. Let $\nu_1, \nu_2 \in \Pens_1(\rset^d)$. Denote
  $(\bfX_{1, t}^\star, \bfX_{2, t}^\star)_{t \geq 0}$ the stochastic process  such that $(\bfX_{1, 0}^\star, \bfX_{2, 0}^\star)$ is the optimal coupling
  between $\nu_1$ and $\nu_2$ w.r.t. to $\wassersteinD[1]$. Then for any
  $T \geq 0$, there exists $C_T \geq 0$ such that for any $t \in \ccint{0,T}$
  \begin{equation}
   \wassersteinD[1](\bflambda_t^\star(\nu_1), \bflambda_t^\star(\nu_2)) \leq C_T \wassersteinD[1](\nu_1, \nu_2),
  \end{equation}
  where for any $\nu \in \Pens_1(\rset^d)$, $\bflambda^\star(\nu)$ is the
  distribution of $(\bfX_t^\star)_{t \geq 0}$ with initial distribution $\nu$.
\end{lemma}

\begin{proof}
  Let $T \geq 0$. Using~\eqref{eq:lipo_lip}, we have for any $t \in \ccint{0,T}$
  \begin{equation}
    \textstyle{
    \expeLigne{\normLigne{\bfX_{1,t}^\star - \bfX_{2,t}^\star}}  \leq \wassersteinD[1](\nu_1, \nu_2)  + C \int_0^t \defEnsLigne{\expeLigne{\normLigne{\bfX_{1,s}^\star - \bfX_{2,s}^\star}} + \wassersteinD[1](\bflambda_s^\star(\nu_1), \bflambda_s^\star(\nu_2))}  \rmd s,}
  \end{equation}
  where $C = (\Mtt/\eta)(1 + (1/\eta))(1 + \alpha \Ltt)(1 + \Mtt)$. Using
  Gr\"{o}nwall's lemma we have for any $t \in \ccint{0,T}$
  \begin{equation}
    \textstyle{
      \expeLigne{\normLigne{\bfX_{1,t}^\star - \bfX_{2,t}^\star}} \leq C(1+T) e^{C(1+T)} \defEnsLigne{\wassersteinD[1](\nu_1, \nu_2) + C \int_0^t \wassersteinD[1](\bflambda_s^\star(\nu_1), \bflambda_s^\star(\nu_2)) \rmd s}.
      }
\end{equation}
Using that for any $t \geq 0$,
$\wassersteinD[1](\bflambda_t^\star(\nu_1), \bflambda_t^\star(\nu_2)) \leq \expe{\norm{\bfX_{1,t}^\star - \bfX_{2,t}^\star}}$ and
Gr\"{o}nwall's lemma we get that there exists $C_T \geq 0$ such that for any $t \in \ccint{0,T}$
  \begin{equation}
   \wassersteinD[1](\bflambda_t^\star(\nu_1), \bflambda_t^\star(\nu_2)) \leq C_T \wassersteinD[1](\nu_1, \nu_2),
 \end{equation}
 which concludes the proof.
\end{proof}

\begin{proof}[Proof of Proposition \ref{prop:particles_min}]
  Let $\alpha, \eta >0$. First we show that
  $\ensembleLigne{\pi^{1, N}}{N \in \nset}$ is relatively compact in
  $\Pens_1(\rset^d)$.  Let $N \in \nset$ and assume that $\bfX_0^{1:N} = 0$.
Let us define for any $t \geq 0$  
    \begin{equation}
    \textstyle{
    \bfM_t^N = \frac{1}{2}\normLigne{\bfX_t^{1,N}}^2 -\int_0^t \defEnsLigne{\langle \bfX_s^{1,N}, b(\bfX_s^{1,N}, \bflambda^N_s) \rangle + \alpha d} \rmd s,  }
  \end{equation}
where $\{(\bfX_t^{k,N})_{t \geq 0}\}_{k=1}^N$ is given in \eqref{eq:particle}.
Using It\^{o}'s formula we have that 
\begin{align}
    \frac{1}{2}\norm{\bm{X}_t^{1,N}}^2 = \frac{1}{2}\norm{\bm{X}_s^{1,N}}^2 +\int_s^t \langle \bm{X}_u^{1,N}, b(\bm{X}_u^{1,N}, \pi^N_u) \rangle \rmd u +d\alpha (t-s).
\end{align}  
Let us denote by $(\mathcal{F}_t^N)_{t \geq 0}$ the filtration associated with
  $\{(\bfB_t^k)_{t \geq 0}\}_{k=1}^N$.
Then, $\mathbb{E}\left[\bfM_t^N\mid \mathcal{F}_s^N\right] = \bfM_s^N$, showing that $\bfM_t^N$ is a $\mathcal{F}_t^N$-martingale.
Since $\mathbb{E}\left[\bfM_t^N -\bfM_s^N \mid \mathcal{F}_s^N\right]=0$, we also have that 
\begin{align}
\frac{1}{2}\mathbb{E}\left[\norm{\bm{X}_t^{1,N}}^2\right] - \frac{1}{2}\mathbb{E}\left[\norm{\bm{X}_s^{1,N}}^2\right] =&\mathbb{E}\left[\int_s^t \left\lbrace\langle \bm{X}_u^{1,N}, b(\bm{X}_u^{1,N}, \bflambda^N_u) \rangle + \alpha d\right\rbrace\rmd u\right]\\
=&\mathbb{E}\left[\int_s^t \left\langle \bm{X}_u^{1,N}, \int_{\rset^d}\bmketa(\bm{X}_u^{1,N}, \bflambda^N_u, y)) \rmd \mu(y)\right\rangle \rmd u\right]\\
&-\mathbb{E}\left[\int_s^t \left\lbrace\langle \bm{X}_u^{1,N}, \alpha\nabla U(\bm{X}_u^{1,N}) \rangle - \alpha d\right\rbrace\rmd u\right],
\end{align}
where we used the definition of $b$ in~\eqref{eq:b}.
 Note that using
  \tup{\rref{assum:lip_U}-\ref{item:dissi}} with $x_1 = x$ and $x_2 =0$, we have
  for any $x \in \rset^d$
  \begin{equation}
    \langle \nabla U(x), x \rangle \geq \langle \nabla U(0), x \rangle + \mtt \normLigne{x}^2 - \ctt \eqsp . 
  \end{equation}
  Therefore, using this result, \eqref{eq:b} and the Cauchy-Schwarz inequality
  we obtain that for any $t \geq 0$
  \begin{align}
\textstyle{\frac{1}{2}\expeLigne{\normLigne{\bfX_t^{1,N}}^2} - \frac{1}{2}\expeLigne{\normLigne{\bfX_s^{1,N}}^2}} \leq& \int_s^t \int_{\rset^d} \expeLigne{\normLigne{\bmketa(\bfX_u^{1,N}, \bflambda^N_u, y)}\normLigne{\bfX_u^{1,N}}} \rmd \mu(y)\rmd u\\
&+\int_s^t \defEnsLigne{
 -\alpha \mtt \expeLigne{\normLigne{\bfX_u^{1,N}}^2} + \alpha\normLigne{\nabla U(0)}\expeLigne{\normLigne{\bfX_u^{1,N}}} + \alpha \ctt + \alpha d} \rmd u, 
  \end{align}
where the last line follows from Tonelli's Theorem (e.g. \cite[Theorem 18.3]{billingsley1995measure}) since all integrated functions are positive (or always negative, in which case we can consider minus the integral itself).

Let $\mathscr{V}_t^N =\expeLigne{\normLigne{\bfX_t^{1,N}}^2}$. Appealing to the fundamental theorem of calculus  we get that
\begin{align}
\frac{1}{2}\rmd \mathscr{V}_t^N / \rmd t \leq& \mathbb{E}\left[ \int_{\rset^d} \norm{\bmketa(\bm{X}_t^{1,N}, \bflambda^N_u, y)}\norm{\bm{X}_t^{1,N}} \rmd \mu(y) \right]\\
  & -\alpha \mtt \mathbb{E}\left[\norm{\bm{X}_t^{1,N}}^2\right] + \alpha\norm{\nabla U(0)}\mathbb{E}\left[\norm{\bm{X}_t^{1,N}}\right] + \alpha \ctt + \alpha d.
\end{align}
Using that for any
$x \in \rset^d$ and $\pi \in \Pens(\rset^d)$ we have
$\normLigne{\bmketa(x, \pi, y)} \leq \Mtt / \eta$, we get that
\begin{align}
\frac{1}{2}\rmd \mathscr{V}_t^N / \rmd t \leq& \Mtt/\eta (\mathscr{V}_t^N)^{1/2}
    -\alpha \mtt \mathscr{V}_t^N + \alpha\norm{\nabla U(0)}  (\mathscr{V}_t^N)^{1/2}  + \alpha \ctt + \alpha d\\
      =& (\Mtt/\eta + \alpha \norm{\nabla U(0)}) (\mathscr{V}_t^N)^{1/2} + \alpha \ctt + \alpha d - \alpha \mtt \mathscr{V}_t^N  
\end{align}
%
Noting that for any $\mathsf{a}, \mathsf{b}$ such that $\mathsf{ab} \geq 1/2$, for any $x \geq 0$ we have $\sqrt{x} \leq \mathsf{a} + \mathsf{b} x$, and setting $\mathsf{a} = (\Mtt/\eta  + \alpha \norm{\nabla U(0)})/ \alpha\mtt$ and $\mathsf{b} = 1/2\mathsf{a}$ we have:
\begin{align}
  \frac{1}{2}\rmd \mathscr{V}_t^N / \rmd t 
     \leq& \frac{(\Mtt/\eta  + \alpha \norm{\nabla U(0)})^2}{\alpha \mtt} + \alpha \ctt + \alpha d - \frac12 \alpha \mtt \mathscr{V}_t^N 
\end{align}
Hence, for any $t \geq 0$ and any $N \in \nset$ we get that 
$\mathscr{V}_t^N \leq C$ with
\begin{equation}
  C = 2 \left[ \frac{(\Mtt/\eta  + \alpha \norm{\nabla U(0)})^2}{\alpha \mtt} + \alpha \ctt + \alpha d\right] / \alpha\mtt
\end{equation}
Therefore, letting $t \to +\infty$ we get that for any $N \in \nset$,
$\int_{\rset^d} \norm{\tilde{x}}^2 \rmd \pi^{1,N}(\tilde{x}) \leq C$. Hence,
$\ensembleLigne{\pi^{1, N}}{N \in \nset}$ is relatively compact in
$\Pens_1(\rset^d)$ using \cite[Proposition 7.1.5]{ambrosio2008gradient}.

Let $\pi^\star$ be a cluster point of $\ensembleLigne{\pi^{1, N}}{N \in
  \nset}$. Let $(N_k)_{k \in \nset}$ be an increasing sequence such that
$\lim_{k \to +\infty} \wassersteinD[1](\pi^{1,N_k}, \pi^\star) = 0$. We have that for any $t \geq 0$
\begin{align}
  \label{eq:la_grosse_borne}
  &\wassersteinD[1](\pi^\star, \pi_{\alpha, \eta}^\star) \leq \wassersteinD[1](\pi^\star, \pi^{1, N_k})   \\ 
  &\qquad\qquad + \wassersteinD[1](\pi^{1, N_k}, \bflambda_t^{1,N_k}(\pi^{N_k}))+ \wassersteinD[1](\bflambda_t^{1,N_k}(\pi^{N_k}), \bflambda_t^{\star}(\pi^{1, N_k})) \\
  &\qquad\qquad+ \wassersteinD[1](\bflambda_t^{\star}(\pi^{1, N_k}), \bflambda_t^{\star}(\pi^{\star})) + \wassersteinD[1](\bflambda_t^{\star}(\pi^{\star}), \pi_{\alpha, \eta}^\star). 
\end{align}
We now control each of these terms. Let $\vareps >0$ and set $t \geq 0$ such
that, using \Cref{prop:la_convergence_star}, 
$\wassersteinD[1](\bflambda_t^{\star}(\pi^{\star}), \pi_{\alpha, \eta}^\star)
\leq \vareps$. Using
\Cref{lemma:stabi_init}, there exists $k_0 \in \nset$ such that for any
$k \geq k_0$ we have
$\wassersteinD[1](\bflambda_t^{\star}(\pi^{1, N_k}),
\bflambda_t^{\star}(\pi^{\star})) \leq \vareps$. Using
\Cref{prop:propagation_chaos}, there exists $k_1 \in \nset$ such that for any
$k \geq k_1$ we have that
$\wassersteinD[1](\bflambda_t^{1,N_k}(\pi^{N_k}), \bflambda_t^{\star}(\pi^{1,
  N_k})) \leq \vareps$. Since $\pi^{N_k}$ is invariant for
$(\bfX_t^{1:N})_{t \geq 0}$ we get that
$\wassersteinD[1](\pi^{1, N_k}, \bflambda_t^{1,N_k}(\pi^{N_k})) = 0$. Finally,
there exists $k_2 \in \nset$ such that for any $k \geq k_2$ we have
$\wassersteinD[1](\pi^\star, \pi^{1, N_k}) \leq \vareps$.  Combining these
results in~\eqref{eq:la_grosse_borne}, we get that
$\wassersteinD[1](\pi^\star, \pi_{\alpha, \eta}^\star) \leq 4 \vareps$. Therefore, since $\vareps > 0$ is arbitrary, we have that $\pi^\star = \pi_{\alpha, \eta}^\star$, which concludes the proof.
\end{proof}


\section{Proofs of \Cref{sec:implementation}}

\subsection{Basics on Lions derivatives}
We briefly recall here the definition of Lions derivative of functionals defined
over $\Pens_2(\rset^d)$.  In this section, we fix an underlying probability
space $(\Omega, \mathcal{F}, \mathbb{P})$ and denote by $\mathrm{L}^2(\Omega, \mathcal{F}, \mathbb{P};\rset^d)$ the space of $\rset^d$-valued, $\mathcal{F}$-measurable random variables with finite second moment. We refer to
\cite{carmona2018probabilistic, cardaliaguet2010notes,hammersley2021mckean,buckdahn2017mean,bao2020first}
for a thorough exposition of Lions derivatives with applications to mean-field
games. The following definitions can be found in
\cite{bao2020first, hammersley2021mckean}.  We start by introducing the
notion of $L$-differentiability.

\begin{definition}
  Let $\Mun: \ \Pens_2(\rset^d) \to \rset$ and
  $\bar{\Mun}: \mathrm{L}^2(\Omega, \mathcal{F}, \mathbb{P};\rset^d)\to \rset$ such that for any
  $X \in \mathrm{L}^2(\Omega, \mathcal{F}, \mathbb{P};\rset^d)$,
  $\bar{\Mun}(X) = \Mun(\mathcal{L}(X))$, where $\mathcal{L}(X)$ is the
  distribution of $X$. $\bar{\Mun}$ is called a lifted version of $\Mun$. We
  say that $\Mun$ is $L$-differentiable at $\nu_0\in\Pens_2(\rset^d)$ if
  there exists a random variable $X_0$ such that $\mathcal{L}(X_0) = \nu_0$ and
  $\bar{\Mun}$ is Fr{\'e}chet differentiable at $X_0$, \ie \ there exists
  $Y \in \mathrm{L}^2(\Omega, \mathcal{F}, \mathbb{P},\rset^d)$ such that
  \begin{equation}
    \lim_{\normLigne{X - X_0}_2 \to 0} \norm{\bar{\Mun}(X) - \bar{\Mun}(X_0) - \langle Y, X- X_0 \rangle} / \normLigne{X - X_0}_2 = 0. 
  \end{equation}
  If $\Mun$ is $L$-differentiable at $\nu_0$ for all
  $\nu_0 \in \Pens_2(\rset^d)$ then we say that $\Mun$ is $L$-differentiable.
\end{definition}

If $\bar{\Mun}$ is differentiable at $X_0$ then there exists
$\xi: \ \rset^d \to \rset^d$ measurable and uniquely defined almost everywhere such that $Y = \xi(X_0)$. Note that
$\xi$ does not depend on $X_0$ except via $\mathcal{L}(X_0)$. Hence $\xi$ does not depend on the choice of $X_0$.

\begin{definition}
  \label{def:basics-lions-deriv}
  Let $\Mun: \ \Pens_2(\rset^d) \to \rset$ be $L$-differentiable at $\mu$. Then we write $\rmD^L\Mun(\mu) = \xi$. Moreover we have $\rmD^L \Mun: \ \Pens_2(\rset^d) \times \rset^d\to \rset^d$ given by $\rmD^L\Mun(\mu, z) := \rmD^L\Mun(\mu)(z)$.
\end{definition}

We emphasize that $\rmD^L$ is well-defined since $\xi$ only depends on
$\mathcal{L}(X_0)$.  A mapping $\Mun: \ \Pens(\rset^d) \to \rset^p$ is said to
be $L$-differentiable if for any $i \in \{1, \dots, p\}$, $\Mun_i$ is
$L$-differentiable.  
We present here a simple example of $L$-derivative which is relevant for the proof of \Cref{lemma:derivatives} below.

\begin{example}
\label{ex:l_derivative}
Consider $\Mun:\ \Pens(\rset^d) \times \rset^p\to \rset$ given by $\Mun(\nu, y) = \frac{1}{\nu\left[\kker(, y)\right]+\eta}$ with $\kker$ satisfying \rref{assum:general_kker} and $\eta>0$. We have $\bar{\Mun}(X, y)= \frac{1}{\mathbb{E}\left[\kker(X, y)\right]+\eta}$ with $X\sim \nu$ and we have for any $H \in \rmL^2(\Omega, \rset^d)$
\begin{align}
   & \bar{\Mun}(X+H) - \bar{\Mun}(X) - \left\langle \frac{\nabla_1 \kker(X, y)}{\mathbb{E}\left[\kker(X, y)\right]+\eta}, H \right\rangle \\
    &\qquad\qquad=\frac{1}{\mathbb{E}\left[\kker(X+H, y)\right]+\eta}-\frac{1}{\mathbb{E}\left[\kker(X, y)\right]+\eta}- \left\langle \frac{\nabla_1 \kker(X, y)}{\mathbb{E}\left[\kker(X, y)\right]+\eta}, H \right\rangle \\
    &\qquad\qquad=\frac{\mathbb{E}\left[\kker(X, y)\right]-\mathbb{E}\left[\kker(X+H, y)\right]- \langle \nabla_1 \kker(X, y), H \rangle }{\left(\mathbb{E}\left[\kker(X, y)\right]+\eta\right)\left(\mathbb{E}\left[\kker(X+H, y)\right]+\eta\right)}.
\end{align}
Now, since $\mathbb{E}\left[\kker(X+H, y)\right]+\eta\geq \eta>0$ and $\mathbb{E}\left[\kker(X, y)\right]+\eta\geq\eta>0$ and, as shown in \cite[Example 2]{bao2020first} the $L$-derivative of $\mathcal{Q}(\nu, y) = \nu\left[\kker(, y)\right]$ is $\rmD^L\mathcal{Q}(\nu, y)(x)=\nabla_1 \kker(\cdot, y)(x)$, we have
  \begin{equation}
    \rmD^L\Mun(\nu, y)(x) = -  \frac{\nabla_1 \kker(x, y)}{\nu\left[\kker(\cdot, y)\right]+\eta}.
  \end{equation}
\end{example}

Similarly to \Cref{def:basics-lions-deriv}, we introduce
the higher-order Lions derivatives (see, e.g., \cite[Appendix A]{bao2020first} or \cite[Appendix A.2]{hammersley2021mckean}).
\begin{definition}
  Let $\Mun: \ \Pens_2(\rset^d) \to \rset$ be $L$-differentiable. For any
  $x \in \rset^d$ and $i \in \{1, \dots, d\}$, define
  $g_x^i: \ \Pens(\rset^d) \to \rset$ such that for any $x \in \rset^d$,
  $i \in \{1, \dots, d\}$ and $\mu \in \Pens_2(\rset^d)$ we have
  $g_x^i(\mu) = \rmD^L \Mun(\mu)(x)_i$. If for any $x \in \rset^d$ and
  $i \in \{1, \dots, d\}$, $g_x^i$ is $L$-differentiable then we say that $\Mun$
  is twice $L$-differentiable and we define
  $(\rmD^L)^2 \Mun: \ \Pens_2(\rset^d)\times \rset^d\times \rset^d \to \rset^d \times \rset^d$ given for any $\mu \in \Pens_2(\rset^d)$ and $x_0, x_1 \in \rset^d$ by
  \begin{equation}
    (\rmD^L )^2 \Mun(\mu)(x_0,x_1) = (\rmD^L(g_{x_0}^i)(\mu)(x_1))_{i \in \{1, \dots, d\}}. 
  \end{equation}
\end{definition}

Finally, we define the class
$\rmc^{2, (2,1)}(\rset^d \times \Pens_2(\rset^d), \rset^p)$ used to establish
the main strong approximation results in \cite{bao2020first}.
An equivalent definition for $\rset$-valued functionals $\Mun$ is given in \cite[Definition A.2]{bao2020first}.

\begin{definition}
  Let $\Mun: \ \rset^d \times \Pens_2(\rset^d) \to \rset^p$. We say that
  $\Mun \in \rmc^{2, (2,1)}(\rset^d \times \Pens_2(\rset^d), \rset^p)$ if the following hold:
  \begin{enumerate}[wide, labelindent=0pt, label=(\alph*)]
  \item For any $\mu \in \Pens_2(\rset^d)$, $\Mun(\cdot, \mu)$ is twice differentiable.
  \item For any $x \in \rset^d$, $\Mun(x, \cdot)$ is $L$-differentiable and
    $(x,\mu,y) \mapsto \rmD^L \Mun(x, \mu)(y) \in \rmc(\rset^d \times
    \Pens_2(\rset^d) \times \rset^d, \rset^{d \times p})$.
  \item For any $x \in \rset^d$, $\Mun(x, \cdot)$ is twice $L$-differentiable and we have that 
    $(x,\mu,y_0,y_1) \mapsto (\rmD^L)^2 \Mun(x, \mu)(y_0,y_1) \in \rmc(\rset^d \times
    \Pens_2(\rset^d) \times \rset^d \times \rset^d, \rset^{d \times d \times p})$.
  \item For any $\mu \in \Pens_2(\rset^d)$ and $y \in \rset^d$,
    $x \mapsto \rmD^L \Mun(x, \mu)(y)$ is differentiable and $(x,\mu,y) \mapsto \rmd_x \rmD^L \Mun(x, \mu)(y) \in \rmc(\rset^d \times
    \Pens_2(\rset^d) \times \rset^d, \rset^{d \times d \times p})$.
  \end{enumerate}

\end{definition}

\subsection{Proof of \Cref{prop:euler}}
\label{app:euler}

In this section, we let $\alpha, \eta >0$ and  verify that the drift in~\eqref{eq:nonlinearSDE} satisfies the
conditions of \cite[Theorem 2.2]{bao2020first}.
For completeness we report the conditions of \cite[Theorem 2.2]{bao2020first} below.
We denote by $b(x, \nu):\rset^d\times \Pens_2(\rset^d)\to \rset^d$ and $\sigma(x, \nu):\rset^d\times \Pens_2(\rset^d)\to \rset^{d\times d}$ the drift and diffusion coefficient of the MKVSDE, respectively.
In the following, $\norm{s}$ denotes the standard Euclidean norm if $s$ is a vector and the Hilbert-Schmidt norm $\norm{s}=(\sum_{ij} s_{ij}^2)^{1/2}$ if $s$ is a matrix.

\emph{Conditions on the drift coefficient}
\begin{description}
\item[$(\mathbf{A}_b^1)$] There exists constants $L_b^1, \alpha_1>0$ such that
\begin{align}
\norm{b(x_1, \nu)-b(x_2, \nu)} &\leq L_b^1\left(1+\norm{x_1}^{\alpha_1}+\norm{x_2}^{\alpha_1}\right)\norm{x_1-x_2}\\
\langle x_1-x_2, b(x_1, \nu)-b(x_2, \nu)\rangle &\leq L_b^1 \norm{x_1-x_2}^2\\
\norm{b(x, \nu_1)-b(x, \nu_2)} &\leq L_b^1\wassersteinD[2](\nu_1, \nu_2)
\end{align}
\item[$(\mathbf{A}_b^2)$] Let $b_i\in\rmc^{2, (2,1)}(\rset^d \times \Pens_2(\rset^d), \rset^d)$, where $b_{i}$ denotes the $i$-th component of $b$. There exists constants $L_b^2,\alpha_2>0$ such that for any $i = 1, \dots, d$
\begin{align}
\norm{\nabla^2b_i(x, \nu)(z)}\vee&\norm{\nabla\{\rmD^L b_i(x, \nu)(\cdot)\}(z)}\vee\norm{(\rmD^L)^2b_i(x, \nu)(z, z)}\\ & \leq  L_b^2
\left(1+\norm{x}^{1+\alpha_2}+\norm{z}^{1+\alpha_2}+\nu(\vert\cdot\vert^2)^{(1+\alpha_2)/2}\right)
\end{align}
\item[$(\mathbf{A}_b^3)$] There exists constants $L_b^3, \alpha_3>0$ such that
\begin{align}
\norm{\nabla b(x_1, \nu_1)-\nabla b(x_2, \nu_2)} \leq& L_b^3\left(\norm{x_1-x_2}+\wassersteinD[2](\nu_1, \nu_2)\right)\\\times\left(1+\norm{x_1}^{\alpha_3}+\norm{x_2}^{\alpha_3}+\right.&\left.\nu_1(\vert\cdot\vert^2)^{\alpha_3/2}+\nu_2(\vert\cdot\vert^2)^{\alpha_3/2}\right)\\
\norm{\rmD^L b(x_1, \nu_1)(z_1)-\rmD^L b(x_2, \nu_2)(z_2)} \leq&\\ L_b^3\left(\norm{x_1-z_1}+\norm{x_2-z_2}+\right.&\left.\wassersteinD[2](\nu_1, \nu_2)\right)\\
\times\left(1+\norm{x_1}^{\alpha_3}+\norm{x_2}^{\alpha_3}+\norm{z_1}^{\alpha_3}\right.&\left.+\norm{z_2}^{\alpha_3}+\nu_1(\vert\cdot\vert^2)^{\alpha_3/2}+\nu_2(\vert\cdot\vert^2)^{\alpha_3/2}\right.\left.\right)
\end{align}
\item[$(\mathbf{A}_b^4)$] There exists a constant $L_b^4>0$ such that
\begin{align}
\norm{b(0, \nu)}\leq L_b^4.
\end{align}
\end{description}

\emph{Conditions on the diffusion coefficient}
\begin{description}
\item[$(\mathbf{A}_\sigma^1)$] There exists a constant $L_\sigma^1>0$ such that
\begin{align}
\norm{\sigma(x_1, \nu_1)-\sigma(x_2, \nu_2)} \leq L_\sigma^1\left(\norm{x_1-x_2}+\wassersteinD[2](\nu_1, \nu_2)\right)
\end{align}
\item[$(\mathbf{A}_\sigma^2)$] Let $\sigma_{ij}\in\rmc^{2, (2,1)}(\rset^d \times \Pens_2(\rset^d), \rset)$, where $\sigma_{ij}$ denotes the $(i, j)$-th component of $\sigma$. There exists a constant $L_\sigma^2>0$ such that for any $i, j = 1, \dots, d$
\begin{align}
\norm{\rmD^L\sigma_{ij}(x, \nu)(z)}+\norm{\nabla\{\rmD^L\sigma_{ij}(x, \nu)(\cdot)\}(z)}+\norm{(\rmD^L)^2\sigma_{ij}(x, \nu)(z, z)} \leq L_\sigma^2
\end{align}
\item[$(\mathbf{A}_\sigma^3)$] There exists a constant $L_\sigma^3>0$ such that
\begin{align}
&\norm{\nabla\sigma(x_1, \nu_1)-\nabla\sigma(x_2, \nu_2)} \leq L_\sigma^3\left(\norm{x_1-x_2}+\wassersteinD[2](\nu_1, \nu_2)\right)\\
&\norm{\rmD^L\sigma(x_1, \nu_1)(z_1)-\rmD^L\sigma(x_2, \nu_2)(z_2)} \\
&\qquad\leq L_\sigma^3\left(\norm{x_1-z_1}+\norm{x_2-z_2}+\wassersteinD[2](\nu_1, \nu_2)\right)
\end{align}
\item[$(\mathbf{A}_\sigma^4)$] There exists a constant $L_\sigma^4>0$ such that
\begin{align}
\norm{\sigma(0, \nu)}+\norm{\rmD^L \sigma(x_1, \nu)(x_2)\sigma(x_2, \nu)} \leq L_\sigma^4.
\end{align}
\end{description}

 We start by computing the
$L$-derivatives of $b$.  We recall that $b: \ \rset^d \times \Pens(\rset^d)\to\rset^d$ is
given for any $x \in \rset^d$ and $\nu \in \Pens(\rset^d)$ by
\begin{equation}
  \textstyle{
    b(x, \nu) = \int_{\rset^p} \tilde{b}_\eta(x, \nu, y) \rmd \mu(y) - \alpha \nabla U(x), \quad \tilde{b}_\eta(x, \nu, y) = \nabla_1 \kker(x,y) / (\nu[\kker(\cdot, y)] + \eta).
    }
\end{equation}
In all of our results, $\mu$ can be replaced by $\mu^M$.

\begin{lemma}
  \label{lemma:derivatives}
  Let $\alpha, \eta >0$ and assume \rref{assum:general_kker} and \rref{assum:lip_U}. Then
  $b \in \rmc^{2,(2,1)}(\rset^d \times \Pens_2(\rset^d), \rset^d)$. In addition, we have 
  for any $\nu \in \Pens_2(\rset^d)$, $x_0, x_1, x_2 \in \rset^d$
  \begin{enumerate}
      \item $\textstyle{ \rmd_x^2 b(x_0, \nu) = \int_{\rset^p} \rmd^3_x \kker(x_0,y) / (\nu[\kker(\cdot, y)] + \eta)\rmd \mu(y) - \alpha \rmd^3_x U(x_0), } $
      \item $\textstyle{\rmD^L b(x_0, \nu)(x_1) = \int_{\rset^p} \nabla_1 \kker(x_0,y) \otimes \nabla_1 \kker(x_1,y) / (\nu[\kker(\cdot, y)] + \eta) \rmd \mu(y),}$
      \item $\textstyle{\rmd_x \rmD^L b(x_0, \nu)(x_1) = \int_{\rset^p} \rmd_x^2 \kker(x_0,y) \otimes \nabla_1 \kker(x_1,y) / (\nu[\kker(\cdot, y)] + \eta) \rmd \mu(y),} $
      \item $\textstyle{(\rmD^L)^2 b(x_0, \nu)(x_1,x_2)= \int_{\rset^p} \nabla_1 \kker(x_0,y) \otimes \nabla_1 \kker(x_1,y) \otimes \nabla_1 \kker(x_2,y) / (\nu[\kker(\cdot, y)] + \eta) \rmd \mu(y),}$
  \end{enumerate}
  where $\otimes$ denotes the outer product.
\end{lemma}

\begin{proof}
We prove each result separately.
\begin{enumerate}
    \item Recall that $b_i(x, \nu) = \int_{\rset^p} \tilde{b}_{\eta, i}(x, \nu, y) \rmd \mu(y) - \alpha \partial_i U(x)$ where $\partial_i$ denotes the derivative w.r.t. the $i$-th component and $\tilde{b}_{\eta, i}(x, \nu, y)  = \partial_{1, i} \kker(x,y) / (\nu[\kker(\cdot, y)] + \eta)$.
  We have that for any $x \in \rset^d$, $\nu \in \Pens_2(\rset^d)$ and $y \in \rset^p$
  \begin{align}
    &\partial_{1, j} \tilde{b}_{\eta, i}(x, \eta, y) = \partial_{1,j}\partial_{1, i} \kker(x,y) / (\nu[\kker(\cdot, y)] + \eta),\\
    & \partial_{1, j}^2 \tilde{b}_\eta(x, \eta, y) = \partial_{1,j}^2\partial_{1, i} \kker(x,y) / (\nu[\kker(\cdot, y)] + \eta) .
  \end{align}
  We have that for any $x \in \rset^d$, $\nu \in \Pens_2(\rset^d)$
  and $y \in \rset^p$,
  $\normLigne{\partial_{1, j} \tilde{b}_{\eta, i}(x, \eta, y)} \leq \Mtt/\eta$ and
  $\normLigne{\partial_{1, j}^2 \tilde{b}_\eta(x, \eta, y)} \leq \Mtt/\eta$. Hence, using
  the dominated convergence theorem, $b \in \rmc^2(\rset^d, \rset^d)$ and we
  have that for any $x_0 \in \rset^d$ and $\nu \in \Pens_2(\rset^d)$
  \begin{equation}
    \textstyle{ \partial_{j}^2 b_i(x_0, \nu) = \int_{\rset^p} \partial_{1,j}^2\partial_{1, i} \kker(x_0,y) / (\nu[\kker(\cdot, y)] + \eta)\rmd \mu(y) - \alpha \partial_{1,j}^2\partial_{1, i} U(x_0). }
  \end{equation}
  \item 
  Define $\bar{b}: \ \rset^d \times \rmL^2(\Omega, \rset^d)\to \rset^d$, for any
  $x \in \rset^d$ and $X \in \rmL^2(\Omega, \rset^d)$, via
  \begin{equation}
    \textstyle{\bar{b}(x, X) = \int_{\rset^p} \nabla_1 \kker(x,y) / (\expeLigne{\kker(X, y)} + \eta) \rmd \mu(y) - \alpha \nabla U(x),}
  \end{equation}
  so that for any $X \in \rmL^2(\Omega, \rset^d)$,
  $\bar{b}(\cdot, X) = b(\cdot, \mathcal{L}(X))$. In addition, using Example~\ref{ex:l_derivative}, the fact that $\rmD^L\Mun$ therein is bounded under \rref{assum:general_kker} and the
  dominated convergence theorem we have that for any
  $H \in \rmL^2(\Omega, \rset^d)$ and $x \in \rset^d$
  \begin{equation}
    \textstyle{\underset{\normLigne{H} \to 0}{\lim} \normLigne{\bar{b}(x,X+H) - \bar{b}(x,X) - \int_{\rset^p}  \frac{\nabla_1 \kker(x,y)\langle \nabla_1 \kker(X, y), H \rangle }{ \expeLigne{\kker(X, y)} + \eta} \rmd \mu(y) }/\normLigne{H} = 0.} 
  \end{equation}
  Therefore, for any $x_0, x_1 \in \rset^d$ and $\nu \in \Pens_2(\rset^d)$
  \begin{equation}
    \textstyle{\rmD^L b(x_0, \nu)(x_1) = \int_{\rset^p} \nabla_1 \kker(x_0,y) \otimes \nabla_1 \kker(x_1,y) / (\nu[\kker(\cdot, y)] + \eta) \rmd \mu(y).}
  \end{equation}
  Using the boundedness of $\nabla_1\kker$ ensured by \rref{assum:general_kker} and the dominated convergence theorem, we have   $(x_0, \nu, x_1) \mapsto \rmD^L b(x_0, \nu)(x_1) \in \rmc(\rset^d \times
  \Pens_2(\rset^d) \times \rset^d, \rset^{d\times d})$.
  \item For any $x_1 \in \rset^d$ and
  $\nu \in \Pens_2(\rset^d)$,
  $x_0 \mapsto \rmD^L b(x_0, \nu)(x_1) \in \rmc^1(\rset^d, \rset^{d\times d})$
  and we have that for any $x_0 \in \rset^d$
  \begin{equation}
    \textstyle{\partial_j \rmD^L b_i(x_0, \nu)(x_1) = \int_{\rset^p} \partial_{1,i}\partial_{1,j} \kker(x_0,y) \otimes \nabla_1 \kker(x_1,y) / (\nu[\kker(\cdot, y)] + \eta) \rmd \mu(y).}
  \end{equation}
  \item 
  Finally,
  we define
  $\bar{g}: \ \rset^d \times \rmL^2(\Omega, \rset^{d}) \times \rset^d \to
  \rset^{d \times d}$ such that for any $x_0, x_1 \in \rset^d$ and
  $X \in \rmL^2(\Omega, \rset^{d})$
  \begin{equation}
    \textstyle{\bar{g}(x_0, X, x_1) = \int_{\rset^p} \nabla_1 \kker(x_0,y) \otimes \nabla_1 \kker(x_1,y) / (\expeLigne{\kker(X, y)} + \eta) \rmd \mu(y).}    
  \end{equation}
For any $X \in \rmL^2(\Omega, \rset^d)$,
  $\bar{g}(\cdot, X) = \rmD^L(\cdot, \mathcal{L}(X), \cdot)$. In addition, using the boundedness ensured by \rref{assum:general_kker} and
  dominated convergence theorem we have that for any
  $H \in \rmL^2(\Omega, \rset^d)$ and $x_0, x_1 \in \rset^d$
  \begin{align}
    & \frac{1}{\normLigne{H}} \norm{\bar{g}(x_0, X+H, x_1) - \bar{g}(x_0, X, x_1)  - \int_{\rset^p} \frac{\nabla_1 \kker(x_0,y) \otimes \nabla_1 \kker(x_1, y) \langle \nabla_1 \kker(X, y), H \rangle}{\expeLigne{\kker(X, y)} + \eta} \rmd \mu(y)} \to 0.
  \end{align}
  as $\normLigne{H} \to 0$.
  Therefore, for any $x_0,x_1, x_2 \in \rset^d$ and $\nu \in \Pens_2(\rset^d)$
  \begin{align}
   & \textstyle{(\rmD^L)^2 b(x_0, \nu)(x_1,x_2) }\\
    &\qquad\textstyle{= \int_{\rset^p} \nabla_1 \kker(x_0,y) \otimes \nabla_1 \kker(x_1,y) \otimes \nabla_1 \kker(x_2,y) / (\nu[\kker(\cdot, y)] + \eta) \rmd \mu(y).}
  \end{align}
  We conclude using again the boundedness of $\nabla_1\kker$ ensured by \rref{assum:general_kker} and the dominated convergence theorem to show that
  $(x_0,\mu,x_1,x_2) \mapsto (\rmD^L)^2 b(x_0, \nu)(x_1,x_2) \in \rmc(\rset^d
  \times \Pens_2(\rset^d) \times \rset^d \times \rset^d, \rset^{d \times d
    \times d})$.
    \end{enumerate}
\end{proof}

\begin{proof}[Proof of \Cref{prop:euler}]
  Using \Cref{lemma:derivatives} there exists $\Ktt \geq 0$ such that for any
  $\nu \in \Pens_2(\rset^d)$ and $x_0, x_1, x_2 \in \rset^d$ we have
  \begin{align}
    \label{eq:bound}
   & \textstyle{\normLigne{b(0, \nu)} + \normLigne{\nabla^2 b(x_0, \nu)} + \normLigne{\rmD^L b(x_0, \nu)(x_1)} }\\
   &\qquad\qquad\textstyle{+ \normLigne{\nabla \rmD^L b_i(x_0, \nu)(x_1)} + \normLigne{(\rmD^L)^2 b(x_0, \nu)(x_1,x_2)} \leq \Ktt. }
  \end{align}
  In addition, using~\eqref{eq:lipo_lip}, for any $x_1,x_2 \in \rset^d$ and $\nu_1, \nu_2 \in \Pens_2(\rset^d)$ we have
  \begin{equation}
    \norm{b(x_1, \nu_1) - b(x_2, \nu_2)} \leq (\Mtt/\eta)(1 + (1/\eta))(1 + \alpha\Ltt)(1+\Mtt) \parenthese{ \norm{x_1 - x_2} + \wassersteinD[2](\nu_1, \nu_2)}. 
  \end{equation}
  Hence, $(\mathbf{A}_b^1)$ in \cite{bao2020first} is satisfied. Using~\eqref{eq:bound} we have that $(\mathbf{A}_b^2)$ and $(\mathbf{A}_b^4)$ in
  \cite{bao2020first} are satisfied. 
Similarly to~\eqref{eq:lipo_lip}, we have
for any $x_1,x_2,x_3,x_4 \in \rset^d$ and $\nu_1, \nu_2 \in \Pens_2(\rset^d)$
\begin{align}
\label{eq:a3_part_two}
 & \textstyle{\normLigne{\rmD^L b(x_1,\nu_1)(x_2) - \rmD^L b(x_2,\nu_2)(x_3)}}\\
 &\qquad\leq \textstyle{\|\int_{\rset^p} \nabla_1 \kker(x_1,y) \otimes \nabla_1 \kker(x_2,y) / (\nu_1[\kker(\cdot, y)] + \eta) \rmd \mu(y)} \\
 &\qquad \textstyle{\qquad  - \int_{\rset^p} \nabla_1 \kker(x_3,y) \otimes \nabla_1 \kker(x_4,y) / (\nu_2[\kker(\cdot, y)] + \eta) \rmd \mu(y)\|} \\
&\qquad\leq (\Mtt^2/\eta)(2+(1/\eta))(1+\Mtt) \{\normLigne{x_1 - x_3} + \normLigne{x_2 - x_4} + \wassersteinD[2](\nu_1, \nu_2)\}.
\end{align}
Combining~\eqref{eq:a3_part_one} and~\eqref{eq:a3_part_two}, we get that
$(\mathbf{A}_b^3)$ in \cite{bao2020first} is satisfied. In addition, note that
$(\mathbf{A}_\sigma^1)$, $(\mathbf{A}_\sigma^2)$, $(\mathbf{A}_\sigma^3)$ and
$(\mathbf{A}_\sigma^4)$ are immediately satisfied since
$\sigma = \sqrt{2 \alpha}$ in our case. We conclude the proof upon using
\cite[Theorem 2.5]{bao2020first}.
\end{proof}

\subsection{Proof of \Cref{prop:approx_mu}}
\label{app:approx_mu}

Let $M, N \in \nsets$ and $\alpha, \eta > 0$.
Using~\eqref{eq:lipo_lip}, we have for any $x_1, x_2 \in \rset^d$ and $\nu_1, \nu_2 \in \Pens_1(\rset^d)$
\begin{align}
  \label{eq:intermezzo}
  &\textstyle{\normLigne{b(x_1, \nu_1) - b^M(x_2,\nu_2)}}\\
   &\textstyle{\leq ((\Mtt/\eta)(1 + (1/\eta)) + \alpha \Ltt) \parentheseLigne{ \norm{x_1 - x_2} + \int_{\rset^p} \abs{\nu_1[\kker(\cdot, y)] - \nu_2[\kker(\cdot, y)]} \rmd \mu(y)}} \\
  & \textstyle{\phantom{\leq} + \normLigne{(1/M) \sum_{j=1}^M \{ \tilde{b}_\eta(x_2, \nu_2, y^{j,M}) - \int_{\rset^p} \tilde{b}_\eta(x_2, \nu_2, y) \rmd \mu(y)\}}}. 
\end{align}
For any $i \in \{1,2\}$ and $x_i^{1:N} \in (\rset^d)^N$, let $\nu_i^N = (1/N) \sum_{k=1}^N \updelta_{x_i^{k,N}}$. Using~\eqref{eq:intermezzo}, we have that for any $x_1^{1:N}, x_2^{1:N} \in (\rset^d)^N$
\begin{align}
  \label{eq:intermezzo_duo}                     
&  \textstyle{\normLigne{b(x_1^{1,N}, \nu_1^N) - b^M(x_2^{1,N},\nu_2^N)}} \\
&\textstyle{\leq ((\Mtt/\eta)(1 + (1/\eta)) + \Ltt)(1 + \Mtt)}\\
&\qquad\times \textstyle{\parentheseLigne{ \normLigne{x_1^{1,N} - x_2^{1,N}} + (1/N) \sum_{k=1}^N \normLigne{x_1^{k,N} - x_2^{k,N}}}} \\
 &\textstyle{\phantom{\leq}+ \normLigne{(1/M) \sum_{j=1}^M \{ \tilde{b}_\eta(x_2^{1,N}, \nu_2^N, y^{j,M}) - \int_{\rset^p} \tilde{b}_\eta(x_2^{1,N}, \nu_2^N, y) \rmd \mu(y)\}}}.
\end{align}
Let $C \geq 0$ and $\{A_k\}_{k=1}^M$ be a family of $d$-dimensional random
variables such that  for any $i,j \in \{1,\dots, M\}$ with $i\neq j$
\begin{equation}
  \label{eq:cond_AK}
  \textstyle{
    \expeLigne{\normLigne{A_i}^2} \leq C, \qquad \expeLigne{\langle A_i, A_j \rangle} = 0.
    }
\end{equation}
Then, we have that
\begin{equation}
  \label{eq:bound_bound}
  \textstyle{
  \expeLigne{\normLigne{(1/M) \sum_{k=1}^N A_k}} \leq (C /M)^{1/2}. }
\end{equation}
Let
$A_k := \tilde{b}_\eta(X_n^{1,N,M}, \lambda_n^{N,M}, y^{k,M}) - \int_{\rset^p} \tilde{b}_\eta(X_n^{1,N,M}, \lambda_n^{N,M}, y) \rmd \mu(y)$ for any $k \in \{1, \dots, M\}$. Then $\{A_k\}_{k=1}^M$ satisfies~\eqref{eq:cond_AK} with
$C = \Mtt/\eta$. Combining this result,~\eqref{eq:intermezzo_duo}, \eqref{eq:bound_bound} and that for any $n\in\nset$,
$\{X_n^{k,N} -  X_n^{k,N,M}\}_{k=1}^N$ is exchangeable, we have \begin{align}
\label{eq:nonconditional}
  &\expeLigne{\textstyle{\normLigne{b(X_n^{1,N}, \lambda_n^{N}) - b^M(X_n^{1,N,M},\lambda_n^{N,M})}}} \\
  &\qquad \textstyle{\leq ((\Mtt/\eta)(1 + (1/\eta)) + \Ltt)(1 + \Mtt)}\\
  &\qquad\qquad\times\textstyle{ \parentheseLigne{ \expeLigne{\normLigne{X_n^{1,N} - X_n^{1,N,M}}} + (1/N) \sum_{k=1}^N \expeLigne{\normLigne{X_n^{k,N} - X_n^{k,N,M}}}}} \\
  &\qquad\quad \textstyle{+ \expeLigne{\normLigne{(1/M) \sum_{k=1}^M A_k}}} \\
  &\qquad \textstyle{\leq 2((\Mtt/\eta)(1 + (1/\eta)) + \Ltt)(1 + \Mtt) \expeLigne{\normLigne{X_n^{1,N} - X_n^{1,N,M}}}} + (\Mtt/(M\eta))^{1/2}.
\end{align}
Using this result, we have for any $n\in\nset$
\begin{align}
\label{eq:resampling_bound1}
   \expeLigne{\normLigne{X_{n+1}^{1,N} - X_{n+1}^{1,N,M}}} \leq& \expeLigne{ \normLigne{X_n^{1,N} - X_n^{1,N,M}}}\\
   \qquad &+ \gamma\expe{\frac{\gamma\norm{b(X_n^{1, N}, \lambda_n^N)}}{(1+\gamma \norm{b(X_n^{1, N}, \lambda_n^N)})}\frac{\norm{b(X_n^{1, N}, \lambda_n^N)-b^M(X_n^{1, N, M}, \lambda_n^{N,M})}}{(1+\gamma \norm{b^M(X_n^{1, N, M}, \lambda_n^{N,M})})}}\\
   \qquad&+\gamma\expe{\frac{\norm{b(X_n^{1, N}, \lambda_n^N)-b^M(X_n^{1, N, M}, \lambda_n^{N,M})}}{(1+\gamma \norm{b^M(X_n^{1, N, M}, \lambda_n^{N,M})})}} \\
   \quad\leq& \expeLigne{ \normLigne{X_n^{1,N} - X_n^{1,N,M}}} + 2\gamma\expe{\norm{b(X_n^{1, N}, \lambda_n^N)-b^M(X_n^{1, N, M}, \lambda_n^{N,M})}}\\
\quad\leq& C\left((1+2\gamma)\expeLigne{ \normLigne{X_n^{1,N} - X_n^{1,N,M}}} + 2\gamma M^{-1/2}\right).
\end{align}
Proceeding recursively and recalling that $X_0^{1, N}=X_0^{1, N, M}$ the result follows.

\subsection{Proof of \Cref{prop:resampling}}
\label{app:resampling}

The proof follows the same structure of that of \Cref{prop:approx_mu}. Let $M, N, m \in \nsets$ with $m\leq M$ and $\alpha, \eta > 0$.
Proceeding as above we can write a result equivalent to~\eqref{eq:nonconditional} 
\begin{align}
\label{eq:conditional}
  &\expeLigne{\textstyle{\normLigne{b(X_n^{1,N}, \lambda_n^{N}) - b^M(X_n^{1,N,m},\lambda_n^{N,m})}}} \\
  & \textstyle{\leq ((\Mtt/\eta)(1 + (1/\eta)) + \Ltt)(1 + \Mtt)}\\
  &\quad\times\textstyle{ \parentheseLigne{ \expeLigne{\normLigne{X_n^{1,N} - X_n^{1,N,m}}} + (1/N) \sum_{k=1}^N \expeLigne{\normLigne{X_n^{k,N} - X_n^{k,N,m}}}}} \\
  &  \quad \textstyle{+ \expeLigne{\normLigne{(1/m) \sum_{k=1}^m A_k}}},
\end{align}
where $A_k:= \tilde{b}_\eta(X_n^{1,N,m}, \lambda_n^{N,m}, y^{j,m}) - \int_{\rset^p} \tilde{b}_\eta(X_n^{1,N,m}, \lambda_n^{N,m}, y) \rmd \mu(y)$ and $\{y_1, \dots, y_m\}$ is a sample from $\mu^M$.
Consider the last expectation and decompose
\begin{align}
    &\expe{\norm{(1/m) \sum_{j=1}^m \{ \tilde{b}_\eta(X_n^{1,N,m}, \lambda_n^{N,m}, y^{j,m}) - \int_{\rset^p} \tilde{b}_\eta(X_n^{1,N,m}, \lambda_n^{N,m}, y) \rmd \mu(y)\}}}\\
    &\qquad\leq\expe{\norm{(1/m) \sum_{j=1}^m \{ \tilde{b}_\eta(X_n^{1,N,m}, \lambda_n^{N,m}, y^{j,m}) - (1/M)\sum_{k=1}^M\tilde{b}_\eta(X_n^{1,N,m}, \lambda_n^{N,m}, y^{j,M})\}}}\\
    &\qquad\quad+\expe{\norm{ (1/M)\sum_{k=1}^M\tilde{b}_\eta(X_n^{1,N,m}, \lambda_n^{N,m}, y^{j,M})- \int_{\rset^p} \tilde{b}_\eta(X_n^{1,N,m}, \lambda_n^{N,m}, y) \rmd \mu(y)\}}}.
\end{align}
As shown in the proof of \Cref{prop:approx_mu}, the second term satisfies
\begin{align}
&\expeLigne{\normLigne{ (1/M)\sum_{k=1}^M\tilde{b}_\eta(X_n^{1,N,m}, \lambda_n^{N,m}, y^{j,M})- \int_{\rset^p} \tilde{b}_\eta(X_n^{1,N,m}, \lambda_n^{N,m}, y) \rmd \mu(y)\}}}\\
&\qquad\qquad\leq (\Mtt/(M\eta))^{1/2}.
\end{align}
Let us denote the $\sigma$-field generated by $\mu^M$ by $\mathcal{S}^M:=\sigma\left(y_j : j=1, \ldots, M\right)$.
For any $k \in \{1, \dots, m\}$, $x_1 \in \rset^d$ let $A_k = \tilde{b}_\eta(x_1, \nu_1, y^{k,m}) - (1/M)\sum_{k=1}^M\tilde{b}_\eta(X_n^{1,N,m}, \lambda_n^{N,m}, y^{j,M}) $. 
Conditionally on $\mathcal{S}^M$, $\{A_k\}_{k=1}^m$ satisfies~\eqref{eq:cond_AK} with
$C = \Mtt/\eta$.
Hence, 
\begin{align}
&\expeLigne{\normLigne{(1/m) \sum_{j=1}^m \{ \tilde{b}_\eta(X_n^{1,N,m}, \lambda_n^{N,m}, y^{j,m}) - (1/M)\sum_{k=1}^M\tilde{b}_\eta(X_n^{1,N,m}, \lambda_n^{N,m}, y^{j,M})\}}}\\
&\qquad\qquad\leq (\Mtt/(m\eta))^{1/2}.
\end{align}
Combining the two expectations we obtain the following bound for~\eqref{eq:conditional}
\begin{align}
    &\expeLigne{\textstyle{\normLigne{b(X_n^{1,N}, \lambda_n^{N}) - b^M(X_n^{1,N,m},\lambda_n^{N,m})}}} \\
  &\qquad \textstyle{\leq ((\Mtt/\eta)(1 + (1/\eta)) + \Ltt)(1 + \Mtt)}\\
  &\qquad\qquad\times\textstyle{ \parentheseLigne{ \expeLigne{\normLigne{X_n^{1,N} - X_n^{1,N,m}}} + (1/N) \sum_{k=1}^N \expeLigne{\normLigne{X_n^{k,N} - X_n^{k,N,m}}}}} \\
  & \qquad \quad  \textstyle{+(\Mtt/(m\eta))^{1/2}+(\Mtt/(M\eta))^{1/2} }\\
  &\qquad \textstyle{\leq ((\Mtt/\eta)(1 + (1/\eta)) + \Ltt)(1 + \Mtt)}\\
  &\qquad\qquad\times\textstyle{ \parentheseLigne{ \expeLigne{\normLigne{X_n^{1,N} - X_n^{1,N,m}}} + (1/N) \sum_{k=1}^N \expeLigne{\normLigne{X_n^{k,N} - X_n^{k,N,m}}}}} \\
  & \qquad \quad  \textstyle{+2(\Mtt/(m\eta))^{1/2} },
\end{align}
where we used the fact that $m\leq M$.
The above allows us to obtain an equivalent result to~\eqref{eq:resampling_bound1}.
Proceeding recursively and recalling that $X_0^{1, N}=X_0^{1, N, m}$ we have the result.

\subsection{Stability of continuous time process}
\label{app:prop10_continuous_time}
\begin{proposition}
\label{prop:approx_mu_continuous_time}
Assume \tup{\rref{assum:general_kker}}, \rref{assum:lip_U}. Let
$\eta, \alpha > 0$. For any $N,M \in \nsets$, let $(\bfX_t^{1,N})_{t \geq 0}$
and $(\bfX_t^{1,N,M})_{t \geq 0}$ be the solution of \eqref{eq:particle} and
\eqref{eq:particle_M} respectively, with initial condition
$\bfX_0^{1:N} \in \Pens_1((\rset^d)^N)$ such that $\{\bfX_0^{k,N}\}_{k=1}^N$ is
exchangeable and $\{\bfX_0^{k,N}\}_{k=1}^N = \{\bfX_0^{k,N,M}\}_{k=1}^N$. Then
for any $T \geq 0$ there exists $C_T \geq 0$  such that for any $M, N \in \nsets$ and $\ell \in \{1, \dots, N\}$
  \begin{equation}
    \textstyle{
      \expeLigne{\sup_{t \in \ccint{0,T}} \normLigne{\bfX_t^{\ell,N}- \bfX_t^{\ell,N, M}}} \leq C_T M^{-1/2} \eqsp .
      }
  \end{equation}
\end{proposition}
\begin{proof}
The proof is identical to that of \Cref{prop:approx_mu} with the discrete recursion replaced by an application of Gr\"{o}nwall's lemma. Let $M, N \in \nsets$ and $\alpha, \eta > 0$. 
For any $k \in \{1, \dots, M\}$, $x_1 \in \rset^d$ and
$\nu_1 \in \Pens_1(\rset^d)$ let
$A_k = \tilde{b}_\eta(x_1, \nu_1, y^{k,M}) - \int_{\rset^p} \tilde{b}_\eta(x_1,
\nu_1, y) \rmd \mu(y)$. Then $\{A_k\}_{k=1}^M$ satisfies \eqref{eq:cond_AK} with
$C = \Mtt/\eta$. 

Combining this result, \eqref{eq:bound_bound},
\eqref{eq:intermezzo_duo} and that for any $t \geq 0$,
$\{\bfX_t^{k,N} - \bfX_t^{k,N,M}\}_{k=1}^N$ is exchangeable, we have that for
any $t \geq 0$,
\begin{align}
  &\expeLigne{\textstyle{\normLigne{b(\bfX_t^{1,N}, \bflambda_t^{N}) - b^M(\bfX_t^{1,N,M},\bflambda_t^{N,M})}}} \\
  &\qquad \textstyle{\leq ((\Mtt/\eta)(1 + (1/\eta)) + \Ltt)(1 + \Mtt) \parentheseLigne{ \expeLigne{\normLigne{\bfX_t^{1,N} - \bfX_t^{1,N,M}}} + (1/N) \sum_{k=1}^N \expeLigne{\normLigne{\bfX_t^{k,N} - \bfX_t^{k,N,M}}}}} \\
  & \qquad \qquad  \textstyle{+ \expeLigne{\normLigne{(1/M) \sum_{j=1}^M \{ \tilde{b}_\eta(\bfX_t^{1,N,M}, \bflambda_t^{N,M}, y^{j,M}) - \int_{\rset^p} \tilde{b}_\eta(\bfX_t^{1,N,M}, \bflambda_t^{N,M}, y) \rmd \mu(y)\}}}} \\
  &\qquad \textstyle{\leq 2((\Mtt/\eta)(1 + (1/\eta)) + \Ltt)(1 + \Mtt) \expeLigne{\normLigne{\bfX_t^{1,N} - \bfX_t^{1,N,M}}}} + (\Mtt/(M\eta))^{1/2}     \eqsp .
\end{align}
Using this result, we have for any $t \geq 0$
\begin{align}
  &\textstyle{ \expeLigne{\sup_{s \in \ccint{0,t}} \normLigne{\bfX_s^{1,N} - \bfX_s^{1,N,M}}} }\\
  & \qquad \qquad \qquad\textstyle{\leq 2((\Mtt/\eta)(1 + (1/\eta)) + \Ltt)(1 + \Mtt) \int_0^t \expeLigne{\sup_{u \in \ccint{0,s}} \normLigne{\bfX_u^{1,N} - \bfX_u^{1,N,M}}} \rmd s + t (\Mtt/(M\eta))^{1/2} \eqsp . }
\end{align}
We conclude upon using Gr\"{o}nwall's lemmma.
\end{proof}

\section{Additional Experiments}
\label{app:ex}

\subsection{Choice of initial distribution and reference measure}
\label{app:pi0}

To study the influence of the reference measure $\pi_0$ on $\Gun_{\alpha}^{\eta}$ and of the initial condition $X_0^{1:N}$ on the speed of convergence we consider the following toy Fredholm integral equation 
\[
  \textstyle{
    \mathcal{N}(y;0,\sigma_{\kker}^{2}+\sigma_{\pi}^{2})=\int_{\rset}\mathcal{N}(x;0,\sigma_{\pi}^{2})\mathcal{N}(y;x,\sigma_{\kker}^{2})\rmd
    x,}
\]
where $\mathcal{N}(x; m, \sigma^2)$ is a Gaussian distribution with mean $m$ and
variance $\sigma^2$, with $\sigma_\pi^2=0.43^2$ and
$\sigma_{\kker}^2=0.45^2$. We assume that $\mu$ is known through a sample of size $M=10^4$, we consider four reference measures: a uniform
distribution, a zero-mean Gaussian distribution with variance given by the
empirical variance of the sample from $\mu$, a Student's t-distribution with 100
degrees of freedom and a zero-mean Laplace distribution with variance given by the
empirical variance of the sample from $\mu$; and three initial conditions: the
solution $\pi$, a uniform over the interval $[-2, 2]$ which contains more than
$99\%$ of the mass and a highly peaked zero-mean Gaussian distribution.  As an
example we set $N=m=500$, $\gamma = 10^{-2}$, $\alpha = 0.02$ and iterate for
300 time steps. The behaviour for other values of $N, m$ is similar.

\begin{figure}
\centering
\resizebox{0.75\textwidth}{!}{%
\begin{tikzpicture}[every node/.append style={font=\normalsize}]
\node (img1) {\includegraphics[width=0.4\textwidth]{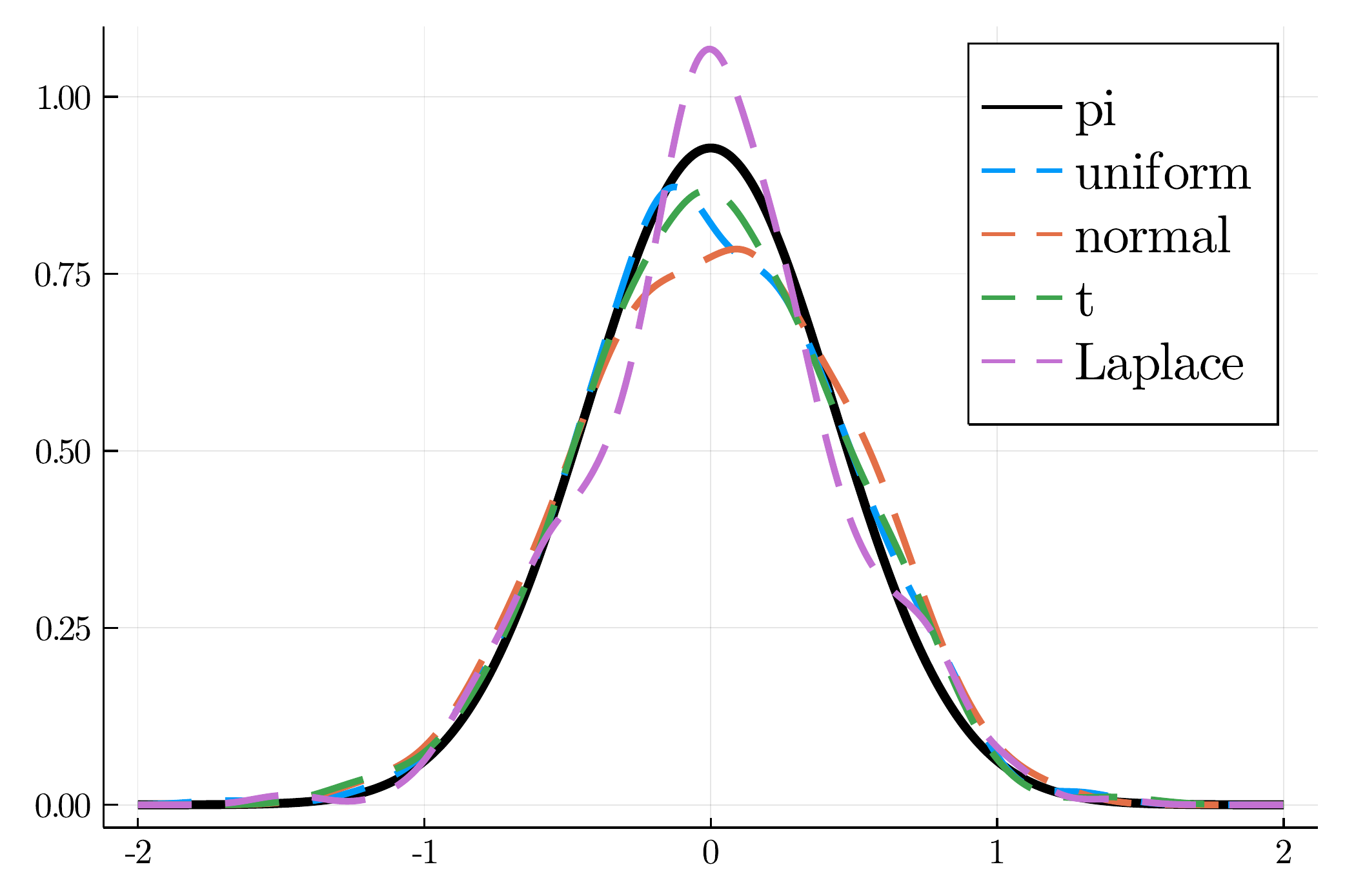}};
\node[left=of img1, node distance = 0, rotate=90, anchor = center, yshift = -1cm] {$X_0^{1:N}\sim \pi$};
\node[right=of img1, node distance = 0, xshift = -0.5cm] (img2) {\includegraphics[width=0.4\textwidth]{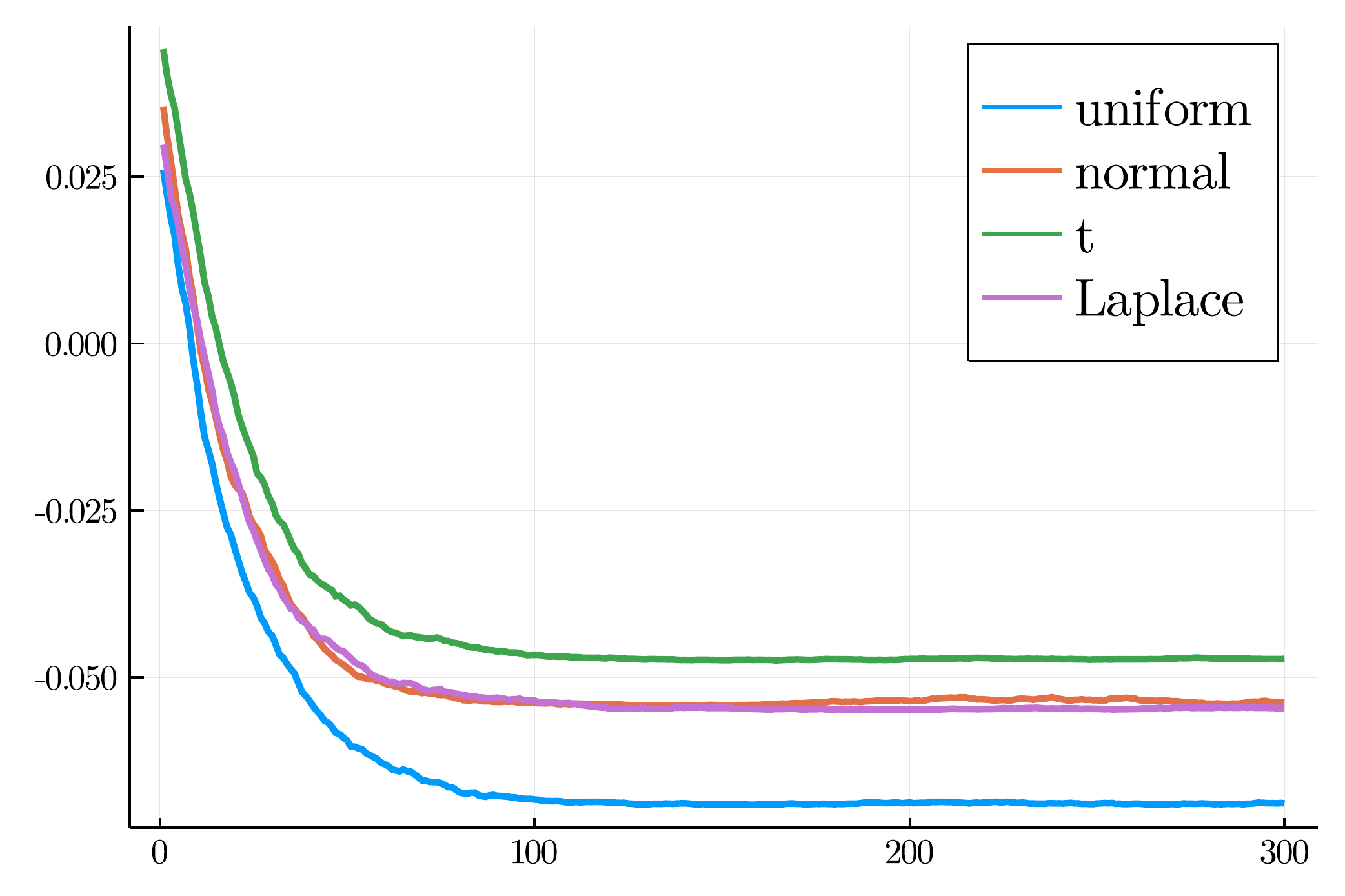}};
\node[above=of img1, node distance = 0, yshift = -1cm] {$\hat{\pi}$};
\node[above=of img2, node distance = 0, yshift = -1cm] {$\log\Gun_{\alpha}^{\eta}$};
\node[below=of img1] (img3) {\includegraphics[width=0.4\textwidth]{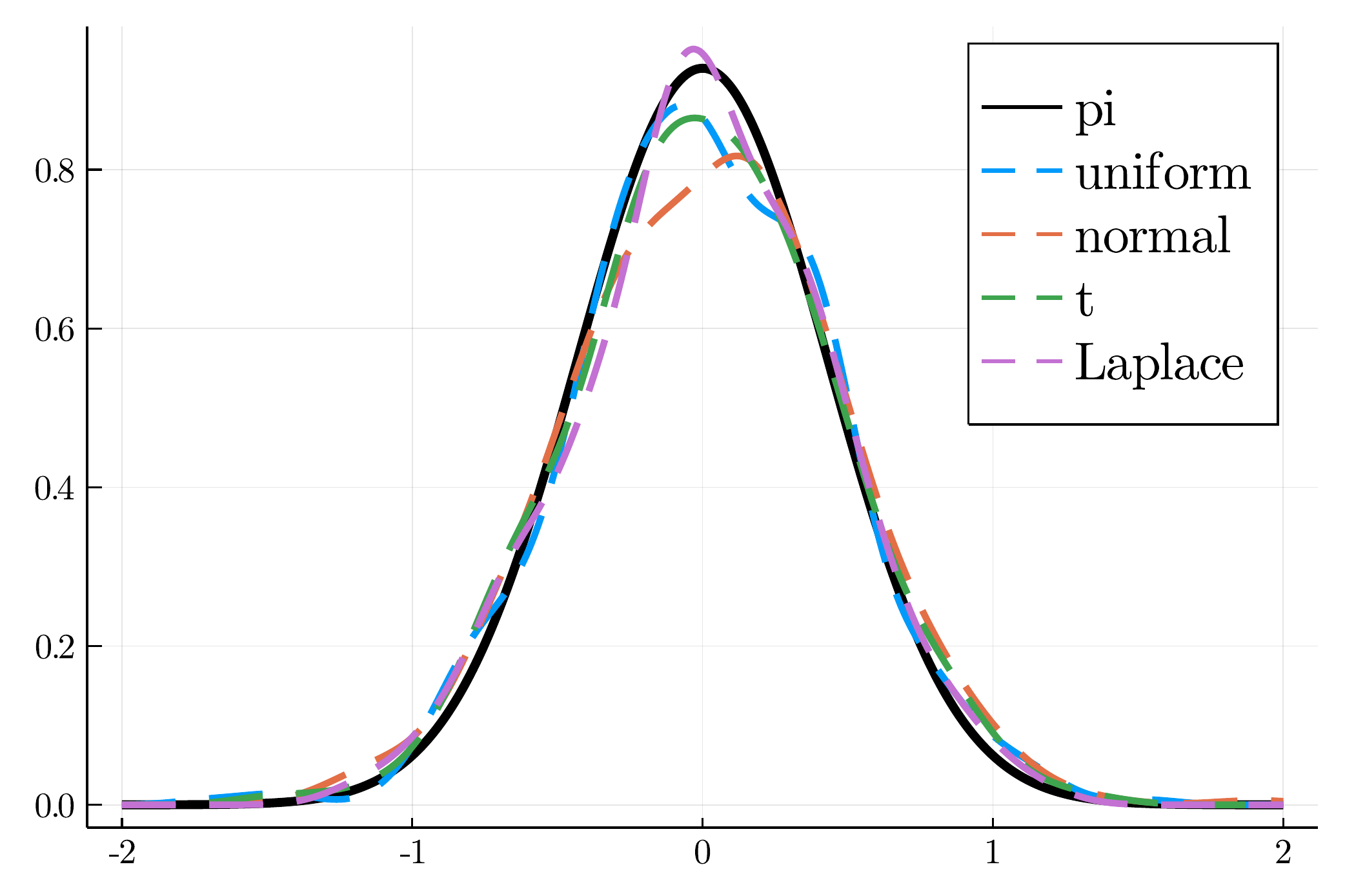}};
\node[left=of img3, node distance = 0, rotate=90, anchor = center, yshift = -1cm] {$X_0^{1:N}\sim U(-2, 2)$};
\node[right=of img3, node distance = 0, xshift = -0.5cm] (img4) {\includegraphics[width=0.4\textwidth]{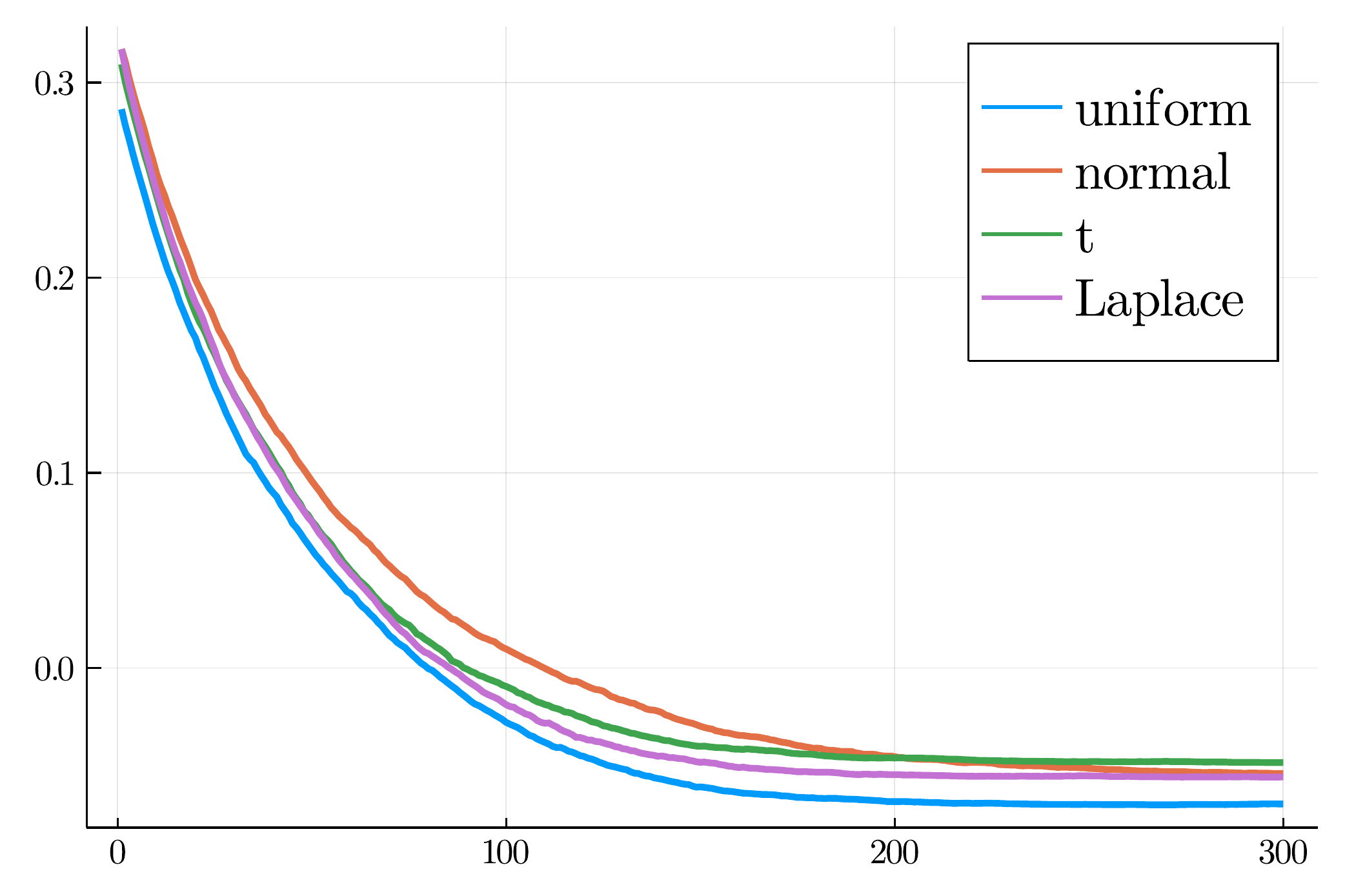}};
\node[below=of img3] (img5) {\includegraphics[width=0.4\textwidth]{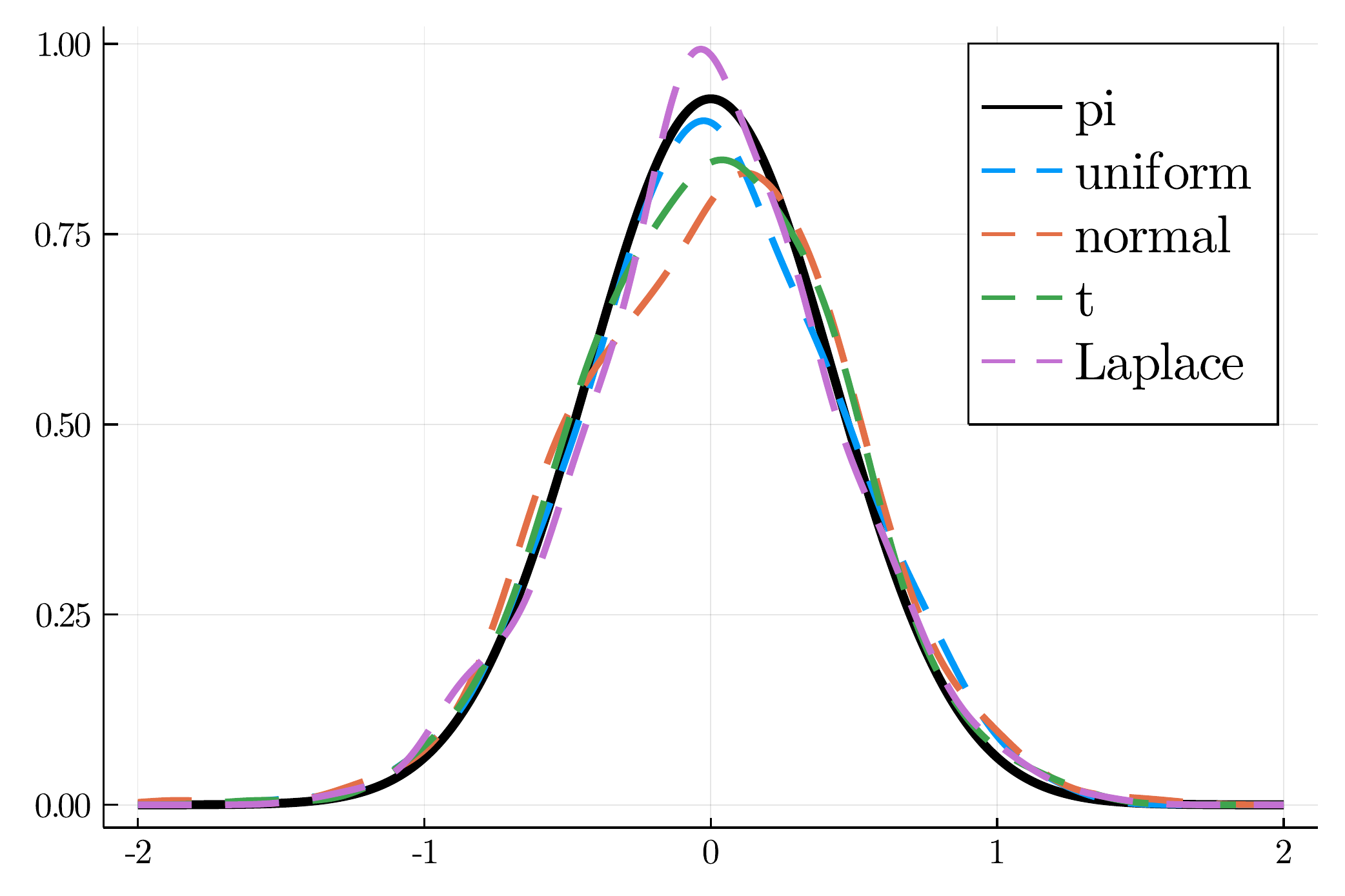}};
\node[left=of img5, node distance = 0, rotate=90, anchor = center, yshift = -1cm] {$X_0^{1:N}\sim \mathcal{N}(0, 10^{-8})$};
\node[right=of img5, node distance = 0, xshift = -0.5cm] (img6) {\includegraphics[width=0.4\textwidth]{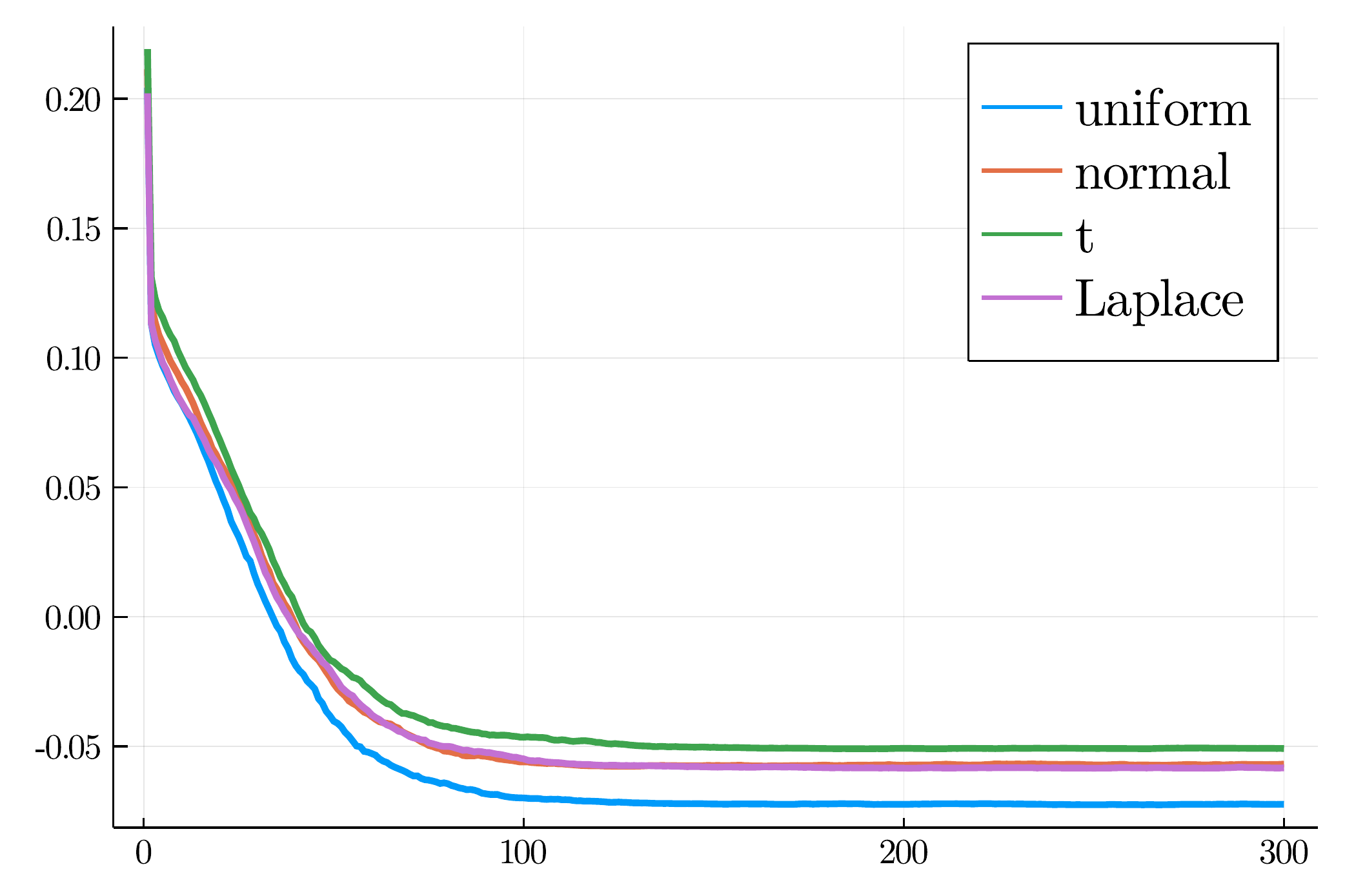}};
  \end{tikzpicture}
  }
\caption{Influence of reference measure and initial condition, reconstructions (left panel) and approximate value of $\Gun_{\alpha}^{\eta}$ (right panel).}
\label{fig:pi0}
\end{figure}
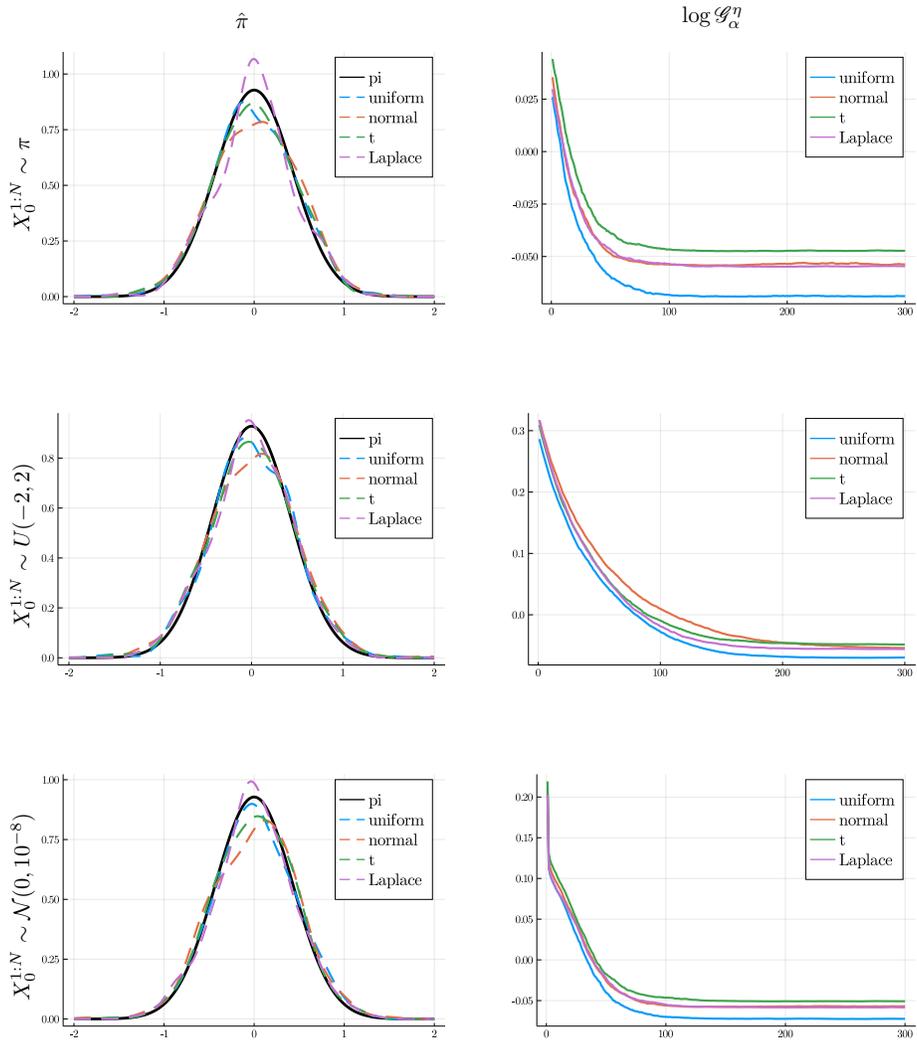

Figure~\ref{fig:pi0} shows that the reference measure does not influence the
speed of convergence but influences the shape of the reconstructions, which are
more peaked when the reference measure is Laplace and flatter when the reference
measure is a Gaussian distribution; this corresponds to slightly different
minima for $\Gun_{\alpha}^{\eta}$.  In particular, the order of the minima is
not influenced by the initial condition, with the uniform reference measure always
giving the smallest minima.  Indeed, if we compare the values of
$\Gun_{\alpha}^{\eta}$ at the solution $\pi$ we find that when the reference
measure is Uniform, $\KL{\pi}{\pi_0}$ is not defined and
\begin{align}
\Gun_{\alpha}^{\eta}(\textrm{t})\geq \Gun_{\alpha}^{\eta}(\textrm{Gaussian})\simeq \Gun_{\alpha}^{\eta}(\textrm{Laplace}). 
\end{align}

The initial condition $X_0^{1:N}$ affects the speed of convergence with light
tail distributions giving faster convergence as already observed in
\cite{bossy1997stochastic,antonelli2002rate}, but does not affect the shape of
the reconstructions. This is not surprising as \Cref{prop:propagation_chaos}
guarantees the existence of a unique solution for any initial condition
$X_0^{1:N}$.

\subsection{Choice of $\alpha$}
\label{app:alpha}
\subsubsection{Toy problem}

Consider again the toy Fredholm integral equation in \Cref{app:pi0} and consider
$\Gun_\alpha$ with the reference measure $\pi_0$ given by a Gaussian
distribution with mean 0 and variance $\sigma_0^2$. If we restrict the class of
distributions on which we look for a minimizer to that of Gaussian distributions
with mean 0 and variance $\beta$, \ie\
$\pi(x)=\mathcal{N}(x; 0, \beta)$, the functional $\Gun_\alpha$
can be computed exactly
\begin{align}
\Gun_\alpha(\pi) &\textstyle{= -\int_{\rset^p} \log(\pi[\kker(\cdot, y)]) \rmd \mu(y) + \alpha \KL{\pi}{\pi_0}}\\
&\textstyle{ =(1/2)\log (2\pi(\beta+\sigma_\kker^2))+\sigma_\mu^2/(2(\beta+\sigma_\kker^2))+\alpha/2\left( \log \sigma_0^2/\beta+\beta/\sigma_0^2-1\right).}
\end{align}
Finding the optimal $\beta$ for given $\alpha$ amounts to solving the following cubic equation
\begin{align}
\alpha\beta^3 +\beta^2(2\alpha\sigma_{\kker}^2+(1-\alpha)\sigma_0^2)+\beta(\alpha\sigma_\kker^4-\sigma_\mu^2\sigma_0^2+\sigma_0^2\sigma_\kker^2-2\alpha\sigma_\kker^2)-\alpha\sigma_0^2\sigma_\kker^4=0,
\end{align}
clearly when $\alpha=0$ (no cross-entropy constraint), $\beta=\sigma_{\pi}^{2}$, while for all $\alpha>0$ (at least) one optimal $\beta$ exists. However, as $\alpha$ increases the value of $\beta$ stops increasing and stabilizes around a fixed value which depends on $\sigma_0^2$, Figure~\ref{fig:alpha_toy} shows the result for $\alpha\in[0,1]$ and $\sigma_0^2 = 0.1^2$.

\begin{figure}
\centering
\resizebox{0.9\textwidth}{!}{%
\begin{tikzpicture}[every node/.append style={font=\normalsize}]
\node (img1) {\includegraphics[width=0.40\textwidth]{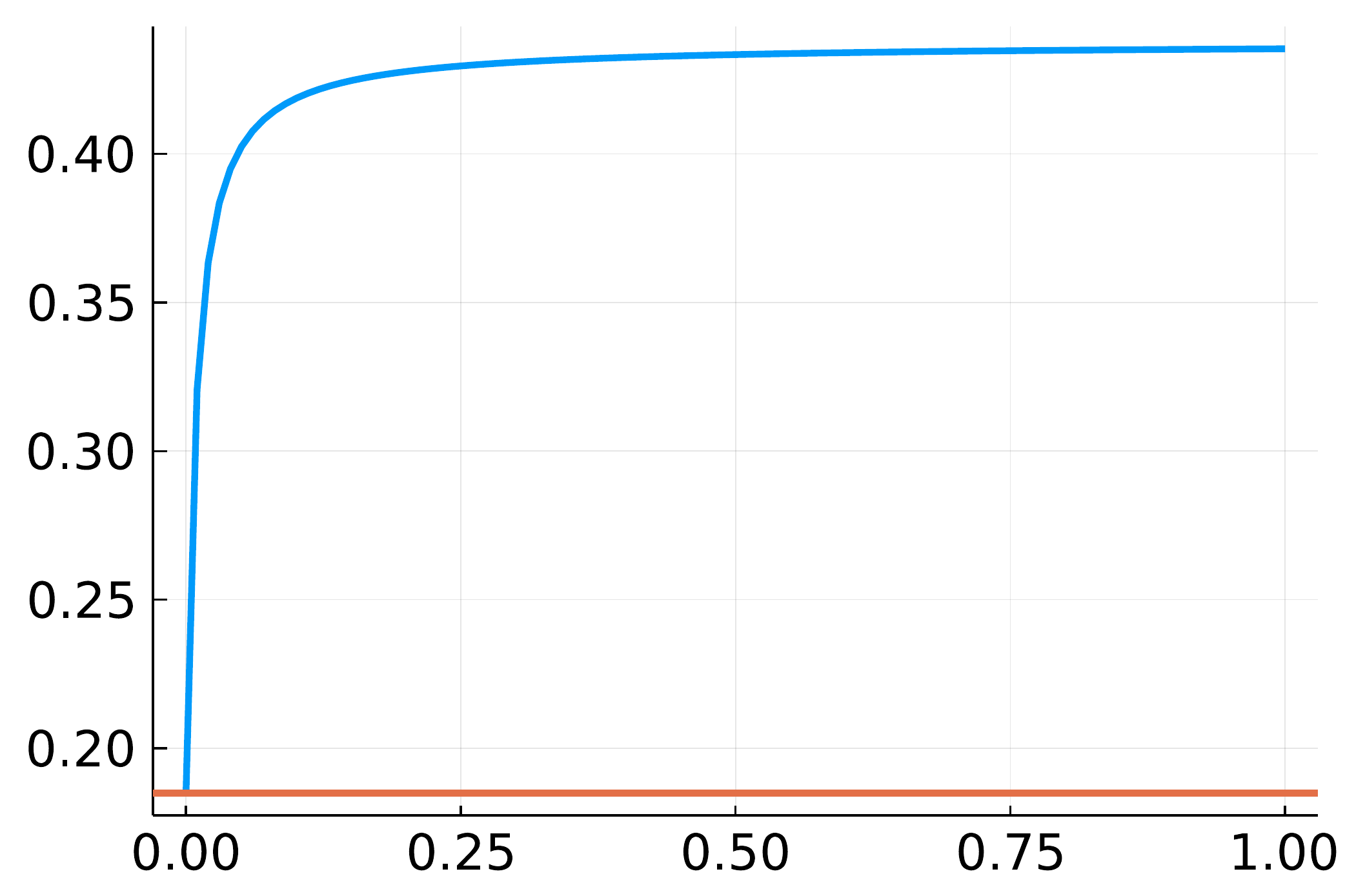}};
\node[below=of img1, node distance = 0, yshift = 1cm] { $\alpha$};
  \node[left=of img1, node distance = 0, rotate=90, anchor = center, yshift = -0.7cm] {$\beta(\alpha)$};
\node[right= of img1, yshift = 0.1cm] (img3){\includegraphics[width=0.40\textwidth]{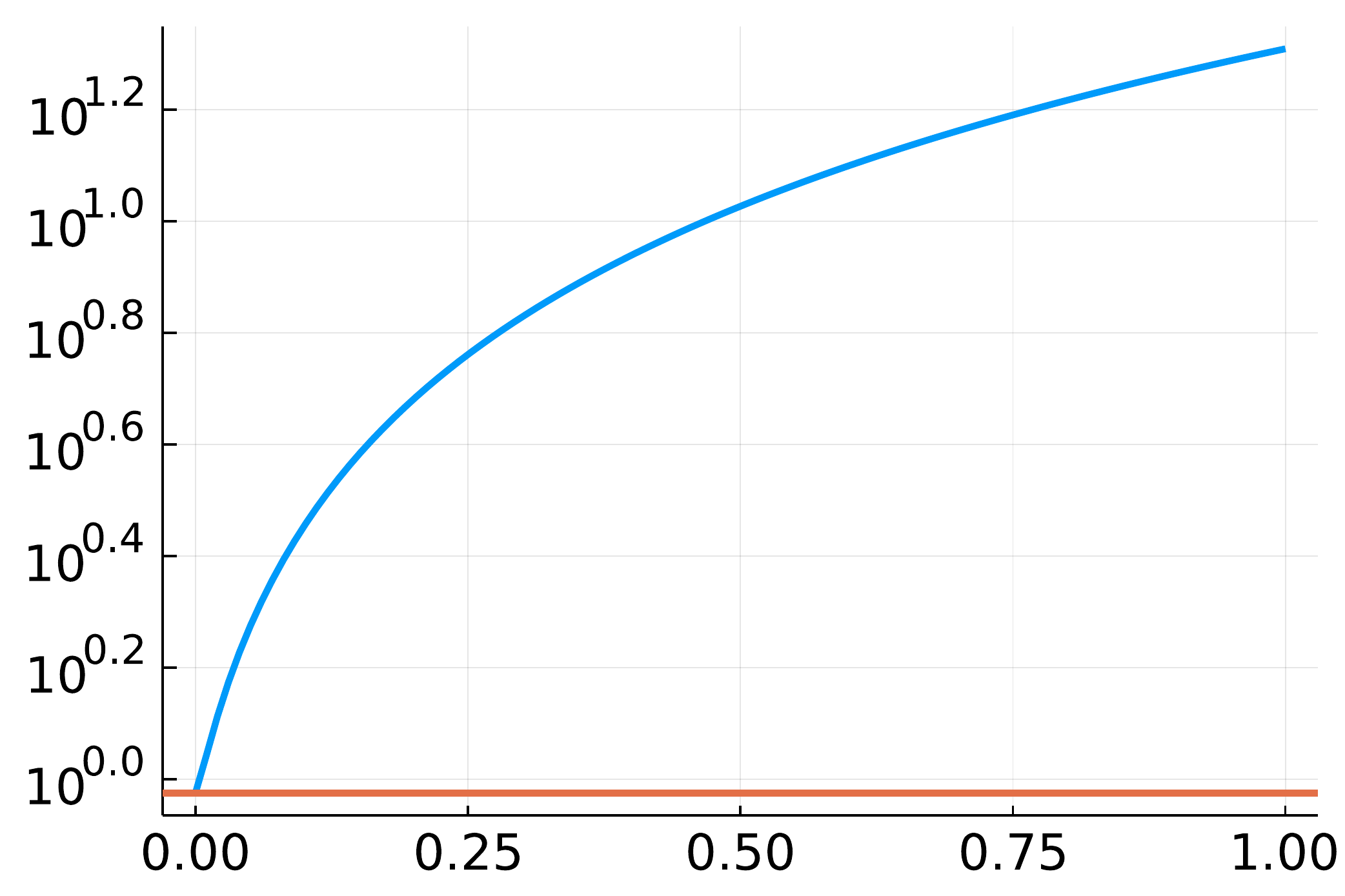}};
\node[below=of img3, node distance = 0, yshift = 1cm] {$\alpha$};
  \node[left=of img3, node distance = 0, rotate=90, anchor = center, yshift = -0.7cm] {$\Gun_\alpha$};
\end{tikzpicture}
}
\caption{Example of dependence of the variance of the minimizer $\pi^\star_{\alpha}$ (left) and the functional $\Gun_{\alpha}$ (right) on the parameter $\alpha$. The red lines denote the values obtained with $\alpha=0$, \ie\ no regularization.}
\label{fig:alpha_toy}
\end{figure}
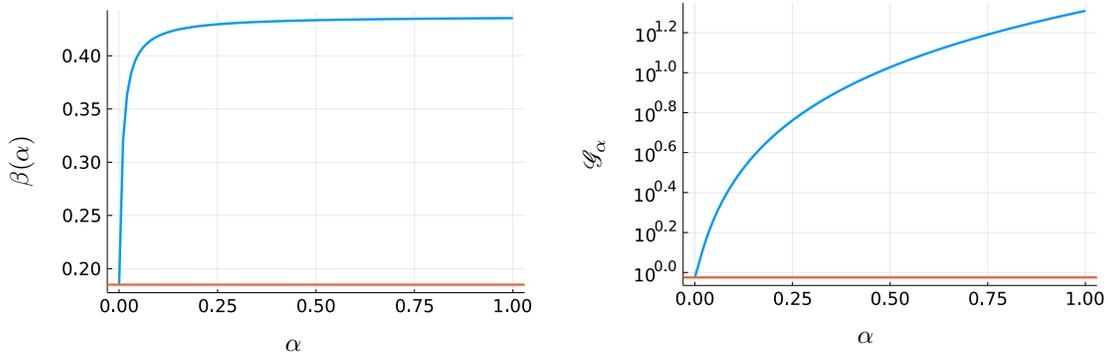

As discussed in \Cref{sec:variants}, taking $\alpha=1$ amounts to Bayesian inference for a model with prior $\pi_0$ and data given by the available sample from $\mu$. Using the cubic equation above we obtain that the corresponding optimal $\beta$ is $0.44$; comparing the minimizer obtained with $\alpha=1$ with the solution $\pi$ and the minimizer corresponding to the value of $\alpha=4\cdot 10^{-5}$ obtained by cross-validation we find that $\alpha=1$ results in significantly over-smoothed reconstructions (Figure~\ref{fig:alpha_toy_prior}) while the value chosen by cross-validation provides reconstructions which are indistinguishable from the solution $\pi$.

\begin{figure}
\centering
\resizebox{0.6\textwidth}{!}{%
\begin{tikzpicture}[every node/.append style={font=\normalsize}]\node (img1) {\includegraphics[width=0.40\textwidth]{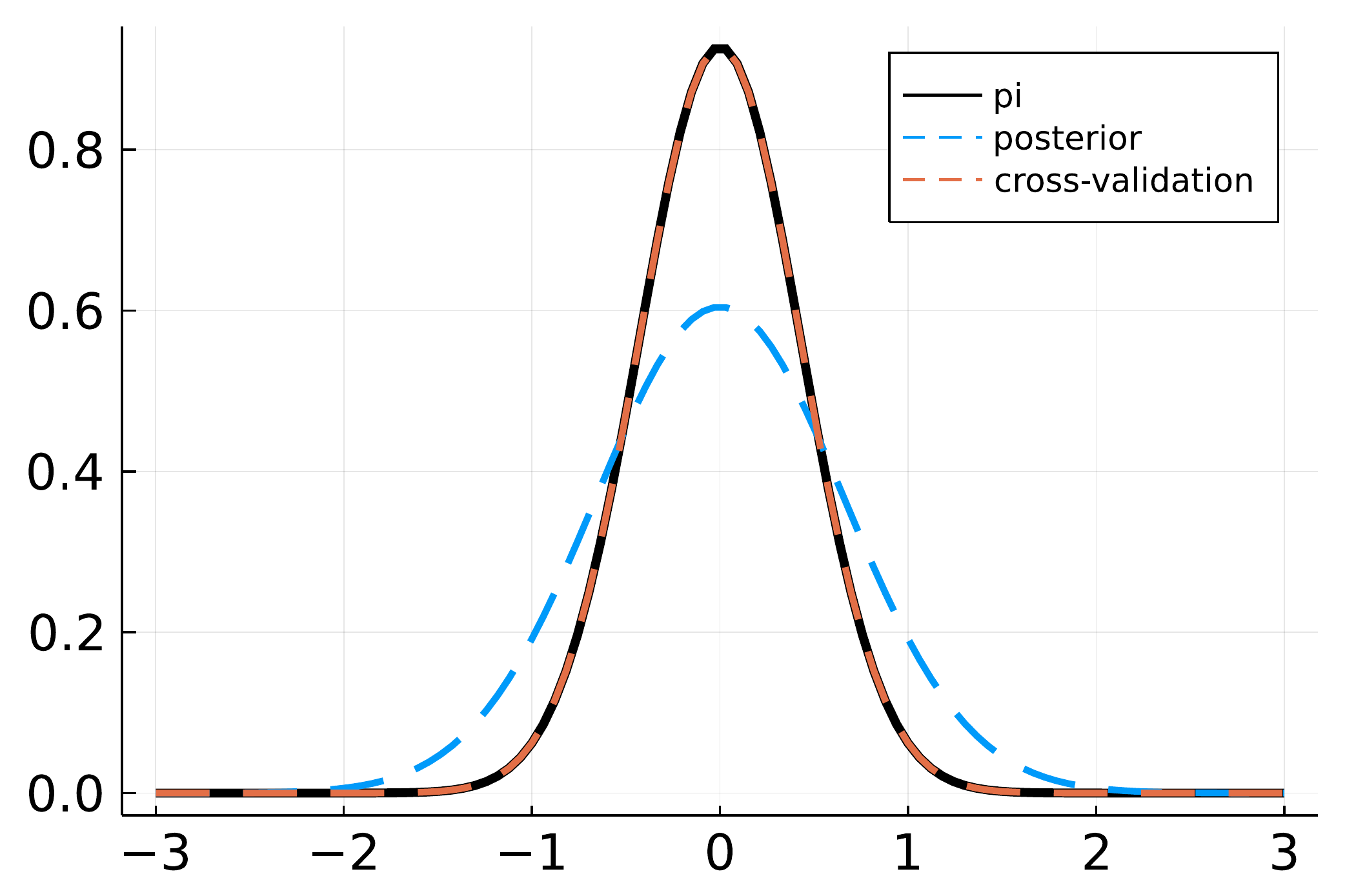}};
\end{tikzpicture}
}
\caption{Effect of regularizing parameter $\alpha$ on reconstructions. We compare $\alpha=1$ which corresponds to Bayesian inference for the incomplete data model with prior $\pi_0$ and $\alpha=$ selected by cross-validation.}
\label{fig:alpha_toy_prior}
\end{figure}


\end{document}